\documentclass[11pt]{article}

\usepackage{graphicx}
\usepackage{amsmath,amssymb,amsfonts,amsthm}
\usepackage{mathrsfs}
\usepackage{xcolor}
\usepackage{mathtools}
\usepackage{booktabs}
\usepackage{algorithm}
\usepackage{algorithmicx}
\usepackage{algpseudocode}
\usepackage{hyperref}
\hypersetup{hidelinks}
\usepackage{geometry}
\usepackage{multirow}

\usepackage{float}

\newcommand{\la}{\langle}
\newcommand{\ra}{\rangle}

\newcommand{\Rbb}{\mathbb{R}}

\newcommand{\aff}{{\rm aff}}

\newcommand{\dww}{\Delta \ww}
\newcommand{\dyy}{\Delta y}

\newcommand{\mE}{\mathcal{E}}

\newcommand{\mM}{\mathcal{M}}

\newcommand{\mP}{\mathcal{P}}
\newcommand{\mQ}{\mathcal{Q}}
\newcommand{\mT}{\mathcal{T}}

\newcommand{\dist}{{\rm dist}}

\newcommand{\one}{\mathbf{1}}

\newcommand{\tIm}{{\rm Im}}
\newcommand{\tNu}{{\rm Null}}

\newcommand{\supp}{{\rm supp}}
\DeclareMathOperator*{\argmin}{arg\,min}

\newcommand{\nuk}{k}
\newcommand{\ww}{w}

\usepackage{tikz}
\usetikzlibrary{arrows.meta}
\usetikzlibrary{shapes.geometric, arrows}
\tikzset{global scale/.style={
    scale=#1,
    every node/.append style={scale=#1}
  }
}
\tikzstyle{blank} = [diamond, aspect = 3, text centered]
\tikzstyle{global} = [rectangle, rounded corners, minimum width = 2cm, minimum height=1cm, text centered, draw = black, fill = blue!40]
\tikzstyle{point} = [rectangle, rounded corners, minimum width = 2cm, minimum height=1cm, text centered, draw = black, fill = green!40]
\tikzstyle{ypoint} = [rectangle, rounded corners, minimum width = 2cm, minimum height=1cm, text centered, draw = black, fill = yellow!20]
\tikzstyle{arrow} = [thick, ->, >=latex, draw = black]

\tikzstyle{arrowback} = [thick, ->, >=latex, draw = blue]
\tikzstyle{equal} = [thick, <->, >=latex, draw = black]

\theoremstyle{plain}
\newtheorem{theorem}{Theorem}
\newtheorem{lemma}{Lemma}
\newtheorem{proposition}{Proposition}
\theoremstyle{definition}
\newtheorem{definition}{Definition}
\newtheorem{example}{Example}
\theoremstyle{remark}
\newtheorem{remark}{Remark}

\begin{document}

\title{Semismooth Newton methods for degenerate polyhedral projection}
\author{
Chao Ding\thanks{State Key Laboratory of Mathematical Sciences, Academy of Mathematics and Systems Science, Chinese Academy of Sciences, Beijing 100190, P.R. China; School of Mathematical Sciences, University of Chinese Academy of Sciences, Beijing 100049, P.R. China; Institute of Applied Mathematics, Academy of Mathematics and Systems Science, Chinese Academy of Sciences, Beijing 100190, P.R. China. Email: \texttt{dingchao@amss.ac.cn}.}
\and
Fuxiaoyue Feng\thanks{Institute of Applied Mathematics, Academy of Mathematics and Systems Science, Chinese Academy of Sciences, Beijing, P.R. China; School of Mathematical Sciences, University of Chinese Academy of Sciences, Beijing 100049, P.R. China. Email: \texttt{fengfuxiaoyue@amss.ac.cn}.}
\and
Xudong Li\thanks{School of Data Science, Fudan University, Shanghai 200433, P.R. China. Email: \texttt{lixudong@fudan.edu.cn}.}
}

\date{This version: \today}

\maketitle

	\begin{abstract}
		
		In this paper, we study dual semismooth Newton (SSN) methods for degenerate
		polyhedral projection problems, where generalized Jacobians of the dual
		residual may remain singular even arbitrarily close to the solution set.
		Rather than regularizing these singular systems, we exploit the nonuniqueness
		of the dual representation. We introduce a primal--dual lifted
		projection-equivalent set that always possesses extreme points without
		additional structural assumptions on the polyhedron, and show that its
		extreme-point geometry identifies dual representatives at which nonsingular
		generalized Jacobians of the dual residual can be constructed. This geometry is further
		linked to a full-column-rank condition and a generalized weak strict Robinson constraint
		qualification, showing that the regularity required by the Newton step can be
		recovered rather than imposed \emph{a priori}. We also establish displacement
		bounds that connect representative selection throughout the algorithm with
		the local Newton mechanism. Building on this variational framework, we
		develop an inexact dual SSN method with local superlinear convergence and a
		globalized version combining monotone representative selection with a
		Wolfe line search. The resulting method is globally convergent and eventually
		recovers the fast local rate. Numerical experiments on regularized optimal
		transport, battery-scheduling feasibility restoration, and occupation-measure
		projection demonstrate its robustness in highly degenerate settings.

		\medskip

	\end{abstract}

\section{Introduction}\label{section:intro}

Given $c\in \Rbb^n$, $b\in\Rbb^{m_E}$, and $A\in\Rbb^{m_E\times n}$, we consider the polyhedral projection problem
\begin{equation}\label{prob:K-lin-proj}
	\begin{array}{cl}
		\min & \displaystyle \frac{1}{2} \| x-c\|^2 \\ [5pt]
		\mbox{s.t.} & A x = b, \\ [2pt]
		& x\in K,
	\end{array}
\end{equation}
where $K = \{x\in \Rbb^n\mid Gx \ge g\}$ is a nonempty polyhedron with $G\in \Rbb^{m_I \times n}$ and $g\in \Rbb^{m_I}$. We assume that the feasible
set of~\eqref{prob:K-lin-proj} is nonempty. As a strongly convex quadratic program, problem~\eqref{prob:K-lin-proj} admits a unique optimal solution, denoted by $x^*$.  Redundant equality constraints can be
removed, so we take $A$ to have full row rank. We also assume that $K$ is simple in the sense that the metric projection onto $K$, denoted by $\Pi_K$, and an associated multiplier can be computed efficiently. Typical examples of $K$ include the nonnegative orthant, box constraints, and the simplex. 

From an algorithmic perspective, problem~\eqref{prob:K-lin-proj} serves as a basic computational primitive in a variety of optimization algorithms, including augmented Lagrangian methods, sequential quadratic programming, projected gradient methods, and splitting algorithms, in which strongly convex quadratic subproblems with linear side constraints frequently arise~\cite{NocedalWright2006,Bertsekas2016Nonlinear,li2018efficiently,LiSunToh2020BirkhoffJacobian}.
Beyond its algorithmic role, problem~\eqref{prob:K-lin-proj} also appears in a broad class of applications. Complicated coupling relations, such as conservation, balance, assignment, or other side constraints, are represented by the linear system \(Ax=b\), while simple feasibility requirements, including nonnegativity, box-type restrictions, generalized box constraints, simplex constraints, and Cartesian products thereof, are absorbed into the set \(K\). We next describe three representative examples.

The first example is quadratically regularized optimal transport:
\begin{equation}\label{prob:app-ot}
\min_{X\in\Rbb^{M \times N}}
\left\{
\langle C,X\rangle + \frac{\rho}{2}\|X\|_F^2
\;\middle|\;
X\one_N = \omega,\;\;
X^\top \one_M = \nu,\;\;
X\ge0
\right\},
\end{equation}
where $C\in\mathbb{R}^{M\times N}$ is a cost matrix, $\omega\in\mathbb{R}^M_+$ and $\nu\in\mathbb{R}^N_+$ are prescribed vectors satisfying $\mathbf{1}_M^\top \omega = \mathbf{1}_N^\top \nu$, and $\rho>0$ is a regularization parameter. After vectorization, this problem fits exactly into~\eqref{prob:K-lin-proj} with $K=\Rbb^{MN}_+$. 
Optimal transport and its regularized variants have become fundamental tools in machine learning, computer vision, and data science~\cite{Cuturi2013,PeyreCuturi2019,CourtyEtAl2017}. In particular, quadratic regularization is well suited for obtaining transport plans that are both computationally tractable and structurally informative, especially in settings where sparsity of the transport plan is desirable~\cite{BlondelEtAl2018,LorenzMannsMeyer2021,nutz2025quadratically}. Model~\eqref{prob:app-ot} also covers the nearest doubly stochastic matrix problem~\cite{FogelEtAl2013,LiSunToh2020BirkhoffJacobian}, where the set of permutation matrices is relaxed to the Birkhoff polytope. A second example comes from separable resource allocation and scheduling with
cumulative or nested constraints. A typical model takes the form
\begin{equation}\label{eq:gen-box}
\min_{x\in\Rbb^n}
\left\{
\frac12\|x-c\|^2
\;\middle|\;
A x = b,\;
l \le Bx \le u
\right\},
\end{equation}
where $-\infty \le l < u \le +\infty$.
Here, $Ax=b$ imposes global balance requirements, while $B$, assumed to have full row rank, encodes cumulative or nested linear aggregates. Such structures arise, for example, in production-sales planning with time-dependent inventory bounds and related models with investment ranges~\cite{Tamir1980,ShiouraShakhlevichStrusevich2016,VidalGribelJaillet2019}, as well as in battery scheduling for decentralized energy management in smart grids~\cite{SchootUiterkampHurinkGerards2021}. A third important family of applications arises when $K$ is chosen as the simplex
\[
\Delta_n:=\{x\in\mathbb{R}^n \mid \mathbf{1}_n^\top x=1, \ x\ge 0 \},
\]
or, more generally, as a Cartesian product of simplices. In this case, problem~\eqref{prob:K-lin-proj} covers Euclidean projection and quadratically regularized subproblems for estimating probability or abundance vectors. 
Such models appear in topic models~\cite{BleiNgJordan2003}, hyperspectral
unmixing~\cite{KeshavaMustard2002,BioucasDiasEtAl2012}, and convex multi-class
labeling~\cite{LellmannEtAl2009}. 

Problem~\eqref{prob:K-lin-proj} can be addressed by a broad range of numerical methods. These include first-order approaches, such as dual gradient-type methods~\cite{HagerZhang2016Projection} and splitting methods~\cite{gabay1983chapter,glowinski1975approximation}, as well as second-order approaches, such as interior-point methods~\cite{wright1997primal,nesterov1994interior}, active-set methods~\cite{NocedalWright2006,goldfarb1983numerically}, dual quasi-Newton methods~\cite{Malick2004DualSDLS}, and dual semismooth Newton (SSN) methods~\cite{BaiChuSun2007DualIQEP,LiSunToh2020BirkhoffJacobian,WangShenZhangYang2023DGASS}. Numerical comparisons reported in~\cite{HagerZhang2016Projection,LiSunToh2020BirkhoffJacobian,WangShenZhangYang2023DGASS} indicate that second-order methods can generally achieve faster convergence and higher accuracy than first-order alternatives within comparable computational budgets. Among these approaches, dual SSN methods are particularly attractive due to their excellent practical performance and local superlinear or quadratic convergence guarantees. The fast local convergence theory of second-order methods, however, typically relies on method-dependent regularity assumptions. Depending on the particular algorithm and formulation, these assumptions may take the form of the linear independence constraint qualification (LICQ), strict complementarity, strong regularity, or the nonsingularity of relevant KKT matrices or generalized Jacobians~\cite{qi1993nonsmooth,Qi1993convergence,wright1997primal,Dontchev1998,BonnansShapiro2000,NocedalWright2006,LiSunToh2020BirkhoffJacobian,WangShenZhangYang2023DGASS}. When such regularity conditions fail, Lagrange multipliers may become nonunique, and the linear systems arising in Newton-type methods may become severely ill-conditioned or singular, leading to significant degradation of numerical performance. Such degeneracy is difficult to avoid in the applications described above. The coupling constraints \(Ax=b\) may interact with many active inequalities defining \(K\), especially when the solution lies on a low-dimensional face of the feasible set. For example, in quadratically regularized optimal transport with a small regularization parameter, the solution may be close to a sparse optimal solution of the underlying transportation linear program; in generalized-box models, multiple nested or cumulative bounds may become simultaneously active; and in simplex or product-simplex models, many components may vanish. In these situations, the active inequality gradients together with the equality-constraint gradients may become linearly dependent, causing LICQ to fail.

Motivated by these challenges, we focus on dual SSN methods in degenerate polyhedral projection problems. Specifically, we consider the following dual formulation of problem~\eqref{prob:K-lin-proj}:
\begin{equation}\label{prob:dual-K-lin-proj}
	\min_{y\in\Rbb^{m_E}}
	\varphi(y)
	:=
	-\langle b,y\rangle
	+
	\inf_{z\in\Rbb^n}
	\left\{
	\delta_K^*(z)
	+
	\frac{1}{2}\|z-(A^\top y+c)\|^2
	\right\},
\end{equation}
where \(\delta_K\) denotes the indicator function of \(K\), and
\(
\delta_K^*(z):=\sup_{x\in K}\langle z,x\rangle
\)
denotes its Fenchel conjugate.
Standard convex analysis \cite[Theorem 31.5]{rockafellar1997convex} shows that \(\varphi\) is convex with a globally Lipschitz continuous gradient:
\begin{equation}\label{eq:grad-dual}
\nabla\varphi(y)
=
A\Pi_K(A^\top y+c)-b.
\end{equation}
Since \(K\) is polyhedral, its metric projector \(\Pi_K\) is piecewise affine. Consequently, \(\nabla\varphi\) is also piecewise affine and hence strongly semismooth~\cite{Qi1993convergence,qi1993nonsmooth}. Therefore, dual SSN methods provide a natural high-order approach for solving the dual optimality equation
\[
\nabla\varphi(y)=0.
\]
The remaining difficulty is that, in degenerate regimes, the generalized Jacobians involved in the SSN steps may become singular even arbitrarily close to the solution set, so that the classical local superlinear or quadratic convergence theory is no longer directly applicable. Common numerical remedies, such as regularization or random perturbation, may improve the solvability of individual Newton systems, but they do not provide a structural resolution of this degeneracy.

The central observation of this paper is that the singularity of the generalized Jacobians used in a dual SSN method need not be an intrinsic defect of the primal projection. Indeed, the same projection
\begin{equation}\label{eq:primal-proj}
    \Pi_K(A^\top y+c)
\end{equation}
may admit multiple dual representatives \(y\). Any two such representatives generate the same residual in~\eqref{eq:grad-dual}, while the associated Han--Sun generalized Jacobians \cite{han1997newton,LiSunToh2020BirkhoffJacobian} may have different nonsingularity properties. This observation suggests a new strategy, i.e., instead of modifying the Newton system through regularization or perturbation, one may replace a degenerate dual representative by a suitable nondegenerate one without changing the primal projection or the dual residual. To turn this observation into a rigorous algorithmic framework, three fundamental questions must be addressed. First, for every given projected point in the form of \eqref{eq:primal-proj}, does there exist an associated dual representative with a nonsingular generalized Jacobian? Second, can such a representative be identified while quantitatively controlling the possible resulting displacement, so that the local Newton mechanism is preserved?
Third, can such a selection be integrated into a globally convergent framework while retaining the fast local convergence behavior of semismooth Newton methods?

We answer the first question by introducing the primal--dual lifted projection-equivalent set $\mathcal Q(y,w)$, whose elements are pairs
of dual representatives $y$ and associated projection multipliers $w$. We show that this lifted set always has extreme points, and at the $y$-component of every extreme pair,
a nonsingular generalized Jacobian of $\nabla\varphi$ can be constructed.
We further connect this
extreme-point geometry with a full-column-rank condition and a generalized weak strict
Robinson constraint qualification (W-SRCQ) introduced
in~\cite{feng2024quadratically}. Thus, the nonsingularity needed by the Newton step can be recovered
from the geometry of the lifted equivalence class rather than imposed as an \emph{a priori} assumption. To answer the second question, we establish quantitative displacement bounds
for the selected extreme-point representatives. Importantly, the selection
can be performed at any iterate and does not require proximity to the
solution set, allowing the dual representative to be modified throughout
the global phase. Once the iterates approach the solution set, the
displacement bounds ensure that the selected representatives remain
quantitatively close to the corresponding extreme-point representatives of
the solution class. This provides the key link between the global
representative-selection mechanism and the local Newton analysis, ensuring
that the modification does not destroy the fast local convergence rate. Building on these results, we develop a dual semismooth Newton framework that
answers the third question. At each iteration, the framework identifies an
extreme-point representative through a monotone procedure, constructs a
Newton step using the resulting nonsingular  generalized Jacobian,
and employs a Wolfe line search for globalization. The resulting method is
globally convergent and eventually recovers the fast local superlinear
convergence rate. Numerical experiments on quadratically regularized optimal transport,
battery-scheduling feasibility restoration, and occupation-measure projection
demonstrate the robustness of the proposed approach in highly degenerate
settings.

A related projection-equivalence and vertex-selection viewpoint appears in
Hu et al.~\cite{hu2024semismooth}, but their analysis is tailored to the
highly structured nearest doubly stochastic matrix problem and exploits its
special orthant and bipartite-graph structure. In contrast, our framework
applies to the general polyhedral projection problem
\eqref{prob:K-lin-proj}. The two frameworks differ at three essential
levels. Geometrically, the natural purely dual projection-equivalent set for
a general polyhedron may contain lines and need not have extreme points,
which necessitates the primal--dual lift \(\mathcal Q(y,w)\).
Theoretically, our lifted framework provides a general characterization
linking extreme-point geometry, a full-column-rank condition, nonsingular
Han--Sun generalized Jacobians, and the W-SRCQ. It also
yields displacement bounds without requiring the current iterate to be
close to the optimal solution set, based on general polyhedral error-bound and
support-reduction arguments rather than graph-specific constructions.
Algorithmically, whereas the convergence theory
in~\cite{hu2024semismooth} is local and cycling outside the local regime is
handled empirically by randomization, our monotone representative-selection procedure,
combined with a Wolfe line search, yields a provably globally convergent
method while preserving the fast local rate.

The remainder of the paper is organized as follows. Section~\ref{sec:prelim} reviews the Han--Sun generalized Jacobian for polyhedral projections and the basic properties of the dual function. Section~\ref{sec:theory} develops the geometry of projection-equivalent sets, establishes the nonsingularity characterization, and derives the associated displacement bounds. Section~\ref{sec:alg} presents the local and globalized semismooth Newton methods and analyzes their convergence properties. Section~\ref{sec:Numerical} reports the numerical results, and Section~\ref{sec:conclusion} concludes the paper.

\paragraph{Notation.}
For a positive integer \(q\), let \([q]:=\{1,\ldots,q\}\).
For an index set \(J\subseteq[m_I]\), we define \(J^C:=[m_I]\setminus J\). For a matrix \(G\) and a vector \(v\),
\(G_J\) denotes the submatrix of \(G\) formed by the rows indexed by
\(J\), and \(v_J\) denotes the subvector of a vector \(v\) indexed by
\(J\). We use
$
\supp(v):=\{i\mid v_i\neq0\}.
$
For a nonempty closed set \(C\), define
\[
\dist(x,C):=\inf_{z\in C}\|x-z\|,
\qquad
\Pi_C(x):=\argmin_{z\in C}\|x-z\|.
\]
If \(C\) is in addition convex, then $\Pi_C(\cdot)$ is single-valued and Lipschitz continuous; otherwise, it may be set-valued.
We write
\(
\mathbb B(x,r):=\{z  \mid \|z-x\|<r\}
\)
for the open Euclidean ball centered at \(x \) with radius \(r\).
The symbols \(\mathcal E(C)\), \(\aff(C)\), and
\(\mathcal T_C(x)\) denote, respectively, the set of extreme points of
\(C\), the affine hull of \(C\), and the Bouligand tangent cone to \(C\)
at \(x\). Moreover, for a vector \(w\), we define its orthogonal complement by
\(
w^\perp:=\{d\mid \langle d,w\rangle=0\},
\)
and use \(\operatorname{Im}(L)\) and \(\operatorname{Null}(L)\) to denote  the
image and null space of a linear mapping \(L\), respectively. For a matrix \(L\), we use \(\sigma_{\min}(L)\) to denote its
smallest singular value and, when \(L\) is symmetric,
\(\lambda_{\min}(L)\) and \(\lambda_{\max}(L)\) to denote its
smallest and largest eigenvalues, respectively.
The space of real \(q\times q\) symmetric matrices is denoted by
\(\mathbb S^q\). The vector of all ones in \(\mathbb R^q\) is denoted
by \(\mathbf 1_q\), whereas \(e_i\) denotes the \(i\)-th standard basis
vector.
Unless otherwise specified,
\(\|\cdot\|\) denotes the Euclidean norm for vectors and the induced
spectral norm for matrices.

\section{Preliminaries}\label{sec:prelim}

In this section, we collect several preliminary results on the projector $\Pi_K$ and the dual function $\varphi$ defined in \eqref{prob:dual-K-lin-proj}. We first review the Han--Sun generalized Jacobian of $\Pi_K$ and the associated multiplier structure for polyhedral projection. We then record several basic variational properties of the dual function, which will be used later in the local and global convergence analysis.

Given $s\in\Rbb^n$, recall that $K=\{x\in\Rbb^n\mid Gx\ge g\}$. Since $\Pi_K(s)$ is the Euclidean projection of $s$ onto the polyhedron $K$, we know that there exists $\ww\in\Rbb^{m_I}$ satisfying the following Karush–Kuhn–Tucker (KKT) conditions:
\begin{equation}\label{eq:kkt-proj-K}
\left\{
\begin{aligned}
& \Pi_K(s)-s+G^\top \ww=0,\\
& G\Pi_K(s)\ge g,\quad \ww\le 0,\\
& \ww^\top(G\Pi_K(s)-g)=0.
\end{aligned}
\right.
\end{equation}
Denote the set of multipliers associated with $s$ by
\begin{equation}\label{eq:M}
\mM(s):=\left\{\ww\in\Rbb^{m_I}\mid (s,\ww)\ \mbox{satisfies \eqref{eq:kkt-proj-K}}\right\}.
\end{equation}
It follows directly from \eqref{eq:kkt-proj-K} that $\mM(s)$ is a nonempty polyhedron. Moreover, since every $\ww\in\mM(s)$ satisfies $\ww\le 0$, the set $\mM(s)$ contains no lines. Consequently, its extreme point set ${\cal E}(\mM(s))$ is nonempty \cite[Corollary 18.5.3]{rockafellar1997convex}. To this end, define
\[
I(s):=\{\,i\in[m_I]\mid G_i\Pi_K(s)=g_i\,\},
\]
the active index set at $\Pi_K(s)$, and the following family of index sets:
\begin{equation}\label{eq:D-collection}
{\cal D}(s)
:=
\left\{
J\subseteq [m_I]\ \middle|\
\exists\,\ww\in\mM(s)\ \mbox{such that } {\rm supp}(\ww)\subseteq J\subseteq I(s),\ \ G_J\ \mbox{has full row rank}
\right\}.
\end{equation}
As noted in \cite{han1997newton}, the nonemptiness of ${\cal E}({\cal M}(s))$ implies the nonemptiness of ${\cal D}(s)$. Indeed, for any $\ww\in{\cal E}(\mM(s))$, from the definition of extreme points, one has ${\rm supp}(\ww)\in{\cal D}(s)$.

Although the Bouligand Jacobian $\partial_B \Pi_K$ and the Clarke Jacobian 
$\partial_C \Pi_K$ provide useful generalized differential information for 
the projection mapping $\Pi_K$, their direct characterization and computation 
can be difficult in practice for a general polyhedron $K$. Following \cite{han1997newton}, we therefore use 
the Han--Sun generalized Jacobian as a computable substitute. Specifically, the 
Han--Sun generalized Jacobian of $\Pi_K$ at $s$ is defined by
\begin{equation}\label{eq:HS-Jac-s}
\partial_{\rm HS}\Pi_K(s)
:=
\left\{
U\in\Rbb^{n\times n}\ \middle|\
U=I_n-G_J^\top(G_JG_J^\top)^{-1}G_J,\ \ J\in{\cal D}(s)
\right\}
\end{equation}
with the convention
$
G_J^\top(G_JG_J^\top)^{-1}G_J:=0
$ if $J=\emptyset$.
Note that for each $J\in{\cal D}(s)$, the matrix
\[
U=I_n-G_J^\top(G_JG_J^\top)^{-1}G_J
\]
is the orthogonal projector onto ${\rm Ker}(G_J)$, and hence is symmetric 
positive semidefinite. The following lemma, taken from 
\cite[Lemma 2.1]{han1997newton}, justifies the use of 
$\partial_{\rm HS}\Pi_K(s)$ in place of the Bouligand and Clarke generalized 
Jacobians.

\begin{lemma}\label{lemma:semismooth_HS}
For any given $s\in\Rbb^n$, there exists a neighborhood $V$ of $s$ such that, for all $x\in V$,
\[
{\cal D}(x)\subseteq {\cal D}(s), \quad \partial_{\rm HS}\Pi_K(x)\subseteq \partial_{\rm HS}\Pi_K(s),
\]
and
\begin{equation}\label{eq:semisooth_HS}
\Pi_K(x)-\Pi_K(s)-U(x-s)=0\quad \forall\,U\in\partial_{\rm HS}\Pi_K(x).
\end{equation}
\end{lemma}

The next lemma gives an explicit primal--dual characterization of all points sharing the same projection onto $K$. This representation will later serve as the basic description of projection-equivalent classes discussed in Section \ref{sec:theory}.
    
\begin{lemma}\label{lemma:representation}
For any given $s\in\Rbb^n$, let $\ww\in\mM(s)$. Then
\[
\{\,s'\in\Rbb^n\mid \Pi_K(s')=\Pi_K(s)\,\}
=
\left\{
s'\in\Rbb^n\ \middle|\
\exists\,\ww'\in\Rbb^{m_I}\ \mbox{such that } s'-s=G^\top(\ww'-\ww),\ \ \ww'\le 0,\ \ \ww'_{I(s)^C}=0
\right\},
\]
where $I(s)^C:=[m_I]\backslash I(s)$.
\end{lemma}

\begin{proof}
Let $x:=\Pi_K(s)$. Since $\ww\in\mM(s)$, the pair $(x,\ww)$ satisfies \eqref{eq:kkt-proj-K}, that is,
\[
x-s+G^\top\ww=0,\quad Gx\ge g,\quad \ww\le 0,\quad \mbox{and} \quad \ww^\top(Gx-g)=0.
\]

We first prove the inclusion ``$\subseteq$''. Let $s'\in\Rbb^n$ satisfy $\Pi_K(s')=x$. Then there exists $\ww'\in\Rbb^{m_I}$ such that $(x,\ww')$ satisfies the KKT system for the projection of $s'$ onto $K$, namely,
\[
x-s'+G^\top\ww'=0,\quad Gx\ge g,\quad \ww'\le 0,\quad \mbox{and} \quad (\ww')^\top(Gx-g)=0.
\]
Since the active set at $x$ is exactly $I(s)$, the complementarity relation implies $\ww'_{I(s)^C}=0$. Subtracting the two stationarity equations yields
$s'-s=G^\top(\ww'-\ww)$. Hence the right-hand side contains every $s'$ with $\Pi_K(s')=\Pi_K(s)$.

Conversely, let $s'\in\Rbb^n$ satisfy
\[
s'-s=G^\top(\ww'-\ww),\quad \ww'\le 0\quad \mbox{and} \quad \ww'_{I(s)^C}=0
\]
for some $\ww'\in\Rbb^{m_I}$. Then
\[
x-s'+G^\top\ww'
=
x-s+G^\top\ww
=
0.
\]
Moreover, since $x=\Pi_K(s)$, we have $Gx\ge g$. For $i\notin I(s)$, one has $G_ix>g_i$ and hence $(\ww')_i=0$ by assumption; for $i\in I(s)$, one has $G_ix=g_i$. Therefore,
\[
(\ww')^\top(Gx-g)=0.
\]
Thus, $(x,w')$ satisfies the KKT system for the projection of $s'$
onto $K$, and hence $\Pi_K(s')=x=\Pi_K(s)$.
\hfill $\Box$
\end{proof}

We next record several basic variational properties of the dual function $\varphi$ given in \eqref{prob:dual-K-lin-proj}. 
From \cite[Theorem 31.5]{rockafellar1997convex}, we see that for any $y\in \Rbb^{m_E}$,
\begin{equation}\label{eq:varphiy}
\begin{aligned}
    \varphi(y) ={}& -\langle b,y\rangle+\inf_z\left\{\delta_K^*(z)+\frac12\|z-(A^\top y+c)\|^2\right\} \\
={}& -\langle b,y\rangle+\frac12\|A^\top y+c\|^2-\frac12\|A^\top y+c-\Pi_K(A^\top y + c)\|^2,
\end{aligned}
\end{equation}
and $\varphi$ is convex, continuously differentiable with $\nabla \varphi$ given in \eqref{eq:grad-dual}.
Since $K$ is polyhedral, the projector $\Pi_K$ is piecewise affine and globally Lipschitz continuous \cite[Propositions 4.1.4 and 4.2.2]{facchinei2003finite}. Consequently, $\varphi$ is a convex piecewise quadratic function, and $\nabla\varphi$ is piecewise affine and globally Lipschitz continuous. More variational properties corresponding to $\varphi$ and $\nabla \varphi$ are summarized below.

\begin{lemma}\label{lemma:error-bound}
The function $\varphi$ attains its finite minimum $\varphi^* := \min_{y\in\Rbb^{m_E}}\varphi(y) < +\infty$, with the corresponding nonempty optimal set \[\mathcal{P}^* := \argmin_{y\in\Rbb^{m_E}}\varphi(y) = \{y \in \Rbb^{m_E} \mid \nabla \varphi(y) = 0\}.\] 
For any $r_0 > 0$, there is a constant $C_0 > 0$ such that
\begin{equation}\label{eq:quad-growth}
\varphi(y) - \varphi^* \ge C_0\,\dist^2(y,\mathcal{P}^*) \quad \text{whenever } \varphi(y) \le \varphi^* + r_0.
\end{equation}
Additionally, there exist  positive constants $C_1, C$, and $r$ such that 
\begin{equation}\label{eq:quad-bound}
\varphi(y) - \varphi^* \le C_1 \,\dist^2(y,\mathcal{P}^*) \quad \forall\,y\in\Rbb^{m_E},
\end{equation}
and 
\begin{equation}\label{eq:error-bound}
\dist(y,\mathcal{P}^*) \le C \,\|\nabla\varphi(y)\| \quad \text{whenever } \|\nabla\varphi(y)\| \le r.
\end{equation}
\end{lemma}

\begin{proof}
    The finiteness of $\varphi^*$ and the nonemptiness of $\mathcal{P}^*$ follow from the duality results \cite[Corollaries 28.2.2 and 28.4.1]{rockafellar1997convex} for the projection problem \eqref{prob:K-lin-proj}.
    
    Since $\varphi$ is convex piecewise quadratic, the estimate \eqref{eq:quad-growth} follows from \cite[Theorem 2.7]{li1995error}. The quadratic upper bound \eqref{eq:quad-bound}  is due to the global Lipschitz continuity of $\nabla\varphi$ \cite[Lemma 1.2.3]{nesterov2013introductory}. Finally, since $\nabla\varphi$ is a piecewise affine map and hence a polyhedral multifunction, the local error bound \eqref{eq:error-bound} follows from  \cite[Equation (5)]{robinson1981some}. \hfill $\Box$
\end{proof}

\section{Variational properties of the projection-equivalent set}\label{sec:theory}

This section develops the variational theory underlying our semismooth
Newton framework without assuming generalized-Jacobian nonsingularity
at the solution.
The central objective is to understand, for the dual residual map
\begin{equation}\label{eq:KKT-proj-K-linear}
\nabla \varphi(y) = -b + A \Pi_{K}(A^\top  y + c), \quad y \in \Rbb^{m_E},
\end{equation}
how degeneracy can be characterized and resolved through the geometry of projection-equivalent classes. To this end, we introduce a primal--dual lifted projection-equivalent set and establish its connection with nonsingular generalized Jacobians of $\nabla\varphi$. We then derive error estimates for the associated correction step, which will later provide the key bridge from the variational theory developed here to the local and global convergence analysis in Section~\ref{sec:alg}.

\subsection{Equivalent conditions for nonsingular generalized Jacobians}\label{sec:sub-equi-condi}

Recall that throughout this paper, the matrix $A$ is assumed to have full row rank. We begin by introducing the central geometric object of this paper. For any given $y\in \mathbb{R}^{m_E}$ and $\ww\in \mM(A^\top y+c)$, define the primal--dual lifted projection-equivalent set
\begin{equation}\label{eq:mQ-y-ww}
\mQ(y,\ww)
:=
\left\{
(y',\ww') \in \Rbb^{m_E} \times \Rbb^{m_I}
\;\middle|\;
A^\top  (y' - y) = G^\top (\ww' - \ww), \ \ww' \le 0, \ \ww'_{I^C} = 0
\right\},
\end{equation}
where
\[
I = I(A^\top y + c) = \{ i \in [m_I] \mid G_i \Pi_K (A^\top  y + c) = g_i \}.
\]
The associated projection-equivalent set in the dual space is defined by
\begin{equation}\label{eq:mP}
\mP(y)
:=
\left\{
y' \in\Rbb^{m_E} \mid \Pi_K(A^\top  y' + c) = \Pi_K(A^\top y+c)
\right\}.
\end{equation}

By Lemma~\ref{lemma:representation}, the set $\mQ(y,\ww)$ gives a
primal--dual characterization of $\mP(y)$. The following proposition
collects the basic properties of these two sets that will be used
throughout the sequel.

\begin{proposition}\label{prop:mQ-property}
For any $y\in \Rbb^{m_E}$ and $\ww \in \mM(A^\top y + c)$, the following statements hold.
\begin{itemize}
    \item[{\rm (i)}] The sets $\mQ(y,\ww)$ and $\mP(y)$ are polyhedra, and
    \begin{equation}\label{eq:mP-Q}
        \mP(y)
        =
        \left\{
        y'\in \Rbb^{m_E}
        \;\middle|\;
        \exists\, \ww' \in \Rbb^{m_I}\ \mbox{such that } (y',\ww')\in \mQ(y,\ww)
        \right\}
    \end{equation}
    and
    \begin{equation}\label{eq:M-Q}
        \mM(A^\top y+c)
        =
        \left\{
        \ww' \in \Rbb^{m_I}
        \;\middle|\;
        (y,\ww')\in \mQ(y,\ww)
        \right\}.
    \end{equation}

    \item[{\rm (ii)}] The set $\mQ(y,\ww)$ and its extreme point set ${\cal E}(\mQ(y,\ww))$ are nonempty. Moreover,
    \begin{equation}\label{eq:Q-mP}
	{\cal E}(\mP(y)) \subseteq \left\{
     	\widehat y \in \Rbb^{m_E}
        \;\middle|\;
        \exists\, \widehat \ww \in \Rbb^{m_I}\ \mbox{such that } (\widehat y, \widehat \ww) \in {\cal E}(\mQ(y,\ww))
	\right\}.
    \end{equation}

    \item[{\rm (iii)}] The set-valued mappings $\mQ(\cdot,\cdot)$ and $\mP(\cdot)$ induce equivalence relations in the sense that
    \[
        \mQ(y,\ww)=\mQ(y',\ww') \quad \forall\, (y',\ww')\in \mQ(y,\ww)
    \]
    and
    \[
        \mP(y)=\mP(y') \quad \forall\, y'\in \mP(y).
    \]
\end{itemize}
\end{proposition}

\begin{proof} \textbf{(i)} The characterization \eqref{eq:mP-Q} follows directly from Lemma~\ref{lemma:representation}, applied to $s=A^\top y+c$. The relation \eqref{eq:M-Q} follows immediately from the definition of $\mQ(y,\ww)$ with $y$ fixed. Since $\mQ(y,\ww)$ is defined by linear equalities and inequalities, it is a polyhedron. Then \eqref{eq:mP-Q} implies that $\mP(y)$ is also polyhedral \cite[Lemma~2.4]{bertsimas1997introduction}.

\textbf{(ii)}  The nonemptiness of $\mQ(y,\ww)$ is immediate from $(y,\ww)\in \mQ(y,\ww)$. We next show that $\mQ(y,\ww)$ contains no lines. Let $(\Delta y,\Delta \ww)$ and $-(\Delta y,\Delta \ww)$ both belong to the recession cone of $\mQ(y,\ww)$. Then
\begin{equation}\label{eq:rankQyw}
A^\top\Delta y-G^\top\Delta \ww=0,\quad \Delta \ww\le 0\quad \mbox{and} \quad -\Delta \ww\le 0.
\end{equation}
Hence $\Delta \ww=0$, and since $A$ has full row rank, it follows that $\Delta y=0$. That is, $\mQ(y,\ww)$ contains no lines, and ${\cal E}(\mQ(y,\ww))\neq\emptyset$ by \cite[Theorem~2.6]{bertsimas1997introduction}.

Let $\widehat y\in {\cal E}(\mP(y))$. By \eqref{eq:mP-Q}, the set
\[
P_{\widehat y}
:=
\left\{
\widetilde\ww\in\Rbb^{m_I}
\;\middle|\;
(\widehat y,\widetilde\ww)\in\mQ(y,\ww)
\right\}
\]
is a nonempty polyhedron. Moreover, $P_{\widehat y}\subseteq\Rbb_-^{m_I}$, so it
contains no lines and therefore has
an extreme point, say $\widehat\ww\in{\cal E}(P_{\widehat y})$. We claim that
\(
(\widehat y,\widehat\ww)\in{\cal E}(\mQ(y,\ww)).
\)
Suppose otherwise. Then there exist distinct
$(y_1,\ww_1),(y_2,\ww_2)\in\mQ(y,\ww)$ and $\theta\in(0,1)$ such that
\[
(\widehat y,\widehat\ww)
=
\theta(y_1,\ww_1)+(1-\theta)(y_2,\ww_2).
\]
By \eqref{eq:mP-Q}, $y_1,y_2\in\mP(y)$. Since
$\widehat y\in{\cal E}(\mP(y))$, we have
$y_1=y_2=\widehat y$. Thus, $\ww_1,\ww_2\in P_{\widehat y}$ and
\[
\widehat\ww=\theta\ww_1+(1-\theta)\ww_2,
\]
contradicting $\widehat\ww\in{\cal E}(P_{\widehat y})$. Hence
\eqref{eq:Q-mP} holds.

\textbf{(iii)} The equality $\mP(y)=\mP(y')$ for all $y'\in\mP(y)$ follows immediately from the definition of $\mP(y)$. Now let $(y',\ww')\in \mQ(y,\ww)$. By \eqref{eq:mP-Q}, we have $y'\in\mP(y)$, and hence
\[
I(A^\top y'+c)=I(A^\top y+c).
\]
Moreover, from the defining equality in \eqref{eq:mQ-y-ww},
\[
A^\top y-G^\top \ww = A^\top y' - G^\top \ww'.
\]
These two facts imply directly that $\mQ(y,\ww)=\mQ(y',\ww')$. \hfill $\Box$
\end{proof}

Proposition~\ref{prop:mQ-property} shows that the lifted set $\mQ(y,\ww)$ always possesses extreme points, whereas the purely dual set $\mP(y)$ may fail to have any extreme point. Motivated by this observation, we introduce the set of $y$-components of
extreme points of $\mQ(y,\ww)$: for any $y\in \Rbb^{m_E}$ and any
$\ww\in \mM(A^\top y+c)$,
\begin{equation}\label{eq:def-hatmP}
        \widehat{\mP}(y)
        :=
        \left\{
        \widehat y\in \Rbb^{m_E}
        \;\middle|\;
        \exists\,\widehat \ww\in \Rbb^{m_I}\ \mbox{such that }
        (\widehat y,\widehat \ww)\in {\cal E}(\mQ(y,\ww))
        \right\}.
\end{equation}
Although the definition of \(\widehat{\mathcal P}(y)\) is written using
a multiplier \(w\in\mathcal M(A^\top y+c)\), the resulting set is
independent of the choice of \(w\). Indeed, for any
\(w_1,w_2\in\mathcal M(A^\top y+c)\), one has
\((y,w_2)\in\mathcal Q(y,w_1)\), and hence
\[
\mathcal Q(y,w_1)=\mathcal Q(y,w_2)
\]
by Proposition~\ref{prop:mQ-property}(iii).
Moreover, if $y'\in\mP(y)$, then $\widehat{\mP}(y')=\widehat{\mP}(y)$.

Next, we establish the connections between the extreme points of the lifted set and nonsingular generalized Jacobians of $\nabla\varphi$. For any $y\in\Rbb^{m_E}$, define the computable generalized Jacobian of $\nabla \varphi$:
\begin{equation}\label{eq:gen-jac}
    \partial^2 \varphi(y)
    :=
    \left\{
    W \in \Rbb^{m_E\times m_E}
    \;\middle|\;
    W = A U A^\top,\ \
    U \in \partial_{\rm HS}\Pi_K(A^\top y + c)
    \right\}.
\end{equation}
Here $\partial_{\mathrm{HS}}\Pi_K$ denotes the Han--Sun generalized Jacobian introduced in Section~\ref{sec:prelim}. Moreover, since $\nabla\varphi(y)=A\Pi_K(A^\top y+c)-b$, one can check from Lemma~\ref{lemma:semismooth_HS} that, for any given $\overline y\in\Rbb^{m_E}$, there exists a neighborhood $N$ of $\overline y$ such that, for all $y\in N$,
\[
\partial^2\varphi(y)\subseteq \partial^2\varphi(\overline y)
\]
and
\begin{equation}\label{eq:semismooth-varphi}
    \nabla \varphi(y)-\nabla \varphi(\overline y)-W(y-\overline y)=0
    \qquad \forall\,W\in\partial^2\varphi(y).
\end{equation}
In particular, $\nabla\varphi$ is strongly semismooth with respect to $\partial^2\varphi$ in the sense of \cite[Definition 1]{li2018efficiently}.

The following theorem is the main result of this subsection. It characterizes when a lifted pair $(y,w)$ is an extreme point of
$\mathcal Q(y,w)$, equivalently, when the generalized Jacobian of $\nabla\varphi$ is nonsingular.
In addition, it shows that the extreme-point geometry, nonsingularity, a suitable full-rank condition, and the W-SRCQ-type condition introduced in \cite{feng2024quadratically} are equivalent characterizations of the same underlying regularity property.

\begin{theorem}\label{thm:nons-gen-jac}
For any given $y\in \Rbb^{m_E}$ and $\ww \in \mM(A^\top y + c)$, let $J = {\rm supp}(\ww)$. Then the following four statements are equivalent:
\begin{itemize}
	\item[{\rm (i)}] $(y,\ww)$ is an extreme point of $\mQ(y,\ww)$, i.e., $(y,\ww)\in {\cal E}(\mQ(y,\ww))$;

	\item[{\rm (ii)}] the matrix $\begin{bmatrix}  A^\top & G_J^\top  \end{bmatrix} \in \Rbb^{n\times (m_E + |J|)}$ has full column rank;

	\item[{\rm (iii)}] the matrix $G_J$ has full row rank and
    \begin{equation}\label{eq:def-AUAT}
        AUA^\top  \in \partial^2\varphi(y) \ \mbox{is nonsingular},
    \end{equation}
	where $U = I_{n} - G_J^\top (G_J G_J^\top )^{-1} G_J \in \partial_{\rm HS}\Pi_K(A^\top  y + c)$;

    \item[{\rm (iv)}] with $x=\Pi_K(A^\top y+c)$, the following condition holds:
    \begin{align}\label{eq:w-SRCQ-Rbb+}
        \begin{bmatrix}
            A \\ G
        \end{bmatrix}\Rbb^n+
        \begin{bmatrix}
            \{0\} \\
            \aff(\mT_{\Rbb^{m_I}_+}(Gx-g)\cap {\ww}^{\perp})
        \end{bmatrix}=
        \begin{bmatrix}
            \Rbb^{m_E} \\
            \Rbb^{m_I}
        \end{bmatrix}.
    \end{align}
\end{itemize}
\end{theorem}

\begin{proof}
{\bf (i)$\iff$(ii)}: By the definition of $\mQ(y,\ww)$ in \eqref{eq:mQ-y-ww}, the point $(y,\ww)$ is an extreme point of $\mQ(y,\ww)$ if and only if the implication
\begin{equation}\label{eq:Q-extreme}
\left\{
\begin{array}{l}
A^\top  \Delta y - G^\top  \Delta \ww = 0,\\[3pt]
\Delta \ww_{J^C} = 0
\end{array}
\right.
\quad\Longrightarrow\quad
(\Delta y,\Delta \ww)=0
\end{equation}
holds. Since $\Delta \ww_{J^C}=0$, condition \eqref{eq:Q-extreme} is equivalent to
\[
A^\top \Delta y = G_J^\top  \Delta \ww_J \quad\Longrightarrow\quad (\Delta y,\Delta \ww_J)=0.
\]
Because $A$ has full row rank, this holds if and only if
\[
\tIm(A^\top )\cap \tIm(G_J^\top )=\{0\}
\]
and $G_J$ has full row rank, which is equivalent to $\begin{bmatrix}  A^\top & G_J^\top  \end{bmatrix}$ having full column rank.

{\bf (ii)$\iff$(iii)}: the nonsingularity of $AUA^\top $ is equivalent to
    \[
    A U A^\top  d = 0 \quad \Longrightarrow \quad d = 0.
    \]
    Since $U$ is symmetric positive semidefinite, $AUA^\top  d=0$ implies
    $d^\top  AUA^\top  d=0$ and hence $(A^\top  d)^\top  U (A^\top  d)=0$, which yields $UA^\top  d=0$.
    Conversely, $UA^\top  d=0$ obviously implies $AUA^\top  d=0$.
    Therefore, the above condition is equivalent to
    \[
    U A^\top  d = 0 \quad \Longrightarrow \quad d = 0.
    \]
    Since $\tNu(U)=\tIm(G_J^\top )$ and $A$ has full row rank, the above relation holds if and only if
    \begin{equation}\label{eq:nons-AUA}
    	\tIm(A^\top )\cap \tIm(G_J^\top )=\{0\},
    \end{equation}
    which, together with $G_J$ being full row rank, is exactly (ii).

{\bf (ii)$\iff$(iv)}: Using the complementarity relations
\[
Gx-g\in\Rbb^{m_I}_+,\quad \ww\le 0\quad \mbox{and} \quad \ww^\top (Gx-g)=0,
\]
one obtains
\[
\aff\!\big(\mT_{\Rbb^{m_I}_+}(Gx-g)\cap \ww^\perp\big)
=
\{d\in\Rbb^{m_I}\mid d_J=0\}.
\]
Hence \eqref{eq:w-SRCQ-Rbb+} is equivalent to
\[
\begin{bmatrix}
A\\G_J
\end{bmatrix}\Rbb^n
=
\begin{bmatrix}
\Rbb^{m_E}\\ \Rbb^{|J|}
\end{bmatrix},
\]
which in turn is equivalent to $\begin{bmatrix}  A^\top & G_J^\top  \end{bmatrix}$ having full column rank. \hfill $\Box$
\end{proof}

The above characterization also admits the following two remarks, which clarify respectively the trivial dual-feasible case and the role of the associated weak regularity condition.

\begin{remark}
    If $J=\emptyset$, then $A^\top y+c\in K$ and hence $\Pi_K(A^\top y+c)=A^\top y+c$. In this case, we choose the Han--Sun generalized Jacobian $U=I_n$. Therefore, $AUA^\top=AA^\top$ is an element of $\partial^2\varphi(y)$, and it is nonsingular since $A$ has full row rank. 
\end{remark}

\begin{remark}\label{remark:w-srcq}
Condition \eqref{eq:w-SRCQ-Rbb+} can be interpreted as a generalized weak strict Robinson constraint qualification (W-SRCQ), introduced in \cite[Remark 13]{feng2024quadratically}, for the reformulation of \eqref{prob:K-lin-proj}
\begin{equation}\label{prob:1-re}
    \begin{array}{cl}
        \min & \displaystyle \frac{1}{2} \| x-c\|^2 \\[10pt]
        \mbox{s.t.} & \begin{bmatrix}
            A \\ G
        \end{bmatrix} x -\begin{bmatrix}
            b\\ g
        \end{bmatrix}\in \begin{bmatrix}
            \{0\}\\ \Rbb^{m_I}_+
        \end{bmatrix}.
    \end{array}
\end{equation}
In \cite{feng2024quadratically}, it is shown that the W-SRCQ is closely related to the existence of a nonsingular generalized Jacobian of the KKT residual mapping. 
Here, we further show that, for the polyhedral projection problem, every lifted projection-equivalence class contains an extreme pair at which the corresponding W-SRCQ-type full-rank condition holds. 
\end{remark}

\subsection{A simplification under the linear independence property}

The results developed so far do not require additional structural assumptions on the polyhedron $K$. Under the following linear independence property, however, the projection-equivalent set admits a simpler description in the dual space, and the extreme point set ${\cal E}(\mP(y))$ becomes nonempty.

\begin{definition}\label{def:LIP}
A nonempty polyhedron $K = \{x\in \Rbb^n \mid Gx \ge g\}$ is said to have the \emph{Linear Independence Property (LIP)} if, for every $x \in K$,
\begin{equation}\label{eq:LIP}
G_I \ \mbox{has full row rank}, \quad I = \{ i\in [m_I] \mid G_i x = g_i \}.
\end{equation}
\end{definition}

The LIP holds, for instance, when $K$ is the nonnegative orthant, a box, or a full dimensional generalized box in \eqref{eq:gen-box}. Thus, this assumption is fully consistent with our simplicity requirement on $K$.
Under the LIP, the multiplier associated with a projection point becomes unique, and the set $\mP(y)$ admits an explicit description.

\begin{proposition}\label{prop:LIP-explicit-P}
Suppose that $K = \{x\in \Rbb^n \mid Gx \ge g\}$ has the LIP. For any $y\in \Rbb^{m_E}$, the set ${\cal M}(A^\top y + c)$ is a singleton. Moreover,
\begin{equation}\label{eq:LIP-explicit-P}
\mP(y)
=
\left\{
y'\in \Rbb^{m_E}
\;\middle|\;
\begin{aligned}
& A^\top (y - y') = G_I^\top (G_IG_I^\top )^{-1}G_IA^\top (y - y'),\\
& (G_I G_I^\top )^{-1}G_I(A^\top y'+ c - \Pi_K(A^\top y + c) ) \le 0
\end{aligned}
\right\},
\end{equation}
where
\[
I = \{i\in [m_I] \mid G_i \Pi_K(A^\top y + c) = g_i \}.
\]
\end{proposition}

\begin{proof}
Let $x=\Pi_K(A^\top y+c)$ and $I=\{i\in[m_I]\mid G_i x=g_i\}$. By the LIP, $G_I$ has full row rank. The KKT system \eqref{eq:kkt-proj-K} then determines the active multiplier uniquely through
\[
x-(A^\top y+c)+G_I^\top \ww_I=0,
\quad \ww_I\le 0,
\quad \ww_{I^C}=0.
\]
Hence ${\cal M}(A^\top y+c)$ is a singleton. The representation \eqref{eq:LIP-explicit-P} then follows by applying Lemma~\ref{lemma:representation} with the unique multiplier and eliminating it through the full-row-rank relation of $G_I$. \hfill $\Box$
\end{proof}

Under the LIP, we can also describe the extreme points of $\mP(y)$ explicitly.

\begin{proposition}\label{prop:Q-mP-inv}
Suppose that $K = \{x\in \Rbb^n \mid Gx \ge g\}$ has the LIP. For any $y\in \Rbb^{m_E}$, one has
\begin{equation}\label{eq:Q-mP-inv}
	\emptyset\neq{\cal E}(\mP(y)) =\widehat{\mP}(y),
\end{equation}
where $\widehat{\mP}$ is defined in \eqref{eq:def-hatmP}.
\end{proposition}

\begin{proof}
Proposition~\ref{prop:LIP-explicit-P} implies the uniqueness of the multiplier $\{w\}={\cal M}(A^\top y+c)$.
Recall from Proposition~\ref{prop:mQ-property} that
\[
{\cal E}(\mP(y)) \subseteq \widehat{\mP}(y)=\left\{
     	\widehat y \in \Rbb^{m_E} \;\middle|\; \exists\, \widehat \ww \in \Rbb^{m_I}\ \mbox{such that } (\widehat y, \widehat \ww) \in {\cal E}(\mQ(y,\ww))
	\right\} \neq \emptyset.
\]
Hence, it remains to prove the reverse inclusion. Fix any $\widehat y\in \widehat{\mP}(y)$. By definition, there exists $\widehat \ww\in \Rbb^{m_I}$ such that $(\widehat y,\widehat \ww)\in {\cal E}(\mQ(y,\ww))$. We next show that $\widehat y\in {\cal E}(\mP(y))$. Suppose on the contrary that $\widehat y\notin {\cal E}(\mP(y))$. Then there exist $y^1,y^2\in \mP(y) = \mP(\widehat y)$, both different from $\widehat y$, and a scalar $\lambda\in(0,1)$ such that
\[
\widehat y=\lambda y^1+(1-\lambda) y^2.
\]
By \eqref{eq:mP-Q}, there exist $\ww^1,\ww^2\in\Rbb^{m_I}$ such that
\[
(\,y^1,\ww^1\,),\ (\;y^2,\ww^2\;)\in \mQ(y,\ww) = \mQ(\widehat y,\widehat \ww).
\]
Let $I=\{i\in[m_I]\mid G_i\Pi_K(A^\top \widehat y+c)=g_i\}$. Then, we have from \eqref{eq:mQ-y-ww} that
\[
A^\top(y^1-\widehat y)=G^\top(\ww^1-\widehat \ww),\quad
A^\top(y^2-\widehat y)=G^\top(\ww^2-\widehat \ww)
\]
with $\widehat\ww_{I^C} = \ww^1_{I^C}=\ww^2_{I^C}=0$. Then, it holds that
\[
G_I^\top(\lambda\ww_I^1+(1-\lambda)\ww_I^2- \widehat \ww_I)=0.
\]
Since the LIP implies that $G_I$ has full row rank, it follows that
\[
\widehat \ww_I=\lambda\ww_I^1+(1-\lambda)\ww_I^2.
\]
Hence
\[
(\widehat y,\widehat \ww)=\lambda(y^1,\ww^1)+(1-\lambda)(y^2,\ww^2),
\]
which contradicts the assumption that $(\widehat y,\widehat \ww)$ is an extreme point of $\mQ(y,\ww)$. Therefore, $\widehat y\in {\cal E}(\mP(y))$. \hfill $\Box$
\end{proof}

\begin{remark}
    Proposition~\ref{prop:Q-mP-inv} shows that, under the LIP, it suffices to focus on the dual projection-equivalent set ${\cal P}(y)$ and its extreme points, since
    $
    \mathcal E(\mathcal P(y))=\widehat{\mathcal P}(y)\neq\emptyset .
    $
    This is the case, for example, for structured problems such as the nearest doubly stochastic matrix problem studied in~\cite{hu2024semismooth}, where $K=\mathbb{R}_+^{m^2}$ satisfies the LIP under the formulation~\eqref{prob:K-lin-proj}. When the LIP fails, however, the situation changes substantially. The multiplier associated with a projection point may no longer be unique, and ${\cal E}(\mathcal P(y))$ may even be empty. In such cases, searching for an extreme point of $\mathcal P(y)$ is not feasible in general. Passing instead to the primal--dual lifted set $\mathcal Q(y,w)$ avoids the need for the LIP. 
    This explains why the primal--dual lifting is needed in the general polyhedral setting.
\end{remark}

\subsection{Error bound analysis with respect to the projection-equivalent set}\label{sec:sub-error}

Theorem~\ref{thm:nons-gen-jac} shows that a nonsingular generalized Jacobian can always be obtained at a suitable extreme point of the lifted projection-equivalent set. To exploit this fact in the subsequent Newton analysis, it remains to understand how far such an extreme point may lie from the current iterate and, more importantly, how this displacement scales relative to the distance to the target set. The purpose of this subsection is to establish precisely such estimates.

\begin{lemma}\label{lemma:dist-hat-y-opt}
There exist constants $C,C'>0$, depending only on $A$ and $G$,
such that the following statements hold. Fix any $\overline y\in\Rbb^{m_E}$ and any
$\overline\ww\in\mM(A^\top\overline y+c)$. Then, for every
$y\in\Rbb^{m_E}$ satisfying
\(
I(A^\top y+c)
\subseteq
I(A^\top\overline y+c),
\)
every $\ww\in\mM(A^\top y+c)$, and every
\(
(\widehat y,\widehat\ww)
\in{\cal E}\bigl(\mQ(y,\ww)\bigr),
\)
the following assertions hold.

\begin{enumerate}
\item[{\rm (i)}]
For every $y^*\in\mP(\overline y)$,
\[
\dist\bigl(\widehat y,\mP(\overline y)\bigr)
\le
\dist\bigl(
    (\widehat y,\widehat\ww),
    \mQ(\overline y,\overline\ww)
\bigr)
\le
C\|y-y^*\|.
\]
Consequently,
\begin{equation}\label{eq:dist_in1}
\dist\bigl(\widehat y,\mP(\overline y)\bigr)
\le \dist\bigl(
    (\widehat y,\widehat\ww),
    \mQ(\overline y,\overline\ww)
\bigr) \le 
C\,\dist\bigl(y,\mP(\overline y)\bigr).
\end{equation}

\item[{\rm (ii)}]
There exists a point
\[
(y^{\rm e},\ww^{\rm e})
\in
{\cal E}\bigl(\mQ(\overline y,\overline\ww)\bigr)
\]
such that
\begin{equation}\label{eq:dist_in2}
\|\widehat y-y^{\rm e}\|
\le
C'\,\dist\bigl(y,\mP(\overline y)\bigr).
\end{equation}
\end{enumerate}
\end{lemma}

\begin{proof}
We first note that, for any $(y',\ww')\in \mQ(\overline y,\overline \ww)$, one has $y'\in \mP(\overline y)$ by \eqref{eq:mP-Q}. Hence
\[
{\rm dist}(\widehat y, \mP(\overline y))
\le
{\rm dist}((\widehat y,\widehat \ww), \mQ(\overline y,\overline \ww)).
\]
We next derive the upper bound. For each $I\subseteq[m_I]$ and
$r\in\Rbb^n$, define
\[
\mQ_I(r)
:=
\left\{
(v,q)\in\Rbb^{m_E}\times\Rbb^{m_I}
\;\middle|\;
A^\top v-G^\top q=r,\quad
q\le0,\quad
q_{I^C}=0
\right\}.
\]
By Hoffman's error bound
\cite[Lemma~3.2.3]{facchinei2003finite}, for every
$I\subseteq[m_I]$, there exists a constant $H_I>0$, depending only
on the coefficient matrices and not on the right-hand side $r$, such
that, whenever $\mQ_I(r)\neq\emptyset$,
\begin{equation}\label{eq:Hoffman-QI}
\begin{aligned}
\dist\bigl((v,q),\mQ_I(r)\bigr)
\le H_I
\Big(
    &\|A^\top v-G^\top q-r\|
    +\|q_{I^C}\|
    +\|\max(q,0)\|
\Big)
\end{aligned}
\end{equation}
for every $(v,q)\in\Rbb^{m_E}\times\Rbb^{m_I}$.
Since there are only finitely many subsets of $[m_I]$, we may set
\[
C_1:=\max_{I\subseteq[m_I]}H_I<\infty.
\]
Thus, $C_1$ depends only on $A$ and $G$.
Now, let
\[
I_y:=I(A^\top y+c),
\qquad
\overline I:=I(A^\top\overline y+c),
\qquad
\overline r:=A^\top\overline y-G^\top\overline\ww.
\]
Then
\(
\mQ(\overline y,\overline\ww)
=
\mQ_{\overline I}(\overline r),
\)
and this set is nonempty because it contains
$(\overline y,\overline\ww)$.
Since $(\widehat y,\widehat\ww)\in\mQ(y,\ww)$, we have
\[
\widehat\ww\le0,
\qquad
\widehat\ww_{I_y^C}=0.
\]
Moreover, $I_y\subseteq\overline I$ implies
$\overline I^C\subseteq I_y^C$, and hence
\(
\widehat\ww_{\overline I^C}=0.
\)
Hence, applying \eqref{eq:Hoffman-QI} with $I = \overline I$, $r = \overline r$, $(v,q) = (\widehat y,\widehat\ww)$, and using $H_I \le C_1$, we obtain, for every $y^*\in\mP(\overline y)$ and every $\ww^*\in\mM(A^\top y^* + c)$, that
\begin{equation}\label{eq:disthaty_hata_Q_bary_barl}
\begin{aligned}
{\rm dist}((\widehat y,\widehat \ww), \mQ(\overline y,\overline \ww)) 
&\le C_1 \| A^\top (\widehat y - \overline y) - G^\top (\widehat \ww - \overline \ww) \| \\
&= C_1 \|A^\top  (y - y^*) - G^\top (\ww - \ww^*)\| \\
&= C_1 \| \Pi_{K}(A^\top  y + c) - \Pi_K(A^\top  y^* + c) \| \\
&\le C_1\| A^\top (y - y^*)\| \le C \|y - y^*\|,
\end{aligned}
\end{equation}
where $C = C_1 \|A\|$.
Here, the equalities follow from the defining relation of $\mQ$ and
the KKT conditions \eqref{eq:kkt-proj-K}.
Since \eqref{eq:disthaty_hata_Q_bary_barl} holds for every $y^*\in\mP(\overline y)$, taking the
infimum over $y^*$ proves \eqref{eq:dist_in1}.

It remains to prove \eqref{eq:dist_in2}. Let $(\widetilde y, \widetilde \ww)$ be the projection of $(\widehat y, \widehat \ww)$ onto $\mQ(\overline y, \overline \ww)$. If $(\widetilde y, \widetilde \ww)\in {\cal E}(\mQ(\overline y, \overline \ww))$, then \eqref{eq:dist_in2} follows immediately from \eqref{eq:disthaty_hata_Q_bary_barl}. Otherwise, we construct an extreme point of $\mQ(\overline y, \overline \ww)$ by a finite support-reduction procedure.

Set $(\widetilde y^0, \widetilde \ww^0)=(\widetilde y,\widetilde \ww)$. Since $(\widetilde y^0,\widetilde \ww^0)\notin {\cal E}(\mQ(\overline y,\overline \ww))$, there exists a nonzero direction $(\Delta y^0,\Delta \ww^0)$ with $\|(\Delta y^0,\Delta \ww^0)\|=1$ such that
\begin{equation}\label{eq:non-extreme-qbarybaralpha}
\left\{
\begin{array}{l}
A^\top  \Delta y^0 - G^\top  \Delta \ww^0 = 0,\\[3pt]
\Delta \ww^0_{{\rm supp}(\widetilde \ww^0)^C} = 0.
\end{array}
\right.
\end{equation}
Since $A$ has full row rank, \eqref{eq:non-extreme-qbarybaralpha} implies $\Delta \ww^0\neq 0$. Define
\[
\tau^0 = \min_{i\in {\rm supp}(\widetilde \ww^0), \, \Delta \ww^0_i \neq 0} \frac{|\widetilde \ww_i^0|}{|\Delta \ww_i^0|}
\]
and choose
\[
i^0 \in {\rm argmin}_{i\in {\rm supp}(\widetilde \ww^0), \, \Delta \ww^0_i \neq 0} \frac{|\widetilde \ww_i^0|}{|\Delta \ww_i^0|}.
\]
Then set
\begin{equation}\label{eq:update_tilde_y1_a1}
\left\{
\begin{aligned}
& \widetilde y^1 = \widetilde y^0 + {\rm sign}(\Delta \ww_{i^0}^0) \tau^0 \Delta y^0, \\
& \widetilde \ww^1 = \widetilde \ww^0 + {\rm sign}(\Delta \ww_{i^0}^0) \tau^0 \Delta \ww^0.
\end{aligned}
\right.
\end{equation}
By construction, $(\widetilde y^1,\widetilde \ww^1)\in \mQ(\overline y,\overline \ww)$ and
\[
{\rm supp}(\widetilde \ww^1)\subsetneq {\rm supp}(\widetilde \ww^0).
\]
In particular, we have $|{\rm supp}(\widetilde \ww^1)^C| \ge |{\rm supp}(\widetilde \ww^0)^C|+ 1$.

We now bound $\tau^0$. For each $J\subseteq[m_I]$, define the linear mapping
\[
L_J(\Delta y,\Delta\ww)
:=
\begin{bmatrix}
A^\top\Delta y-G^\top\Delta\ww\\
\Delta\ww_{J^C}
\end{bmatrix},
\]
and let
\[
\mathscr J_{\rm inj}
:=
\left\{
J\subseteq[m_I]
\;\middle|\;
L_J \text{ is injective}
\right\}.
\]
Note that $\mathscr J_{\rm inj}$ is nonempty because $L_{\emptyset}$ is injective when $A$ has full row rank.
Since there are only finitely many subsets of $[m_I]$, the collection
$\mathscr J_{\rm inj}$ is finite. Define
\[
C_2
:=
\min_{J\in\mathscr J_{\rm inj}}
\sigma_{\min}(L_J)>0.
\]
Note that $C_2$ depends only on $A$ and $G$.
Let $\widehat J:=\supp(\widehat\ww)$. Since
$(\widehat y,\widehat\ww)$ is an extreme point of $\mQ(y,\ww)$,
the extreme-point characterization \eqref{eq:Q-extreme} implies that
$L_{\widehat J}$ is injective. Hence
$\widehat J\in\mathscr J_{\rm inj}$. Using
$\|(\Delta y^0,\Delta\ww^0)\|=1$, we obtain
\[
\left\|
\begin{bmatrix}
A^\top\Delta y^0-G^\top\Delta\ww^0\\
\Delta\ww^0_{\widehat J^C}
\end{bmatrix}
\right\|
\ge C_2.
\]
Since the first block vanishes by
\eqref{eq:non-extreme-qbarybaralpha}, it follows that
\(
\|\Delta\ww^0_{\widehat J^C}\|\ge C_2.
\)
Moreover,
$\Delta\ww^0_{\supp(\widetilde\ww^0)^C}=0$, so there exists
\[
j_0\in
\supp(\widetilde\ww^0)\cap\widehat J^C
\]
such that
\(
|\Delta\ww^0_{j_0}|
\ge
{C_2}/{\sqrt{m_I}}.
\)
Therefore,
\begin{equation}\label{eq:upper_bound_tau0}
\tau^0 \le \frac{|\widetilde \ww_{j_0}^0|}{|\Delta \ww_{j_0}^0|}
\le \frac{\sqrt{m_I}}{C_2} |\widetilde \ww_{j_0}^0|.
\end{equation}
Since $j_0\in \widehat J^C = {\rm supp}(\widehat \ww)^C$, one has $\widehat \ww_{j_0}=0$ and \(
|\widetilde \ww_{j_0}^0|
\le
\|\widetilde \ww^0-\widehat \ww\|\), which, together with \eqref{eq:upper_bound_tau0}, implies 
\begin{equation*}
\begin{aligned}
\|(\widehat y, \widehat \ww) - (\widetilde y^1, \widetilde \ww^1)\|
\le{}& \|(\widehat y, \widehat \ww) - (\widetilde y^0, \widetilde \ww^0)\| + \|(\widetilde y^0, \widetilde \ww^0) - (\widetilde y^1, \widetilde \ww^1)\| \\
\le{}& \|(\widehat y, \widehat \ww) - (\widetilde y^0, \widetilde \ww^0)\| + \tau^0 \\
\le{}& \|(\widehat y, \widehat \ww) - (\widetilde y^0, \widetilde \ww^0)\| + \frac{\sqrt{m_I}}{C_2} \|\widetilde \ww^0 - \widehat \ww\| \\
\le{}& \Big(1 + \frac{\sqrt{m_I}}{C_2}\Big) \|(\widehat y, \widehat \ww) - (\widetilde y^0, \widetilde \ww^0)\|.
\end{aligned}
\end{equation*}
Recall that $(\widetilde y^0, \widetilde \ww^0)$ is the projection of $(\widehat y, \widehat \ww)$ onto $\mQ(\overline y, \overline \ww)$, \eqref{eq:dist_in1} ensures that
\[
\|(\widehat y, \widehat \ww) - (\widetilde y^0, \widetilde \ww^0)\|
=
{\rm dist}((\widehat y, \widehat \ww), \mQ(\overline y, \overline \ww))
\le
C\,{\rm dist}(y,\mP(\overline y)).
\]
Hence
\[
\|(\widehat y, \widehat \ww) - (\widetilde y^1, \widetilde \ww^1) \|
\le
C \Big(1+\frac{\sqrt{m_I}}{C_2}\Big){\rm dist}(y,\mP(\overline y)).
\]
If $(\widetilde y^1,\widetilde\ww^1)$ is not an extreme point of
$\mQ(\overline y,\overline\ww)$, we repeat the same
support-reduction construction. At every subsequent step, the same
argument gives
\[
\begin{aligned}
\|(\widehat y,\widehat\ww)
  -(\widetilde y^{s+1},\widetilde\ww^{s+1})\|
\le
\left(1+\frac{\sqrt{m_I}}{C_2}\right)
\|(\widehat y,\widehat\ww)
  -(\widetilde y^s,\widetilde\ww^s)\|.
\end{aligned}
\]
The support of $\widetilde\ww^s$ decreases strictly at each step. If the support becomes empty, Theorem~\ref{thm:nons-gen-jac},
together with the full row rank of $A$, implies that the current point
is extreme. Hence, the procedure terminates after at most $m_I$ steps at some extreme point
\(
(y^{\rm e},\ww^{\rm e})
\in{\cal E}\bigl(\mQ(\overline y,\overline\ww)\bigr)
\)
satisfying
\[
\|\widehat y-y^{\rm e}\|
\le
\|(\widehat y,\widehat\ww)-(y^{\rm e},\ww^{\rm e})\|
\le
C\left(1+\frac{\sqrt{m_I}}{C_2}\right)^{m_I}
\dist\bigl(y,\mP(\overline y)\bigr).
\]
Thus, \eqref{eq:dist_in2} holds with
\(
C'
=
C\left(1+\frac{\sqrt{m_I}}{C_2}\right)^{m_I}.
\)
\hfill$\Box$
\end{proof}

Recall from Lemma \ref{lemma:error-bound} that
\[
\mP^*=\{y^* \in \Rbb^{m_E} \mid \nabla\varphi(y^*)=0\} 
\]
is the optimal solution set of the dual problem \eqref{prob:dual-K-lin-proj}, and let $x^* \in \Rbb^n$ be the unique optimal solution of the primal projection problem \eqref{prob:K-lin-proj}. Then, \eqref{eq:grad-dual} and \eqref{eq:mP} imply that
\begin{equation}\label{eq:mP*}
    \Pi_K(A^\top y^*+c)=x^*
    \quad\mbox{and}\quad
    {\cal P}(y^*) = {\cal P}^*
    \quad  \forall \, y^*\in\mP^*.
\end{equation}
We define the set of extreme-point representatives of the optimal
projection-equivalence class by
\begin{equation}\label{eq:def-hatmP-star}
\widehat{\mP}^{*}
:=
\widehat{\mP}(y^{*})
\quad \text{for any } y^{*}\in\mP^{*}.
\end{equation}
Proposition~\ref{prop:mQ-property}(iii) and
\eqref{eq:def-hatmP} imply that this definition is independent of the choice of \(y^{*}\in\mP^{*}\).
We also denote the optimal active set by
\[
I^* := \left\{ i\in [m_I] \mid G_ix^* = g_i\right\}
= I(A^\top  y^* + c) \quad \forall\, y^*\in {\cal P}^*.
\]
The next lemma shows that the active index set is locally stable around the optimal solution set.

\begin{lemma} \label{lemma:mP*-I-stable}
There exists a positive constant $R_0$ such that $$I(A^\top y+c)\subseteq I^* \quad \forall\, y\in \left\{y \in \Rbb^{m_E} \mid \dist(y, \mP^*) \le R_0 \right\}.
$$
\end{lemma}

\begin{proof}
For any $y\in\Rbb^{m_E}$, $y^*\in\mP^*$, and $i\in [m_I]$, the nonexpansiveness of $\Pi_K$ implies
\[
\|G_i(\Pi_{K}(A^\top  y + c) - x^*)\|
=
\|G_i(\Pi_K(A^\top  y+c) - \Pi_K(A^\top   y^* + c))\|
\le \|G\| \|A\| \|y - y^*\|.
\]
Hence
\[
\|G_i(\Pi_{K}(A^\top  y + c) - x^*)\|
\le \|G\| \|A\| \dist(y,\mP^*).
\]
If $I^* = [m_I]$, the conclusion is immediate. Otherwise, for any $i\in[m_I]\backslash I^*$, one has $G_i x^* - g_i>0$. Let
\[
R := \min_{i\in [m_I]\backslash I^*} (G_i x^* - g_i) >0.
\]
Then
\[
G_i \Pi_K(A^\top y + c) - g_i
\ge
G_i x^* - g_i - \|G\| \|A\| \dist(y,\mP^*)
\ge
R - \|G\| \|A\| \dist(y,\mP^*).
\]
Therefore, if $0 < R_0 < R/(\|G\| \|A\|)$, then for all $y$ satisfying $\dist(y,\mP^*) \le R_0$,
\[
G_i\Pi_K(A^\top y + c) - g_i > 0 \quad \forall\, i\in[m_I]\backslash I^*.
\]
This proves that $I(A^\top y+c)\subseteq I^*$. \hfill $\Box$
\end{proof}

Combining Lemma~\ref{lemma:dist-hat-y-opt} with the local stability of the active set near $\mP^*$, we obtain the following error bound estimate for the correction step in a neighborhood of the optimal solution set.

\begin{proposition}\label{prop:dist-hatyhatP-yP}
There exist constants $ C > 0 $ and $ R_0 > 0 $ such that for all $(y,\ww)$ satisfying $\dist(y, \mathcal{P}^*) \le R_0$, $\ww\in {\cal M}(A^\top y + c)$, and any $(\widehat{y}, \widehat{w}) \in {\cal E}(\mathcal{Q}(y, \ww))$, it holds that
\begin{equation}\label{eq:dist-hatyhatP-yP}
\dist(\widehat{y}, \widehat{\mathcal{P}}^*) \le C \dist(y, \mathcal{P}^*),
\end{equation}
where $\widehat{\mathcal{P}}^*$ is defined as in \eqref{eq:def-hatmP-star}.
\end{proposition}

\begin{proof}
Take any $y^*\in \mP^*$. By Lemma~\ref{lemma:mP*-I-stable}, if $\dist(y,\mP^*)\le R_0$, then
\[
I(A^\top y+c)\subseteq I(A^\top y^*+c)=I^*.
\]
Hence, by Lemma~\ref{lemma:dist-hat-y-opt}, applied with $\overline y=y^*$ and an arbitrary $\overline \ww\in\mM(A^\top y^*+c)$, we have
\[
\dist(\widehat y,\widehat{\mP}(y^*)) \le C \dist(y,\mP(y^*)).
\]
This is exactly \eqref{eq:dist-hatyhatP-yP} since $\mP(y^*)=\mP^*$ and $\widehat{\mP}(y^*)=\widehat{\mP}^*$ by \eqref{eq:mP*} and \eqref{eq:def-hatmP-star}.
\hfill $\Box$
\end{proof}

\section{Dual semismooth Newton methods with extreme-point correction} \label{sec:alg}

In this section, we exploit the variational theory developed in Section~\ref{sec:theory} to construct dual semismooth Newton methods for problem~\eqref{prob:dual-K-lin-proj}. The key insight is that, although the generalized Jacobian of $\nabla\varphi$ may be singular at the current iterate $y$, Theorem~\ref{thm:nons-gen-jac} guarantees that a nonsingular generalized Jacobian is always available at an extreme point of the lifted projection-equivalent set $\mQ(y,w)$, while Proposition~\ref{prop:dist-hatyhatP-yP} ensures that the associated correction is compatible with the local Newton mechanism. Based on these results, we first develop a local inexact semismooth Newton method with extreme-point correction and establish its local superlinear convergence without imposing nonsingularity or regularity assumptions. We then globalize the method by combining a monotone extreme-point correction with a Wolfe line search, thereby obtaining a globally convergent semismooth Newton framework that eventually recovers the local superlinear rate.

\subsection{Dual semismooth Newton method with extreme-point correction and its local convergence}\label{sec:sub-alg-local}

We first derive a local semismooth Newton framework from the variational characterization in Theorem~\ref{thm:nons-gen-jac}. At each iteration, the current pair $(y^k,w^k)$ is first corrected to an extreme point $(\widehat y^k,\widehat w^k)\in{\cal E}(\mQ(y^k,w^k))$. By Theorem~\ref{thm:nons-gen-jac}, this correction yields a representative at which a nonsingular generalized Jacobian of $\nabla\varphi$ is available. The subsequent Newton step is then computed at the corrected point. In this way, the local semismooth Newton linearization is well-posed by construction rather than by assumption.

\begin{algorithm}[H]
\caption{An inexact dual semismooth Newton method with extreme-point correction for \eqref{prob:dual-K-lin-proj}}
\label{alg:d-SN-poly}
\begin{algorithmic}[1]
\State Initialize $y^0$, $w^0 \in \mM(A^\top y^0+c)$, $\vartheta >0$, and a nonnegative sequence $\{\eta^k\}$.
\For{$k=0,1,\ldots$}
\State Identify an extreme point $(\widehat y^k, \widehat w^k) \in {\cal E}(\mQ(y^k,w^k))$. Let $J^{k} = {\rm supp}(\widehat w^k)$ and
\[
W^k = A\big(I_n - G_{J^k}^\top (G_{J^k}G_{J^k}^\top )^{-1}G_{J^k}\big)A^\top  \in \partial^2 \varphi(\widehat y^k).
\]
\State Find an approximate solution $d^k$ to
\[
W^k d + \nabla \varphi(\widehat y^k)=0
\]
such that
\[
\| W^k d^k + \nabla \varphi(\widehat y^k)\| \le \min(\eta^k, \|\nabla \varphi(\widehat y^k)\|^\vartheta)\| \nabla \varphi(\widehat y^k)\|.
\]
\State Set
\[
y^{k+1}=\widehat y^k+d^k
\quad\mbox{and choose}\quad
w^{k+1} \in \mM(A^\top y^{k+1}+c).
\]
\EndFor
\end{algorithmic}
\end{algorithm}

To implement the correction step in Algorithm~\ref{alg:d-SN-poly}, one needs to compute an extreme point of $\mQ(y^k,w^k)$ from the current feasible pair $(y^k,w^k)\in\mQ(y^k,w^k)$. The following support-reduction procedure provides such a construction, which is a standard technique in linear programming. See, for example, the proof of \cite[Theorem 2.6]{bertsimas1997introduction}.

\begin{algorithm}[H]
\caption{Support-Reduction Procedure for Extreme Point Identification}
\label{alg:find-ex-Q-1}
\begin{algorithmic}[1]
\State Initialize $(y,w) = (\bar y, \bar w) \in \mQ(\bar y,\bar w)$ and set $J=\supp(w)$.
\While{$\begin{bmatrix}  A^\top & G_J^\top  \end{bmatrix}$ does not have full column rank}
\State Find $(\Delta y,\Delta w_J)\neq 0$ such that
\[
A^\top \Delta y- G_J^\top \Delta w_J=0.
\]
\State Find $\tau\neq 0$ such that
\[
w_J+\tau\Delta w_J\le 0
\quad\mbox{and}\quad
\supp(w_J+\tau\Delta w_J)\subsetneq J.
\]
\State Update
\[
y= y+\tau \Delta y,
\quad
w = (w_J+\tau\Delta w_J;\,w_{J^C}).
\]
\State Update $J= \supp(w)$.
\EndWhile
\State Output $(\widehat y,\widehat w)=(y,w)$.
\end{algorithmic}
\end{algorithm}

\begin{proposition}\label{prop:alg-find-ex-Q-1}
For any given $(\bar y, \bar w)$ with $\bar w\in\mM(A^\top \bar y+c)$, Algorithm~\ref{alg:find-ex-Q-1} is well defined and terminates in at most $|\supp(\bar w)|$ iterations at a point $(\widehat y,\widehat w)\in{\cal E}(\mQ(\bar y, \bar w))$, and the total computational cost is bounded by ${\cal O}\!\left(m_I n (m_E + m_I)\min\{m_E + m_I,\, n\}\right)$. 
\end{proposition}

\begin{proof}
    The well-definedness and finite termination of Algorithm~2
    follow directly from the support-reduction argument in the proof of
    \cite[Theorem~2.6]{bertsimas1997introduction}.
    Indeed, the support of $w$ decreases by at least one at each iteration,
    and hence the algorithm terminates in at most
    $|\supp(\bar w)|\le m_I$ iterations.
    By standard linear algebra, the cost of solving the linear system in line~3 is ${\cal O}\!\left(n (m_E + m_I)\min\{m_E + m_I,\, n\}\right)$, and the combined cost in line~4 and line~5 is ${\cal O}(m_I)$. Hence the total computational cost is bounded by ${\cal O}\!\left(m_I n (m_E + m_I)\min\{m_E + m_I,\, n\}\right)$. \hfill $\Box$
\end{proof}

The following example, concerning projection onto the Birkhoff polytope, shows that in structured problems the actual cost of identifying an extreme point can be substantially smaller than the worst-case upper bound obtained in Proposition~\ref{prop:alg-find-ex-Q-1}.

\begin{example}\label{ex:ndsm}
Consider the following problem of computing the nearest doubly stochastic matrix:
\begin{equation*}
	 \begin{array}{cl}
		\min\limits_{X = (x_{ij}) \in \Rbb^{m\times m}} & \displaystyle \frac{1}{2} \| X-C\|^2 \\ [8pt]
		\mbox{s.t.} & \sum_{i=1}^{m}x_{ij}=1,\quad j=1,\cdots,m,\\ [2pt]
				&   \sum_{j=1}^{m}x_{ij}=1,\quad i=1,\cdots,m,\\[2pt]
		& X\ge 0,
	\end{array}
\end{equation*}
where $C\in \Rbb^{m\times m}$ is given data. This problem can be reformulated into the structure of \eqref{prob:K-lin-proj} by vectorizing $X$ as $x\in\Rbb^{m^2}$ and $C$ as $c\in\Rbb^{m^2}$, and setting
\begin{equation*}
A=\begin{bmatrix}
\one_m^\top  & 0 & \cdots & 0\\
0 & \one_m^\top  & \cdots & 0\\
\vdots& \vdots & \ddots & 0\\
0 & 0 & \cdots & \one_m^\top \\
I'& I' & \cdots & I'
\end{bmatrix}\in \Rbb^{(2m-1)\times m^2},\quad b=\one_{2m-1}\in\Rbb^{2m-1}, \quad G = I_{m^2}, \quad \mbox{and} \quad g = 0 \in \Rbb^{m^2}.
\end{equation*}
Here, $I'=\begin{bmatrix} I_{m-1} & 0\end{bmatrix}\in \Rbb^{(m-1)\times m}$.

Given any $y \in \mathbb{R}^{2m-1}$, the special structure of $G$ allows us to write the associated multiplier explicitly as
\[
w = A^\top  y + c - \Pi_{\Rbb^{m^2}_+}(A^\top  y + c).
\]
As a consequence, the main step of Algorithm~\ref{alg:find-ex-Q-1}, namely line~3, reduces to finding $\Delta y \in \mathbb{R}^{2m-1}$ such that
\begin{equation*}
(A^\top )_{J^C}\,\Delta y = 0
\quad \mbox{with} \quad
J = \operatorname{supp}(w)
= \{ j \in [m^2] \mid (A^\top  y + c)_j < 0 \}.
\end{equation*}
We note that, up to a sign change in one vertex class,
$(A^\top)_{J^C}$ is the reduced incidence matrix of the bipartite
graph $\mathcal G_{J^C}$. Hence, a nonzero vector in 
${\rm Null}((A^\top)_{J^C})$ can be obtained from the connected components
of $\mathcal G_{J^C}$ in
$\mathcal O(m+|J^C|)=\mathcal O(m^2)$ time. Such a vector vanishes on
every edge whose endpoints belong to the same connected component.
Therefore, each support-reduction step moves at least one edge joining
two distinct components from $J$ to $J^C$, and hence reduces the
number of connected components by at least one. Since
$\mathcal G_{J^C}$ has $2m$ vertices, the algorithm terminates in at
most $2m-1$ steps. As each step costs $\mathcal O(m^2)$, the total
complexity is bounded by $\mathcal O(m^3)$, substantially improving
upon the generic bound $\mathcal O(m^8)$ in
Proposition~\ref{prop:alg-find-ex-Q-1}.
\end{example}

Before proceeding to the convergence analysis, we present a simple example together with a graphical illustration of Algorithm~\ref{alg:d-SN-poly}. This example demonstrates both the necessity and the effectiveness of the extreme-point correction step. Specifically, the Jacobian at the current point is singular, whereas after correcting to an extreme point of the projection-equivalent set, a nonsingular generalized Jacobian becomes available and the exact semismooth Newton step reaches the optimal solution in one iteration.

\begin{example}\label{ex:one-step}
Consider the problem
\[
\begin{array}{cl}
\min & \displaystyle \frac{1}{2} \| x-(0,1,0)^\top \|^2 \\ [5pt]
\mbox{s.t.} & \begin{bmatrix}1 & 1 & 0 \\ 1 & 0 & 1\end{bmatrix} x = \begin{bmatrix}1 \\ 0\end{bmatrix}, \quad x\ge 0.
\end{array}
\]
The unique primal solution is $(0,1,0)^\top $, while the dual optimal solution set is
\[
\mP^*=\{y\in \Rbb^2\mid \nabla\varphi(y)=0\}=\{(0,t)^\top \mid t\le 0\}.
\]
For any $a\in(-0.5,0)$ and $b\in(-0.5,0.5)$, let
\(
y^0=(a,-1+b)^\top ,
\)
which lies in a neighborhood of $y^*=(0,-1)^\top \in\mP^*$. One checks directly that $\nabla\varphi$ is differentiable at $y^0$, but its Jacobian is singular and given by $\begin{bmatrix}
1&0\\
0&0
\end{bmatrix}$. Hence the classical Newton method is not directly covered by its standard local theory. By contrast, our method first corrects $y^0$ to
\(
\widehat y^0=(a,0)^\top ,
\)
at which the generalized Jacobian chosen in Algorithm~\ref{alg:d-SN-poly} is 
\[
W=
\begin{bmatrix}
1&0\\
0&1
\end{bmatrix}
\in \partial^2\varphi(\widehat y^0),
\]
which is nonsingular. The resulting Newton step gives
\[
y^1=\widehat y^0-W^{-1}\nabla\varphi(\widehat y^0)=(0,0)^\top \in \mP^*.
\]
Therefore, the method reaches an optimal dual solution in a single iteration. Figure~\ref{fig:vis-alg1} visualizes the correction--Newton mechanism for the specific initial point $y^0=(-0.25,-0.75)^\top $.
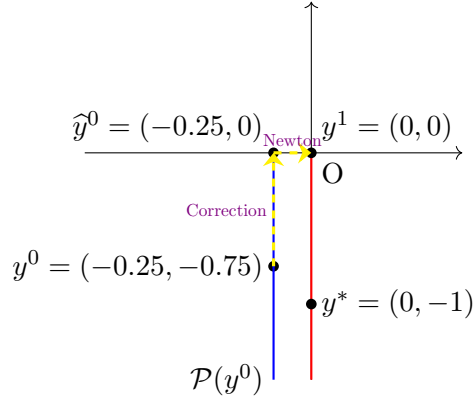
\begin{figure}[H]
\begin{center}
\begin{tikzpicture}[scale=2]
 		\draw[->] (-1.5,0) -- (1,0);
 		\draw[->] (0,-1.5) -- (0,1);

 		\node at (0,0) [below right] {O};

 		\draw[red, thick] (0,0) -- (0,-1.5);

 		\fill (0,-1) circle (1pt);
 		\node at (0,-1) [right]{$y^*=(0,-1)$};

 		\fill (-0.25,-0.75) circle (1pt);
 		\node at (-0.25,-0.75) [left] {$y^0=(-0.25,-0.75)$};

 		\draw[blue, thick] (-0.25,0) -- (-0.25,-1.5);
 		\node at (-0.25,-1.5) [left] {${\mathcal P}(y^0)$};

 		\fill (-0.25,0) circle (1pt);
 		\node at (-0.25,0) [above left] {$\widehat y^0=(-0.25,0)$};

 		\fill (0,0) circle (1pt);
 		\node at (0,0) [above right] {$y^1=(0,0)$};

 		\draw[dashed,yellow,very thick, -{Stealth[length=5pt,width=8pt]}] (-0.25,-0.75) -- node[left, scale=0.6, violet]{Correction}(-0.25,0);
 		\draw[dashed,yellow,very thick,-{Stealth[length=5pt,width=8pt]}] (-0.25,0) --node[above, scale=0.6, violet]{Newton} (0,0);
\end{tikzpicture}
\end{center}
\caption{Illustration of the correction--Newton mechanism in Algorithm~\ref{alg:d-SN-poly}.}
\label{fig:vis-alg1}
\end{figure}
\end{example}

Now, we investigate the local convergence properties of Algorithm~\ref{alg:d-SN-poly}. Motivated by Theorem~\ref{thm:nons-gen-jac}, we collect all index sets $J$ for which $\begin{bmatrix}  A^\top & G_J^\top  \end{bmatrix}$ has full column rank, and denote the collection by
\[
{\cal J} = \{ J \subseteq [m_I] \mid \begin{bmatrix}  A^\top & G_J^\top  \end{bmatrix} \mbox{ has full column rank} \}.
\]
The corresponding nonsingular Jacobian-like matrices are gathered into the finite set
\begin{equation}\label{eq:Jacobian-like-W}
{\cal W}
=
\left\{
W\in {\mathbb S}^{m_E}
\;\middle|\;
W = A\big(I_n - G_J^\top (G_JG_J^\top )^{-1}G_J\big)A^\top ,\ \ J\in {\cal J}
\right\}.
\end{equation}
We then establish the following lemma, which shows the uniform boundedness of the inverses of the ``Jacobian-like'' matrices in ${\cal W}$.

\begin{lemma} \label{lemma:unibounded-W}
    Every matrix in ${\cal W}$ is symmetric positive definite, and ${\cal W}$ is finite. Consequently, there exists a constant $C_{\cal W} > 0$ such that
    \[
        \|W^{-1}\| \le C_{\cal W} \quad \forall\, W\in {\cal W}.
    \]
\end{lemma}

\begin{proof}
The symmetric positive definiteness of each $W\in {\cal W}$ follows from its definition and Theorem~\ref{thm:nons-gen-jac}. Since there are only finitely many subsets of $\{1,\ldots,m_I\}$, the index collection ${\cal J}$ is finite, and hence so is ${\cal W}$. The uniform bound on $\|W^{-1}\|$ follows immediately. \hfill $\Box$
\end{proof}

The presence of the auxiliary extreme-point sequence $\{\widehat y^k\}$ makes the local analysis more delicate than in the classical semismooth Newton method. The key point is that the error bound introduced by the correction step is controlled by Proposition~\ref{prop:dist-hatyhatP-yP}, and therefore remains asymptotically compatible with the Newton step. We can now establish the local superlinear convergence of Algorithm~\ref{alg:d-SN-poly}.

\begin{theorem}\label{thm:convergence-alg}
	Recall the definition of $\mP^*$ from Lemma \ref{lemma:error-bound},
	and let $\{y^{\nuk}\}$ and  $\{\widehat y^{\nuk}\}$  be the infinite sequences generated by Algorithm~\ref{alg:d-SN-poly}. There exist $\overline\eta >0$ and $\varepsilon >0$ such that if $y^0$ satisfies ${\rm dist}(y^0,\mP^*) \le \varepsilon$ and $\eta^\nuk \le \overline\eta$ for all $\nuk\ge0$, it holds that
	\[
	{\rm dist}(y^{\nuk+1},\mP^*)\le \frac{1}{2} {\rm dist}(y^{\nuk},\mP^*)\quad \forall\, \nuk\ge 0.
	\]
    In fact, $\{{\rm dist}(y^k,\mP^*)\}$ converges at least superlinearly to $0$ in the sense that for all sufficiently large  $k$,
    \[
        {\rm dist}(y^{k+1},\mP^*) = {\cal O}\bigl({\rm dist}^{1+\vartheta}(y^k,\mP^*)\bigr).
    \]
	Moreover, if in addition $\eta^{\nuk_0} = 0$ for some $\nuk_0 \ge 0$, then $y^{\nuk_0+1} \in \mP^*$.
\end{theorem}

\begin{proof}
From definitions \eqref{eq:def-hatmP-star} and \eqref{eq:def-hatmP} and the finiteness of extreme points of a given polyhedron, we know that $\widehat{\mP}^*$ is finite. Then, \eqref{eq:semismooth-varphi} implies that there exists $\varepsilon_1>0$ such that, for every $\widehat y^*\in \widehat{\mP}^*$,
\[
\nabla\varphi(y)-\nabla\varphi(\widehat y^*)-W(y-\widehat y^*)=0
\quad
\forall\, W\in \partial^2\varphi(y)\quad \mbox{and} \quad y\in {\mathbb B}(\widehat y^*,\varepsilon_1).
\]
Now let $y^0$ be sufficiently close to $\mP^*$. By Proposition~\ref{prop:dist-hatyhatP-yP}, there exist constants $C'>0$ and $\varepsilon>0$ such that if ${\rm dist}(y^0,\mP^*)\le\varepsilon$, then
\[
{\rm dist}(\widehat y^0,\widehat{\mP}^*)
\le
C'{\rm dist}(y^0,\mP^*)
\le
\varepsilon_1.
\]
Choose $\overline y^0\in \Pi_{\widehat{\mP}^*}(\widehat y^0)$. Since $W^0\in{\cal W}$ and $W^0$ is nonsingular, the updating rule in Algorithm~\ref{alg:d-SN-poly} gives
\begin{align*}
{\rm dist}(y^1,\mP^*)
&\le \|\widehat y^0+d^0-\overline y^0\| \\
&= \left\|
-(W^0)^{-1}
\Big(
\nabla\varphi(\widehat y^0)-\nabla\varphi(\overline y^0)-W^0(\widehat y^0-\overline y^0)
-
\big(W^0d^0+\nabla\varphi(\widehat y^0)\big)
\Big)
\right\| \\
&\le \|(W^0)^{-1}\|\, \|W^0d^0+\nabla\varphi(\widehat y^0)\| \\
&\le C_{\cal W}\,\overline\eta\,\|\nabla\varphi(\widehat y^0)\|.
\end{align*}
Using the global Lipschitz continuity of $\nabla\varphi$ and Proposition~\ref{prop:dist-hatyhatP-yP}, we further obtain
\[
\|\nabla\varphi(\widehat y^0)\|
\le
\|A\|^2\|\widehat y^0-\overline y^0\|
\le
\|A\|^2 C' {\rm dist}(y^0,\mP^*).
\]
Hence
\[
{\rm dist}(y^1,\mP^*)
\le
C_{\cal W}\overline\eta \|A\|^2 C' {\rm dist}(y^0,\mP^*).
\]
Choosing $\overline\eta$ sufficiently small, say
\[
\overline\eta \le \frac{1}{2C_{\cal W}C'\|A\|^2},
\]
yields
\[
{\rm dist}(y^1,\mP^*)\le \frac12 {\rm dist}(y^0,\mP^*).
\]
Repeating the same argument inductively gives
\[
{\rm dist}(y^{k+1},\mP^*)\le \frac12 {\rm dist}(y^k,\mP^*)
\quad \forall\,k\ge0.
\]
This proves the linear contraction.

For the superlinear estimate, let $\overline y^k\in\Pi_{\widehat{\mP}^*}(\widehat y^k)$. By \eqref{eq:semismooth-varphi},
\[
\nabla\varphi(\widehat y^k)-\nabla\varphi(\overline y^k)-W^k(\widehat y^k-\overline y^k)=0
\]
for all sufficiently large $k$. Hence,
\begin{align*}
{\rm dist}(y^{k+1},\mP^*)
&\le \|\widehat y^k+d^k-\overline y^k\| \\
&=
\left\|
-(W^k)^{-1}
\Big(
\nabla\varphi(\widehat y^k)-\nabla\varphi(\overline y^k)-W^k(\widehat y^k-\overline y^k)
-
\big(W^kd^k+\nabla\varphi(\widehat y^k)\big)
\Big)
\right\| \\
&\le \|(W^k)^{-1}\|\, \|W^kd^k+\nabla\varphi(\widehat y^k)\| \\
&\le C_{\cal W}\,\|\nabla\varphi(\widehat y^k)\|^{1+\vartheta}.
\end{align*}
By the global Lipschitz continuity of $\nabla\varphi$ and Proposition~\ref{prop:dist-hatyhatP-yP},
\[
\|\nabla\varphi(\widehat y^k)\|
=
{\cal O}\bigl({\rm dist}(\widehat y^k,\widehat{\mP}^*)\bigr)
=
{\cal O}\bigl({\rm dist}(y^k,\mP^*)\bigr).
\]
Therefore,
\[
{\rm dist}(y^{k+1},\mP^*)
=
{\cal O}\bigl({\rm dist}^{\,1+\vartheta}(y^k,\mP^*)\bigr).
\]

Finally, if $\eta^{k_0}=0$ for some $k_0\ge0$, then the Newton residual vanishes:
\[
W^{k_0}d^{k_0}+\nabla\varphi(\widehat y^{k_0})=0.
\]
Using again the exact semismooth expansion at the corresponding $\overline y^{k_0}\in\Pi_{\widehat{\mP}^*}(\widehat y^{k_0})$, we obtain
\[
\widehat y^{k_0}+d^{k_0}=\overline y^{k_0}\in \widehat{\mP}^*\subseteq \mP^*,
\]
and hence $y^{k_0+1}\in\mP^*$. \hfill $\Box$
\end{proof}

\subsection{Globalization of Algorithm~\ref{alg:d-SN-poly}}\label{section:Global}

The local analysis above shows that, once the semismooth Newton linearization is performed at an extreme point, the resulting step exhibits the expected fast local convergence behavior. We now turn to globalization of Algorithm~\ref{alg:d-SN-poly}. The remaining task is to handle the potential increase in the objective value induced by the correction step and the lack of global descent of the Newton step, without compromising the eventual local superlinear rate.

An essential ingredient for addressing the former issue is the following monotone extreme-point identification subroutine, Algorithm~\ref{alg:find-ex-Q-nonicreasing}.  For any given pair $(\bar y, \bar w)$, it computes an extreme point $(\widehat y,\widehat w)\in{\cal E}(\mQ(\bar y, \bar w))$ such that
\[
\varphi(\widehat y)\le \varphi(\bar y).
\]
This subroutine can be viewed as a monotone extension of Algorithm~\ref{alg:find-ex-Q-1}. Besides identifying an extreme point of the lifted projection-equivalent set, it also ensures that the correction step does not increase the dual objective value. 

\begin{algorithm}[H]
\caption{Identifying an extreme point $(\widehat y, \widehat w) \in {\cal E}(\mQ(\bar y,\bar \ww))$ with $\varphi(\widehat y) \le \varphi(\bar y)$}
\label{alg:find-ex-Q-nonicreasing}
\begin{algorithmic}[1]
\State Initialize $(y,\ww) = (\bar y, \bar \ww)\in {\cal Q}(\bar y, \bar \ww)$ and set $J=\supp(\ww)$.
\While{$\begin{bmatrix} A^\top & G_J^\top \end{bmatrix}$ does not have full column rank}
\State Find $(\widetilde\dyy,\widetilde\dww_J)\neq 0$ such that
\[
A^\top\widetilde\dyy- G_J^\top\widetilde\dww_J=0.
\]
\If{($\widetilde\dyy^\top \nabla \varphi(y) = 0$ {\bf and } $\widetilde\dww_J\le 0$) {\bf or} $\widetilde\dyy^\top \nabla \varphi(y) > 0$}
    \State \[
(\Delta y, \Delta \ww_J) = -(\widetilde\dyy, \widetilde\dww_J);
\]
\Else
    \State \[
(\Delta y, \Delta \ww_J) = (\widetilde\dyy, \widetilde\dww_J).
\]
\EndIf
\State Find $\tau > 0$ such that
\[
\varphi(y+\tau\dyy)\le \varphi(y),\qquad
\ww_J+\tau\dww_J\le 0,\qquad
\supp(\ww_J+\tau\dww_J)\subsetneq J.
\]
\State Update $y= y+\tau \dyy$ and $\ww= (\ww_J+\tau\dww_J;\ww_{J^C})$.
\State Update $J=\supp(\ww)$.
\EndWhile
\State Output $(\widehat y, \widehat w) = (y,\ww)$.
\end{algorithmic}
\end{algorithm}

The following proposition establishes that Algorithm~\ref{alg:find-ex-Q-nonicreasing} is well defined and gives a polynomial bound on its computational cost.
It reveals the fact that obtaining an extreme point of the projection-equivalent set with nonincreasing dual objective value can be achieved by solving a linear programming problem. Instead of calling a general-purpose linear programming solver, Algorithm~\ref{alg:find-ex-Q-nonicreasing} exploits 
a monotone feasible pivoting procedure used in the proof of \cite[Theorem 2.8]{bertsimas1997introduction} to identify a suitable extreme point by solving a sequence of homogeneous linear systems. 

\begin{proposition}\label{prop:alg-find-ex-Q-monotone}
For any $(\bar y,\bar \ww)$ with $\bar \ww\in\mM(A^\top \bar y+c)$, the minimization problem
\[
\min\left\{\varphi(y') \;\middle|\; (y',\ww')\in \mQ(\bar y,\bar \ww)\right\}
\]
is a linear programming problem with a nonempty optimal solution set. Moreover, Algorithm~\ref{alg:find-ex-Q-nonicreasing} is well defined, terminates in at most $|\supp(\bar \ww)|$ iterations at a point $(\widehat y,\widehat \ww)\in{\cal E}(\mQ(\bar y,\bar \ww))$ satisfying
\[
\varphi(\widehat y)\le \varphi(\bar y),
\]
and the total computational cost is bounded by
$
{\cal O}\!\left(m_I n(m_E + m_I)\min\{m_E + m_I,\, n\}\right).
$
\end{proposition}

\begin{proof}
	Let
	\[
	x:=\Pi_K(A^\top \bar y+c).
	\]
	By Proposition~\ref{prop:mQ-property}, for every
	$(y',\ww')\in\mQ(\bar y,\bar\ww)$, one has
	\[
	\Pi_K(A^\top y'+c)
	=
	\Pi_K(A^\top \bar y+c)
	=
	x.
	\]
	Hence, it follows from \eqref{eq:varphiy} that
	\[
	\begin{aligned}
		\varphi(y')
		&=
		-\langle b,y'\rangle
		+\frac12\|A^\top y'+c\|^2
		-\frac12\|A^\top y'+c-x\|^2  \\
		&=
		\langle y',Ax-b\rangle
		+\langle c,x\rangle
		-\frac12\|x\|^2.
	\end{aligned}
	\]
	Since
	\[
	Ax-b
	=
	A\Pi_K(A^\top\bar y+c)-b
	=
	\nabla\varphi(\bar y),
	\]
	the restriction of $\varphi$ to $\mQ(\bar y,\bar\ww)$ is affine:
	\[
	\varphi(y')
	=
	\langle\nabla\varphi(\bar y),y'\rangle
	+\langle c,x\rangle
	-\frac12\|x\|^2
	\qquad
	\forall\,(y',\ww')\in\mQ(\bar y,\bar\ww).
	\]
	Consequently,
	\[
	\min\left\{
	\varphi(y')
	\;\middle|\;
	(y',\ww')\in\mQ(\bar y,\bar\ww)
	\right\}
	\]
	is a linear programming problem. Its optimal solution set is nonempty
	because $\mQ(\bar y,\bar\ww)$ is a nonempty polyhedron and the
	objective is bounded below there by the global lower boundedness of
	$\varphi$.
	
	We next analyze Algorithm~\ref{alg:find-ex-Q-nonicreasing}.
	Let $(y^s,\ww^s)$ denote the pair generated after $s$ iterations of
	the while-loop, with
	\[
	(y^0,\ww^0)=(\bar y,\bar\ww),
	\qquad
	J_s:=\supp(\ww^s).
	\]
	We prove by induction that
	\[
	(y^s,\ww^s)\in\mQ(\bar y,\bar\ww)
	\quad\text{and}\quad
	\varphi(y^s)\le \varphi(y^{s-1})
	\]
	for every iteration $s\ge 1$.
	The assertion is immediate at initialization, since
	\(
	(y^0,\ww^0)=(\bar y,\bar\ww)
	\in\mQ(\bar y,\bar\ww).
	\)
	Now suppose, as the induction hypothesis, that
	\[
	(y^s,\ww^s)\in\mQ(\bar y,\bar\ww).
	\]
	By Proposition~\ref{prop:mQ-property}, we  have
	\(
	\mQ(y^s,\ww^s)=\mQ(\bar y,\bar\ww)\) and \(	\Pi_K(A^\top y^s+c)
	=
	\Pi_K(A^\top\bar y+c).
	\)
	Thus,
	\begin{equation}\label{eq:proof-varphi-ys}
	\nabla\varphi(y^s)=\nabla\varphi(\bar y).
	\end{equation}
	Now, suppose that
	\(
	\begin{bmatrix}
		A^\top & G_{J_s}^\top
	\end{bmatrix}
	\)
	does not have full column rank. Then there exists a nonzero pair
	$(\widetilde\Delta y^s,\widetilde\Delta\ww^s_{J_s})$ such that
	\[
	A^\top\widetilde\Delta y^s
	-
	G_{J_s}^\top\widetilde\Delta\ww^s_{J_s}
	=0.
	\]
	Since $A$ has full row rank, necessarily
	$\widetilde\Delta\ww^s_{J_s}\neq0$.
	Lines~4--8 of Algorithm~\ref{alg:find-ex-Q-nonicreasing}
	select $(\widetilde\Delta y^s,\widetilde\Delta\ww^s_{J_s})$ or its negative and obtain
	$(\Delta y^s,\Delta\ww^s_{J_s})$, such that
	\begin{equation}\label{eq:proof-decent-varphiys}
	A^\top\Delta y^s-G_{J_s}^\top\Delta\ww^s_{J_s}=0,
	\qquad
	\langle\nabla\varphi(y^s),\Delta y^s\rangle\le0.
	\end{equation}
	We next show that $\Delta\ww^s_{J_s}$ has at least one positive
	component. If
	\(
	\langle\nabla\varphi(y^s),\Delta y^s\rangle=0,
	\)
	this follows directly from the sign-selection rule in
	Lines~4--8. If instead
	\(
	\langle\nabla\varphi(y^s),\Delta y^s\rangle<0
	\)
	and $\Delta\ww^s_{J_s}\le0$, extend
	$\Delta\ww^s_{J_s}$ to $\Delta\ww^s\in\Rbb^{m_I}$ by setting
	$\Delta\ww^s_{J_s^C}=0$. Then
	\[
	(y^s+t\Delta y^s,\ww^s+t\Delta\ww^s)
	\in\mQ(y^s,\ww^s)
	=\mQ(\bar y,\bar\ww)
	\qquad \forall\,t\ge0.
	\]
	Since $\varphi$ is affine on $\mQ(\bar y,\bar\ww)$ and
	$\nabla\varphi(y^s)=\nabla\varphi(\bar y)$ by \eqref{eq:proof-varphi-ys}, it follows that
	\[
	\varphi(y^s+t\Delta y^s)
	= \varphi(y^s) +
	t\langle\nabla\varphi(y^s),\Delta y^s\rangle
	\to -\infty,
	\]
	contradicting the boundedness from below of $\varphi$.
	Thus, $\Delta\ww^s_{J_s}$ has a positive component.
	
	Therefore, the stepsize
	\[
	\tau_s
	:=
	\min\left\{
	-\frac{\ww^s_i}{(\Delta\ww^s_{J_s})_i}
	\;\middle|\;
	i\in J_s,\;
	(\Delta\ww^s_{J_s})_i>0
	\right\}
	\]
	is well defined and positive. Define
	\[
	y^{s+1}:=y^s+\tau_s\Delta y^s,
	\qquad
	\ww^{s+1}:=\ww^s+\tau_s\Delta\ww^s,
	\]
	where $\Delta\ww^s_{J_s^C}=0$. By construction,
	\[
	\ww^{s+1}\le0,
	\qquad
	\supp(\ww^{s+1})\subsetneq\supp(\ww^s),\qquad \mbox{and} \qquad A^\top(y^{s+1}-y^s)
	=
	G^\top(\ww^{s+1}-\ww^s). 
	\]
	Hence,
	\[
	(y^{s+1},\ww^{s+1})
	\in\mQ(y^s,\ww^s)
	=
	\mQ(\bar y,\bar\ww).
	\]
	Finally, using the affine representation of $\varphi$ on
	$\mQ(\bar y,\bar\ww)$, together with \eqref{eq:proof-varphi-ys} and \eqref{eq:proof-decent-varphiys}, we obtain
	\[
	\begin{aligned}
		\varphi(y^{s+1})-\varphi(y^s)
		=
		\tau_s
		\langle\nabla\varphi(\bar y),\Delta y^s\rangle =
		\tau_s
		\langle\nabla\varphi(y^s),\Delta y^s\rangle
		\le0.
	\end{aligned}
	\]
	Thus, by induction, all iterates remain in
	$\mQ(\bar y,\bar\ww)$, the objective values are nonincreasing, and
	the support of the multiplier decreases strictly at every iteration.
	
	The algorithm therefore terminates after at most
	$|\supp(\bar\ww)|\le m_I$ iterations. At termination, the final pair
	$(\widehat y,\widehat\ww) \in {\cal Q}(\bar y, \bar w)$ satisfies
	\[
	\begin{bmatrix}
		A^\top&G_{\widehat J}^\top
	\end{bmatrix}
	\quad\text{has full column rank},
	\qquad
	\widehat J:=\supp(\widehat\ww), \qquad \mbox{and} \qquad \varphi(\widehat y)\le\varphi(\bar y).
	\]
	By Theorem~\ref{thm:nons-gen-jac}, we also have
	$(\widehat y,\widehat\ww) \in \mE(\mQ(\bar y,\bar w))$. 
	Finally, the main computational cost at each iteration is finding a
	nonzero solution of
	\[
	A^\top\Delta y-G_J^\top\Delta\ww_J=0,
	\]
	which requires at most
	\(
	{\cal O}\!\left(
	n(m_E+m_I)\min\{m_E+m_I,n\}
	\right)
	\)
	operations.  Hence, the total
	computational cost is bounded by
	\[
	{\cal O}\!\left(
	m_I n(m_E+m_I)\min\{m_E+m_I,n\}
	\right).
	\]
	This completes the proof. \hfill $\Box$
\end{proof}

We are now ready to present a globalized semismooth Newton framework for solving problem~\eqref{prob:dual-K-lin-proj}. At each iteration, Algorithm~\ref{alg:find-ex-Q-nonicreasing} is first invoked to compute an extreme point
\[
(\widehat y^k,\widehat w^k)\in{\cal E}(\mQ(y^k,w^k))
\]
such that the dual function value does not increase. An inexact semismooth Newton step is then performed at the corrected point $(\widehat y^k,\widehat w^k)$, followed by a Wolfe line search.

\begin{algorithm}[H]
\caption{A globalized inexact semismooth Newton method with extreme-point correction for solving \eqref{prob:dual-K-lin-proj}}
\label{alg:global}
\begin{algorithmic}[1]
\State{Initialize $y^0$, $\ww^0 \in \mM(A^\top y^0+c)$, $\gamma_1\in (0,\frac{1}{2})$, $\gamma_2\in (\gamma_1,1)$, $\vartheta \in (0,1]$, and $\zeta \in [0,1)$. 
} 
\For{$k = 0,1,\ldots$}
\State{Use Algorithm \ref{alg:find-ex-Q-nonicreasing}  to find an extreme point $(\widehat y^{\nuk},  \widehat\ww^{\nuk}) \in {\cal E}(\mQ(y^{\nuk}, \ww^{\nuk}))$ such that $\varphi(\widehat y^{\nuk})\le \varphi( y^{\nuk})$.}
\State{Let $J^{\nuk} = {\rm supp}(\widehat\ww^\nuk)$ and
\[
W^\nuk = A(I_n - G_{J^\nuk}^\top (G_{J^\nuk}G_{J^\nuk}^\top )^{-1}G_{J^\nuk})A^\top .
\]
If $\zeta > 0$, apply the practical conjugate gradient algorithm \cite{golub2013matrix,zhao2010newton} to find an approximate solution $d^\nuk$ to
	\[W^\nuk d + \nabla \varphi( \widehat y^\nuk)=0 \quad \mbox{such that} \quad
\| W^\nuk d^\nuk + \nabla \varphi(\widehat y^\nuk)\| \le \min\big(\zeta, \|\nabla \varphi(\widehat y^\nuk)\|^\vartheta\big)\| \nabla \varphi(\widehat y^\nuk)\|;
\]
otherwise if $\zeta = 0$, compute $d^k = -(W^k)^{-1}\nabla \varphi(\widehat y^k).$
}
\State{Find a stepsize $\alpha^k > 0$ (with $\alpha^k = 1$ tried first) that satisfies the Wolfe conditions:
\begin{align}
    \varphi(\widehat y^{k} + \alpha^k d^k) &\le \varphi(\widehat y^{k}) + \gamma_1 \alpha^k \langle \nabla \varphi(\widehat y^k), d^k \rangle, \label{eq:Wolfe-1} \\
    \langle \nabla \varphi(\widehat y^{k} + \alpha^k d^k), d^k \rangle &\ge \gamma_2 \langle \nabla \varphi(\widehat y^k), d^k \rangle. \label{eq:Wolfe-2}
\end{align}}
\State{Let $y^{\nuk+1}=\widehat y^\nuk+\alpha^\nuk d^\nuk$ and find $\ww^{\nuk + 1} \in \mM(A^\top  y^{\nuk+1}+c)$.}
\EndFor
\end{algorithmic}
\end{algorithm}

To show that Algorithm~\ref{alg:global} is well defined, we first establish the following lemma, which shows that, at each iteration \(k\), the matrix \(W^k\) is symmetric positive definite and the direction \(d^k\) is a descent direction.
To facilitate the analysis, we first introduce two spectral constants associated with the set
\({\cal W}\) of ``Jacobian-like'' matrices defined in \eqref{eq:Jacobian-like-W}:
\begin{equation}\label{eq:constant}
    c_1 := \max_{W\in{\cal W}} \lambda_{\max}(W) \quad \mbox{and} \quad 
    c_2 := \min_{W\in{\cal W}} \lambda_{\min}(W).
\end{equation}
Note that Lemma~\ref{lemma:unibounded-W} guarantees that both constants are finite and strictly positive; in particular,
$c_1 \ge c_2 > 0.$

\begin{lemma}\label{lemma:bound-alg-dk}
    In Algorithm~\ref{alg:global}, the matrix $W^k$ is symmetric positive definite, and whenever $\nabla \varphi(\widehat y^k) \neq 0$, the direction $d^k$ satisfies
    \begin{equation}\label{eq:decentdk}
        -\frac{1}{c_2} \le \frac{\langle \nabla \varphi(\widehat y^k), d^k \rangle}{\|\nabla \varphi(\widehat y^k)\|^2} \le -\frac{1}{c_1},
    \end{equation}
    and
    \begin{equation}\label{eq:dk-nablavarphi-bound}
        \frac{1 - \zeta}{c_1} \le \frac{\|d^k\|}{\|\nabla \varphi(\widehat y^k)\|} \le \frac{1 + \zeta}{c_2}.
    \end{equation}
    Moreover, it holds that
    \begin{equation}\label{eq:mnablaykdk}
    \frac{c_2^2}{(1+\zeta)^2c_1} \|d^k\|^2\le -\langle \nabla \varphi(\widehat y^k), d^k \rangle\le \frac{c_1^2}{(1-\zeta)^2c_2} \|d^k\|^2.
\end{equation}
\end{lemma}

\begin{proof}
The positive definiteness of $W^k$ follows from $(\widehat y^k, \widehat w^k) \in {\cal E}(\mQ(y^k, w^k))$ and Theorem~\ref{thm:nons-gen-jac}. Suppose that $\nabla \varphi(\widehat y^k) \neq 0$. If $\zeta >0$, then by \cite[Proposition~3.3]{zhao2010newton}, the practical conjugate gradient algorithm generates a direction $d^k$ satisfying
\[
-\frac{1}{\lambda_{\min}(W^k)} \le \frac{\langle \nabla \varphi(\widehat y^k), d^k \rangle}{\|\nabla \varphi(\widehat y^k)\|^2} \le -\frac{1}{\lambda_{\max}(W^k)}.
\]
The same bound holds trivially in the exact case $\zeta = 0$. Hence \eqref{eq:decentdk} follows from the definition of $c_1$ and $c_2$.

Since
\[
\|W^k d^k + \nabla \varphi(\widehat y^k)\| \le \zeta \|\nabla \varphi(\widehat y^k)\|,
\]
we have
\[
(1 - \zeta)\|\nabla \varphi(\widehat y^k)\| \le \|W^k d^k\| \le (1 + \zeta) \|\nabla \varphi(\widehat y^k)\|.
\]
Combining this with \eqref{eq:constant} yields
\[
c_1 \|d^k\| \ge (1 -\zeta)\|\nabla \varphi(\widehat y^k)\|\quad \mbox{and}
\quad
(1 + \zeta) \|\nabla \varphi(\widehat y^k)\| \ge c_2 \|d^k\|.
\]
Thus \eqref{eq:dk-nablavarphi-bound} follows. Finally, \eqref{eq:mnablaykdk} is an immediate consequence of \eqref{eq:decentdk} and \eqref{eq:dk-nablavarphi-bound}.
\hfill $\Box$
\end{proof}

The bound~\eqref{eq:decentdk}, together with the boundedness of $\varphi$ from below, guarantees the existence of a stepsize satisfying the Wolfe conditions in Algorithm~\ref{alg:global}; see \cite[Lemma~3.1]{NocedalWright2006}. Hence, Algorithm~\ref{alg:global} is well defined. 
To avoid trivial cases, we analyze the convergence properties of Algorithm~\ref{alg:global} under the assumption that $\nabla\varphi(\widehat y^k) \neq 0$ for all $k\ge0$. As shown in the following proof, this assumption is equivalent to $\nabla\varphi(y^k) \neq 0$ for all $k\ge0$.

\begin{proposition}\label{prop:global-convergence}
    Let $ \{y^{\nuk}\} $ and $ \{\widehat{y}^{\nuk}\} $ be the infinite sequences generated by Algorithm \ref{alg:global}. Then, it holds that
    \begin{equation}\label{eq:dist-hatyk-P}
    	\lim_{\nuk\to \infty} \nabla \varphi(\widehat y^\nuk) = 0, \quad \lim_{\nuk\to \infty} {\rm dist}(\widehat y^\nuk, {\cal P}^*) = 0,
    \end{equation}
    and
    \begin{equation}\label{eq:dist-yk-P}
    	\lim_{\nuk\to \infty} \nabla \varphi( y^\nuk) = 0, \quad \lim_{\nuk\to \infty} {\rm dist}( y^\nuk, {\cal P}^*) = 0.
    \end{equation}
Moreover,
\begin{equation}\label{eq:distykhatP}
\lim_{\nuk\to \infty} {\rm dist}(\widehat y^\nuk, \widehat{\cal P}^*) = 0.
\end{equation}
Hence, $\{\widehat y^k\}$ is bounded and any accumulation point $\widehat y^*$ of $\{\widehat y^\nuk\}$ satisfies $\widehat y^*\in \widehat{\cal P}^*$, and
\begin{equation}\label{eq:converge-func-value}
    \lim_{k\to \infty} \varphi(y^k) = \lim_{k\to \infty} \varphi(\widehat y^k) = \varphi^*,
\end{equation}
where $\varphi^*$ is the optimal value of problem \eqref{prob:dual-K-lin-proj} defined in Lemma \ref{lemma:error-bound}.
\end{proposition}

\begin{proof}
The Wolfe condition~\eqref{eq:Wolfe-2} and the Lipschitz continuity of $\nabla \varphi$ yield
\begin{equation*}
    \alpha^k L_{\varphi} \|d^k\|^2 \ge \langle \nabla \varphi(\widehat y^k + \alpha^k d^k) - \nabla \varphi(\widehat y^k), d^k \rangle
    \ge (\gamma_2 - 1)\,\langle \nabla \varphi(\widehat y^k), d^k \rangle
    \quad \forall\, k \ge 0,
\end{equation*}
where $L_{\varphi}$ denotes the Lipschitz constant of $\nabla \varphi$,
satisfying $0 < L_{\varphi} \le \|A\|^2$.
Thus, we have
\begin{equation*}
    \alpha^\nuk\ge \frac{(\gamma_2-1) \la\nabla\varphi(\widehat y^\nuk), d^k\ra}{L_{\varphi} \|d^k\|^2} \quad \forall \, k\ge 0.
\end{equation*}
Substituting this into the Wolfe condition \eqref{eq:Wolfe-1}, we obtain
\begin{equation}\label{eq:value_convergence}
    \varphi(\widehat y^{\nuk}) - \frac{\gamma_1 (1-\gamma_2) \la\nabla\varphi(\widehat y^\nuk), d^k\ra^2}{L_{\varphi} \|d^k\|^2} \ge \varphi(\widehat y^{\nuk}+\alpha^\nuk d^\nuk)=\varphi(y^{\nuk + 1})\ge \varphi(\widehat y^{\nuk + 1}),
\end{equation}
which, together with the fact that $\varphi$ is bounded from below, implies
\begin{equation}\label{eq:converge-value}
    \varphi(\widehat y^\nuk)\downarrow \varphi^\infty \in \Rbb \mbox{ as } k\to \infty  \quad\quad\mbox{and}\quad\quad  \sum_{\nuk=1}^{\infty} \frac{\la\nabla\varphi(\widehat y^\nuk), d^k\ra^2}{\|d^k\|^2}<+\infty.
\end{equation}
Hence, we have
\begin{equation}\label{eq:Wolfe-lim-0}
    \frac{\la\nabla\varphi(\widehat y^\nuk), d^k\ra}{\|d^k\|}\to 0 \mbox{ as } k\to \infty.
\end{equation}
Then, it follows from \eqref{eq:Wolfe-lim-0}, \eqref{eq:decentdk} and \eqref{eq:dk-nablavarphi-bound}  that
\begin{equation*}
    \|\nabla\varphi(\widehat y^\nuk)\|=\frac{\la\nabla\varphi(\widehat y^\nuk), d^k\ra}{\|d^k\|}\times \frac{\|\nabla \varphi(\widehat y^\nuk)\|^2}{\langle \nabla \varphi(\widehat y^\nuk) , d^\nuk \rangle }\times\frac{\|d^\nuk\|}{\|\nabla\varphi(\widehat y^\nuk)\|}\to 0 \mbox{ as } k\to \infty.
\end{equation*}
Thus, we conclude that
\begin{equation*}
    \lim_{\nuk\to \infty} \nabla \varphi(\widehat y^\nuk) = 0,
\end{equation*}
which, together with the error bound \eqref{eq:error-bound} in Lemma \ref{lemma:error-bound}, implies
\[
    \lim_{\nuk\to \infty} \dist(\widehat y^\nuk,\mP^*)= 0.
\]
Since $(\widehat y^{\nuk},  \widehat\ww^{\nuk}) \in {\cal E}(\mQ(y^{\nuk}, \ww^{\nuk}))$ for all $k\ge 0$, we know from Proposition \ref{prop:mQ-property} that
\[\nabla\varphi(y^k) = A\Pi_K(A^\top  y^k + c) - b = A\Pi_K(A^\top  \widehat y^k + c) - b =\nabla\varphi(\widehat y^k).\]
Thus, \eqref{eq:error-bound} further implies $\dist(y^k, \mP^*)\to 0$ as $k\to \infty$. Then, by Proposition~\ref{prop:dist-hatyhatP-yP}, we have that
there exists a constant $C'>0$ such that
\[
\dist(\widehat y^k,\widehat{\mP}^*)
\le
C'\dist(y^k,\mP^*)
\quad \mbox{for all sufficiently large $k$}.
\]
Consequently,
\(
\lim_{k\to\infty}
\dist(\widehat y^k,\widehat{\mP}^*)=0.
\)
Since $\widehat{\mP}^*$ is finite, the sequence $\{\widehat y^k\}$ is bounded,
and every accumulation point $\widehat y^*$ belongs to $\widehat{\mP}^*$. Then, the continuity of $\varphi$ and \eqref{eq:converge-value} yield
\(
\varphi(\widehat y^k)\to\varphi^*.
\)
The inequalities in \eqref{eq:value_convergence} then imply that
\(
\varphi(y^k)\to\varphi^* 
\) as $k\to\infty.$
$\hfill \Box$
\end{proof}

The preceding global convergence result highlights an important distinction between our approach and that of~\cite{hu2024semismooth}. The convergence analysis in~\cite{hu2024semismooth} is inherently local. A direct globalization of the algorithm proposed therein may suffer from cycling, for which randomization is used as an empirical remedy \cite[page~746]{hu2024semismooth}. In contrast, Proposition~\ref{prop:alg-find-ex-Q-monotone} shows that our primal--dual extreme-point identification procedure admits a monotone simplex-type descent mechanism over the lifted set $\mathcal{Q}(y,w)$. Combined with the Wolfe line search in Algorithm~\ref{alg:global}, this mechanism yields a provably globally convergent framework while retaining the local superlinear convergence rate established in the following proposition.

\begin{proposition}\label{prop:converg-rate-global}
	Let $ \{y^{\nuk}\} $ and $ \{\widehat{y}^{\nuk}\} $ be the infinite sequences generated by Algorithm \ref{alg:global}. Then, the line search step in Algorithm~\ref{alg:global} satisfies
    \[
        \alpha^k \equiv 1 \quad \mbox{ and } \quad {y}^{k+1} = \widehat y^k + d^k \quad \mbox{for all sufficiently large $k$}.
    \]
    Thus, $\{y^k\}$ is bounded and any accumulation point of $\{y^k\}$ lies in ${\widehat \mP}^*$.
	Moreover, it holds that for all sufficiently large $k$,
	\[
\dist(\widehat y^{k+1}, \widehat{\mP}^*)
=
{\cal O}\bigl(\dist^{\,1+\vartheta}(\widehat y^k, \widehat{\mP}^*)\bigr),
\quad
\dist(\widehat y^{k+1}, \mP^*)
=
{\cal O}\bigl(\dist^{\,1+\vartheta}(\widehat y^k, \mP^*)\bigr),
\]
and
\[
\dist(y^{k+1}, \widehat{\mP}^*)
=
{\cal O}\bigl(\dist^{\,1+\vartheta}(y^k, \widehat{\mP}^*)\bigr), \quad \dist(y^{k+1}, {\mP}^*)
=
{\cal O}\bigl(\dist^{\,1+\vartheta}(y^k, {\mP}^*)\bigr).
\]
\end{proposition}

\begin{proof}
	By \eqref{eq:dk-nablavarphi-bound}, we have
    \[
        \frac{(1 - \zeta)}{c_1} \|\nabla \varphi(\widehat y^k)\| \le \|d^k\| \le \frac{(1+\zeta)}{c_2} \|\nabla \varphi(\widehat y^k)\|,
    \]
    where $c_1$ and $c_2$ are defined in \eqref{eq:constant}. Since $\nabla \varphi(\widehat y^k) \to 0$ in \eqref{eq:dist-hatyk-P}, it follows that $d^k \to 0$ as $k \to \infty$.

For each $k$, choose
\(
\widehat y_k^*
\in
\Pi_{\widehat{\mathcal P}^*}(\widehat y^k).
\)
Then
\(
\|\widehat y^k-\widehat y_k^*\|
=
\dist(\widehat y^k,\widehat{\mathcal P}^*)
\to 0.
\)
Since $\widehat{\mathcal P}^*$ is finite, the neighborhoods in
\eqref{eq:semismooth-varphi} can be chosen uniformly over all
$\widehat y^*\in\widehat{\mathcal P}^*$. Hence, there exists $N_0$
such that, for all $k\ge N_0$,
\begin{equation}\label{eq:app_kN0}
\nabla\varphi(\widehat y^k)
-\nabla\varphi(\widehat y_k^*)
-W^k(\widehat y^k-\widehat y_k^*)
=0.
\end{equation}
Therefore,
we see that for all $\nuk \ge N_0$
\begin{equation*}
\begin{aligned}
	\|\widehat y^\nuk+d^\nuk-\widehat y^*_k \|
	={}& \|-(W^\nuk)^{-1}\left(\nabla \varphi( \widehat y^\nuk) - \nabla \varphi(\widehat y^*_k)-W^\nuk( \widehat y^\nuk - \widehat y^*_k) -(W^\nuk d^\nuk + \nabla \varphi( \widehat y^\nuk)) \right) \| \\
	\le{} & \|(W^\nuk)^{-1}\|\|\nabla \varphi(\widehat y^\nuk)\|^{1+\vartheta} \\
	\le{}&  \frac{c_1^{1+\vartheta}}{c_2(1 - \zeta)^{1+\vartheta}}\big \|d^\nuk\|^{1+\vartheta},
\end{aligned}
\end{equation*}
i.e., for all $\nuk \ge N_0$,
\begin{equation}\label{eq:ykdkmyhatk}
\widehat y^\nuk + d^\nuk - \widehat y^*_k = {\cal O}(\|\nabla \varphi(\widehat y^\nuk)\|^{1+\vartheta}) = {\cal O}(\|d^\nuk\|^{1+\vartheta}).
\end{equation}
Next, invoking \cite[Proposition~7]{li2018efficiently}, \cite[Theorem~2.1]{pang1995globally}, and the fact that $\widehat{\cal P}^*$ is finite, there exists an integer $N_1$ such that, for all $k \ge N_1$, the following expansions hold:
\begin{equation*}
	\varphi(\widehat y^\nuk + d^\nuk) - \varphi(\widehat y_\nuk^*) - \langle \nabla \varphi(\widehat y_\nuk^*) , \widehat y^\nuk + d^\nuk - \widehat y_\nuk^* \rangle - \frac{1}{2} (\widehat y^\nuk + d^\nuk - \widehat y_\nuk^*)^\top  V^k (\widehat y^\nuk + d^\nuk - \widehat y_\nuk^*) = {o}(\| \widehat y^\nuk + d^\nuk - \widehat y_\nuk^* \|^2)
\end{equation*}
and
\begin{equation*}
	\varphi(\widehat y^\nuk) - \varphi(\widehat y_\nuk^*) - \langle \nabla \varphi(\widehat y_\nuk^*) , \widehat y^\nuk - \widehat y_\nuk^* \rangle - \frac{1}{2} (\widehat y^\nuk - \widehat y_\nuk^*)^\top  W^k (\widehat y^\nuk  - \widehat y_\nuk^*) = {o}(\| \widehat y^\nuk - \widehat y_\nuk^* \|^2),
\end{equation*}
where $V^k \in \partial^2 \varphi(\widehat y^\nuk + d^\nuk)$. 
The definitions of $\partial^2 \varphi$ and $\partial_{\rm HS}\Pi_K$ in \eqref{eq:gen-jac} and \eqref{eq:HS-Jac-s} imply that the matrices $\{V^k\}$ are uniformly bounded. Then, from \eqref{eq:ykdkmyhatk}, we deduce that,
for all $k\ge \max\{N_0,N_1\}$,
\[
\varphi(\widehat y^\nuk + d^\nuk) - \varphi(\widehat y^\nuk)
= -\frac{1}{2}\langle \widehat y^\nuk - \widehat y^*_k, W^k(\widehat y^\nuk - \widehat y^*_k)\rangle + o(\|d^\nuk\|^2).
\]
We further know from \eqref{eq:app_kN0} and \eqref{eq:ykdkmyhatk} that
\begin{equation}
	\label{eq:proof_ls0}
	\begin{aligned}
		&\varphi(\widehat y^\nuk + d^\nuk) - \varphi(\widehat y^\nuk) - \frac{1}{2} \langle \nabla \varphi(\widehat y^\nuk), d^\nuk \rangle\\
		={}&  -\frac{1}{2}\langle \widehat y^\nuk + d^\nuk - \widehat y^*_k, W^k(\widehat y^\nuk - \widehat y^*_k)\rangle + \frac{1}{2}\langle d^\nuk, W^k(\widehat y^\nuk - \widehat y^*_k) - \nabla \varphi(\widehat y^\nuk) \rangle + o(\|d^\nuk\|^2)
		\\
		={}& -\frac{1}{2} \langle \nabla \varphi(\widehat y^\nuk) - \nabla \varphi(\widehat y^*_{\nuk}) - W^k(\widehat y^\nuk - \widehat y^*_k), d^k \rangle + o(\|d^k\|^2)
		\\
		={}& o(\|d^k\|^2),
	\end{aligned}
\end{equation}
where the second equality follows from \eqref{eq:ykdkmyhatk} and the fact that $W^k(\widehat y^k - \widehat y_k^*) = {\cal O}(\|d^k\|)$.
This, together with \eqref{eq:mnablaykdk}, further implies that for all $k$ sufficiently large,
\begin{align*}
\varphi(\widehat y^\nuk + d^\nuk) - \varphi(\widehat y^\nuk) = {}&
\frac{1}{2} \langle \nabla \varphi(\widehat y^\nuk), d^\nuk \rangle + o(\|d^k\|^2) \\
= {}& \gamma_1 \langle \nabla \varphi(\widehat y^\nuk), d^\nuk \rangle  + \left(\frac{1}{2} - \gamma_1 \right) \langle \nabla \varphi(\widehat y^\nuk), d^\nuk \rangle  + o(\|d^\nuk\|^2) \\
\le {}& \gamma_1 \langle \nabla \varphi(\widehat y^\nuk), d^\nuk \rangle  -\frac{c_2^2 }{(1+\zeta)^2c_1}(\frac{1}{2} - \gamma_1) \|d^k\|^2 + o(\|d^k\|^2)\\
\le {}& \gamma_1 \langle \nabla \varphi(\widehat y^\nuk), d^\nuk \rangle.
\end{align*}
Thus, we have verified that the sufficient decrease condition \eqref{eq:Wolfe-1} of the Wolfe conditions is satisfied for $\alpha^k = 1$ for all $k$ sufficiently large.

Next, we verify the curvature condition \eqref{eq:Wolfe-2} of the Wolfe conditions. Notice that the Lipschitz continuity of $\nabla\varphi$ and \eqref{eq:ykdkmyhatk} imply that
\[
    \nabla\varphi(\widehat y^k + d^k) = \nabla\varphi(\widehat y^k + d^k) - \nabla\varphi(\widehat y_k^*) = O(\|\widehat y^k + d^k - \widehat y_k^*\|) = o(\|d^k\|).
\]
Thus, by \eqref{eq:mnablaykdk}, it holds for $k$ sufficiently large,
\begin{equation*}
    \begin{aligned}
        \la\nabla\varphi(\widehat y^\nuk + d^\nuk),d^\nuk\ra - \gamma_2 \la \nabla \varphi(\widehat y^\nuk), d^\nuk \ra & \ge  o(\|d^k\|^2) + \frac{c_2^2}{(1+\zeta)^2c_1} \gamma_2 \|d^k\|^2 \ge 0.
    \end{aligned}
\end{equation*}
That is, \eqref{eq:Wolfe-2} holds with $\alpha^k = 1$ for all $k$ sufficiently large. Hence, we conclude that
\begin{equation}\label{eq:ykp1}
    y^{k+1} = \widehat y^k + d^k \quad \mbox{ for all $k$ sufficiently large.}
\end{equation}
Since $d^k\to 0$, the boundedness of $\{\widehat y^k\}$ established in Proposition~\ref{prop:global-convergence} implies that $\{y^k\}$ is also bounded. Moreover, $\{\widehat y^k\}$ and $\{y^k\}$ share the same set of accumulation points. Hence, by Proposition~\ref{prop:global-convergence}, every accumulation point of $\{y^k\}$ also belongs to ${\widehat \mP}^*$.

Next, we focus on establishing the superlinear convergence of $\{\widehat y^k\}$ and $\{y^k\}$. Note that \eqref{eq:dist-hatyk-P} in Proposition \ref{prop:global-convergence} and Proposition \ref{prop:dist-hatyhatP-yP} imply that there exists a constant $C' >0$ such that for $k$ sufficiently large
\begin{equation}\label{eq:dist-haty-hatp-p}
{\rm dist}(\widehat y^{k+1}, \widehat\mP^*) \le C' {\rm dist}(\widehat y^{\nuk+1},\mP^*).
\end{equation}
Meanwhile, \eqref{eq:converge-func-value} in Proposition \ref{prop:global-convergence} and Lemma \ref{lemma:error-bound} imply that there exist constants $C_0, C_1 >0$ such that for all $k$ sufficiently large
\begin{equation}\label{eq:dist-hatyp-distyhatp}
\begin{aligned}
\dist(\widehat y^{k+1}, \mP^*) \le \sqrt{\frac{1}{C_0}(\varphi(\widehat y^{\nuk+1})-\varphi^*)} \le{}& \sqrt{\frac{1}{C_0}(\varphi( y^{\nuk+1})-\varphi^*)} \\[2pt]
\le{}& \sqrt{\frac{C_1}{C_0}}\dist(y^{k+1}, \mP^*) \le  \sqrt{\frac{C_1}{C_0}}\dist(y^{k+1}, \widehat \mP^*).
\end{aligned}
\end{equation}
Moreover, by \eqref{eq:ykp1}, \eqref{eq:ykdkmyhatk}, and the Lipschitz continuity of $\nabla \varphi$, we have
\[
\dist(y^{k+1}, \widehat{\mP}^*)
\;\le\;
\|\widehat y^k + d^k - \widehat y_k^*\|
\;=\;
{\cal O}\bigl(\|\nabla \varphi(\widehat y^k)\|^{1+\vartheta}\bigr)
\;=\;
{\cal O}\bigl(\dist^{\,1+\vartheta}(\widehat y^k, \mP^*)\bigr)
\;=\;
{\cal O}\bigl(\dist^{\,1+\vartheta}(\widehat y^k, \widehat{\mP}^*)\bigr).
\]
Combining this with \eqref{eq:dist-haty-hatp-p} and \eqref{eq:dist-hatyp-distyhatp} yields
\[
\dist(\widehat y^{k+1}, \widehat{\mP}^*)
=
{\cal O}\bigl(\dist^{\,1+\vartheta}(\widehat y^k, \widehat{\mP}^*)\bigr),
\quad
\dist(\widehat y^{k+1}, \mP^*)
=
{\cal O}\bigl(\dist^{\,1+\vartheta}(\widehat y^k, \mP^*)\bigr),
\]
and
\[
\dist(y^{k+1}, \widehat{\mP}^*) = {\cal O}\bigl( \dist^{1+\vartheta}(\widehat y^k, \widehat{\mP}^*) \bigr)
=
{\cal O}\bigl(\dist^{\,1+\vartheta}(y^k, \widehat{\mP}^*)\bigr).
\]
    Finally, $\dist(y^{\nuk+1}, \mathcal{P}^*) = \mathcal{O}\bigl( \dist^{1+\vartheta}(y^{\nuk}, \mathcal{P}^*) \bigr)$ follows directly from Theorem~\ref{thm:convergence-alg}, as $\alpha^k \equiv 1$ for all sufficiently large $k$, and $\dist(y^k,\mathcal P^*)\to0$ and $\min\{\zeta, \|\nabla\varphi(\widehat y^k)\|^\vartheta\}  \le  \|\nabla\varphi(\widehat y^k)\|^\vartheta\to 0$ as $k \to \infty$ by Proposition~\ref{prop:global-convergence}.
$\hfill \Box$
\end{proof}

\section{Numerical experiments}\label{sec:Numerical}
In this section, we conduct preliminary numerical experiments to examine the practical performance of the proposed semismooth Newton framework on polyhedral projection problems. In particular, we aim to evaluate the effectiveness of the extreme-point correction mechanism developed in this paper. Accordingly, our numerical study focuses on two aspects: the reliability of Algorithm~\ref{alg:global} in solving the tested instances to high accuracy, and its numerical performance compared with the classical uncorrected semismooth Newton (classical SSN) framework \cite{zhao2010newton,LiSunToh2020BirkhoffJacobian}. We emphasize that the purpose of these experiments is not to demonstrate the superiority of Algorithm~\ref{alg:global} over existing classical methods for projection problems or convex quadratic programs. Rather, our goal is to isolate and assess the effect of the proposed correction mechanism within a semismooth Newton framework for addressing singularity issues. Therefore, we do not compare with algorithms that are not based on semismooth Newton methods, such as interior-point methods, splitting methods or gradient-type methods.

We consider three classes of test problems. The first class is derived from the quadratically regularized optimal transport problem \eqref{prob:app-ot}. The second class is a real-data box-constrained projection problem motivated by feasibility restoration in battery scheduling. The third class is derived from convex Markov decision processes and concerns the projection of state-action occupation measures onto a simplex under discounted flow-balance constraints. For these classes, we compare the semismooth Newton method with extreme-point correction against its classical uncorrected counterpart, in which the possible singularity of the generalized Jacobian of $\nabla \varphi$ is handled by adding adaptive small perturbations\footnote{See the implementation details in \cite{LiSunToh2020BirkhoffJacobian}. The corresponding codes are available at \url{https://github.com/MatOpt/Proj_Birkhoff}.}. For each test, both methods use the same randomly generated $y^0$, from which the same associated multiplier $w^0 \in {\cal M}(A^\top y^0 + c)$ is computed.

Both tested algorithms are implemented in MATLAB. The experiments are conducted by running MATLAB (version 25.2, R2025b) on a Mac Studio workstation (Apple M3 Ultra, 28-core CPU, 256 GB of RAM). 
In our numerical experiments, we measure the accuracy of an approximate optimal solution
$y$ for~\eqref{prob:dual-K-lin-proj} by using the following relative KKT residual:
\[
\eta_{\rm KKT}
:=
\frac{\|A\Pi_K(A^\top y+c)-b\|}{1+\|b\|}.\]
Let $\epsilon >0$ be a prescribed tolerance. We terminate the tested algorithms when $\eta_{\rm KKT} \le \epsilon$ or when the maximum number of iterations is reached.

\subsection{Quadratically regularized optimal transport}\label{sec:num-qrot}
We first test our method on quadratically regularized optimal transport problems of the form \eqref{prob:app-ot}. 
For convenience, we recall the formulation:
\[
\min_{X\in\Rbb^{M \times N}}
\left\{
\langle C,X\rangle + \frac{\rho}{2}\|X\|_F^2
\;\middle|\;
X\one_N = \omega,\;\;
X^\top \one_M = \nu,\;\;
X\ge0
\right\}.
\]
Here, $C\in\Rbb^{M\times N}$ is the cost matrix, 
$\omega\in\Rbb^M_+$ and $\nu\in\Rbb^N_+$ are two nonnegative marginal vectors satisfying
$\one_M^\top \omega=\one_N^\top \nu$, and $\rho>0$ is a prescribed regularization parameter. In our experiments, we remove the last scalar equation in \(X^\top \one_M=\nu\), which is redundant, so that the equality-constraint matrix has full row rank.
Below, we report numerical results for quadratically regularized optimal transport problems using both synthetic and real data. 

\subsubsection{Synthetic Gaussian-mixture instances}
We first consider synthetic Gaussian-mixture data and generate the problem instances following the procedure in \cite[Section~4.1]{chu2023efficient}.
In this setting, we assume that $M = N$ and the two marginal vectors $\omega$ and $\nu$ are generated from two independent random vectors with entries sampled uniformly from $(0,1)$, and are then normalized to have unit total mass, i.e., $\one_M^\top \omega = \one_N^\top \nu = 1$. 

The support points are generated from two Gaussian-mixture distributions. 
Specifically, we choose $r$ mixture centers
$
\mu_1,\ldots,\mu_r
$
equally spaced on an interval $[\ell,u]$, where $\ell<u$, and use a common variance parameter $\sigma^2>0$ for all Gaussian components. 
Two independent mixture-weight vectors
$
\alpha\in\Rbb^r_+,
\beta\in\Rbb^r_+
$
are generated from uniform random vectors and normalized to have unit total mass.
This defines two one-dimensional Gaussian-mixture distributions
\[
p_1(t)
=
\sum_{k=1}^r \alpha_k \mathcal{N}(t;\mu_k,\sigma^2),
\qquad
p_2(t)
=
\sum_{k=1}^r \beta_k \mathcal{N}(t;\mu_k,\sigma^2).
\]
The support points
$
\{\xi_i\}_{i=1}^M\subset\Rbb^3,
\{\zeta_j\}_{j=1}^N\subset\Rbb^3
$
are then generated by sampling their coordinates independently from these two mixtures. Specifically, the coordinates of each $\xi_i$ are sampled from $p_1$, while the coordinates of each $\zeta_j$ are sampled from $p_2$. The cost matrix is defined by the squared Euclidean distance
\[
C_{ij}
=
\|\xi_i-\zeta_j\|^2,
\qquad
1\le i\le M,\quad 1\le j\le N,
\]
and is rescaled by its largest entry.
In our experiments, we set $[\ell,u]=[-20,20]$, $r=5$, $\sigma^2=5$.
This construction yields a scalable family of dense instances, with the problem dimension set to $M=N\in\{100,200,500,1000,1200\}$ and the regularization parameter chosen as $\rho\in\{8\times 10^{-2},4\times 10^{-2},1\times 10^{-2},5\times 10^{-3}\}$.
We stop the tested algorithms when $\eta_{\rm KKT} \le 10^{-12}$ or when the iteration limit of $300$ is reached.

\begin{table}[htbp]
\centering
\caption{Comparisons between Algorithm~\ref{alg:global} and the classical SSN for synthetic Gaussian-mixture instances. In the table, ``iter.'' denotes the number of iterations, ``$\eta_{\rm KKT}$'' denotes the final relative KKT residual, and ``time'' denotes the total runtime in seconds.}
\label{tab:gmm-qrot-main}
\scriptsize
\begin{tabular}{cccccccccc}
\toprule
\multicolumn{4}{c}{Instance} & \multicolumn{3}{c}{Algorithm \ref{alg:global}} & \multicolumn{3}{c}{Classical SSN} \\
\cmidrule(lr){1-4}\cmidrule(lr){5-7}\cmidrule(lr){8-10}
$(M,N)$ & $n$ & $m_E$ & $\rho$ & iter. & $\eta_{\rm KKT}$ & time & iter. & $\eta_{\rm KKT}$ & time\\
\midrule
\multirow{4}{*}{$(100,100)$} & \multirow{4}{*}{$10{,}000$} & \multirow{4}{*}{$199$}
& 0.08 & 21 & 3.9e-15 & 0.187 & 54 & 1.4e-13 & 0.133 \\
& & & 0.04 & 25 & 7.0e-15 & 0.088 & 72 & 1.6e-15 & 0.104 \\
& & & 0.01 & 28 & 2.5e-14 & 0.160 & 128 & 7.2e-15 & 0.114 \\
& & & 0.005 & 39 & 5.7e-14 & 0.175 & 140 & 2.1e-14 & 0.087 \\
\midrule
\multirow{4}{*}{$(200,200)$} & \multirow{4}{*}{$40{,}000$} & \multirow{4}{*}{$399$}
& 0.08 & 33 & 4.2e-15 & 0.279 & 81 & 1.8e-15 & 0.183 \\
& & & 0.04 & 39 & 9.4e-15 & 0.351 & 122 & 1.1e-13 & 0.268 \\
& & & 0.01 & 54 & 3.6e-14 & 0.744 & 206 & 1.5e-14 & 0.445 \\
& & & 0.005 & 59 & 7.6e-14 & 0.963 & 241 & 2.7e-14 & 0.493 \\
\midrule
\multirow{4}{*}{$(500,500)$} & \multirow{4}{*}{$250{,}000$} & \multirow{4}{*}{$999$}
& 0.08 & 45 & 1.5e-14 & 2.062 & 121 & 2.1e-15 & 1.017 \\
& & & 0.04 & 56 & 2.8e-14 & 3.493 & 185 & 2.0e-13 & 1.464 \\
& & & 0.01 & 97 & 1.1e-13 & 7.817 & \textbf{300} & \textbf{1.2e-3} & 2.190 \\
& & & 0.005 & 106 & 2.2e-13 & 10.363 & \textbf{300} & \textbf{2.5e-2} & 2.031 \\
\midrule
\multirow{4}{*}{$(1000,1000)$} & \multirow{4}{*}{$1{,}000{,}000$} & \multirow{4}{*}{$1{,}999$}
& 0.08 & 51 & 7.6e-15 & 9.323 & \textbf{300} & \textbf{4.3e-9} & 5.216 \\
& & & 0.04 & 74 & 1.5e-14 & 16.409 & 278 & 4.9e-15 & 5.722 \\
& & & 0.01 & 143 & 5.7e-14 & 39.741 & \textbf{300} & \textbf{8.4e-2} & 6.150 \\
& & & 0.005 & 193 & 1.1e-13 & 62.578 & \textbf{300} & \textbf{1.1e-1} & 5.561 \\
\midrule
\multirow{4}{*}{$(1200,1200)$} & \multirow{4}{*}{$1{,}440{,}000$} & \multirow{4}{*}{$2{,}399$}
& 0.08 & 60 & 1.1e-14 & 19.069 & 213 & 2.7e-15 & 6.292 \\
& & & 0.04 & 80 & 2.1e-14 & 27.932 & \textbf{300} & \textbf{8.8e-3} & 8.533 \\
& & & 0.01 & 138 & 8.8e-14 & 71.851 & \textbf{300} & \textbf{7.0e-2} & 7.971 \\
& & & 0.005 & 204 & 1.6e-13 & 97.269 & \textbf{300} & \textbf{1.9e-1} & 7.693 \\
\bottomrule
\end{tabular}
\end{table}

The detailed comparison results are summarized in Table~\ref{tab:gmm-qrot-main}. For smaller instances, or when the regularization parameter is relatively large, the resulting problems are relatively well conditioned. In these cases, both methods converge reliably, but Algorithm~\ref{alg:global} generally requires fewer iterations than the classical SSN method. As the problem dimension increases or the regularization parameter decreases, the instances become increasingly difficult, and both methods require more iterations. However, the classical uncorrected method often stagnates and fails to achieve high accuracy within the prescribed limit of 300 iterations. In particular, for larger and more difficult instances, it terminates at relatively inaccurate solutions, with KKT residuals remaining on the order of ${\cal O}(10^{-2})$. In contrast, Algorithm~\ref{alg:global} consistently attains high accuracy across all tested instances, demonstrating its robustness in handling degeneracy and computing high-precision solutions. Regarding the per-iteration cost, the extreme-point correction step in Algorithm~\ref{alg:global} involves simplex-pivot operations and is therefore expected to be more expensive than the classical uncorrected method. Nevertheless, as shown in the table, the total runtime of Algorithm~\ref{alg:global} remains competitive for small-dimensional instances and for instances with relatively large regularization parameters.

\begin{figure}[htbp]
\centering
\includegraphics[width=\textwidth]{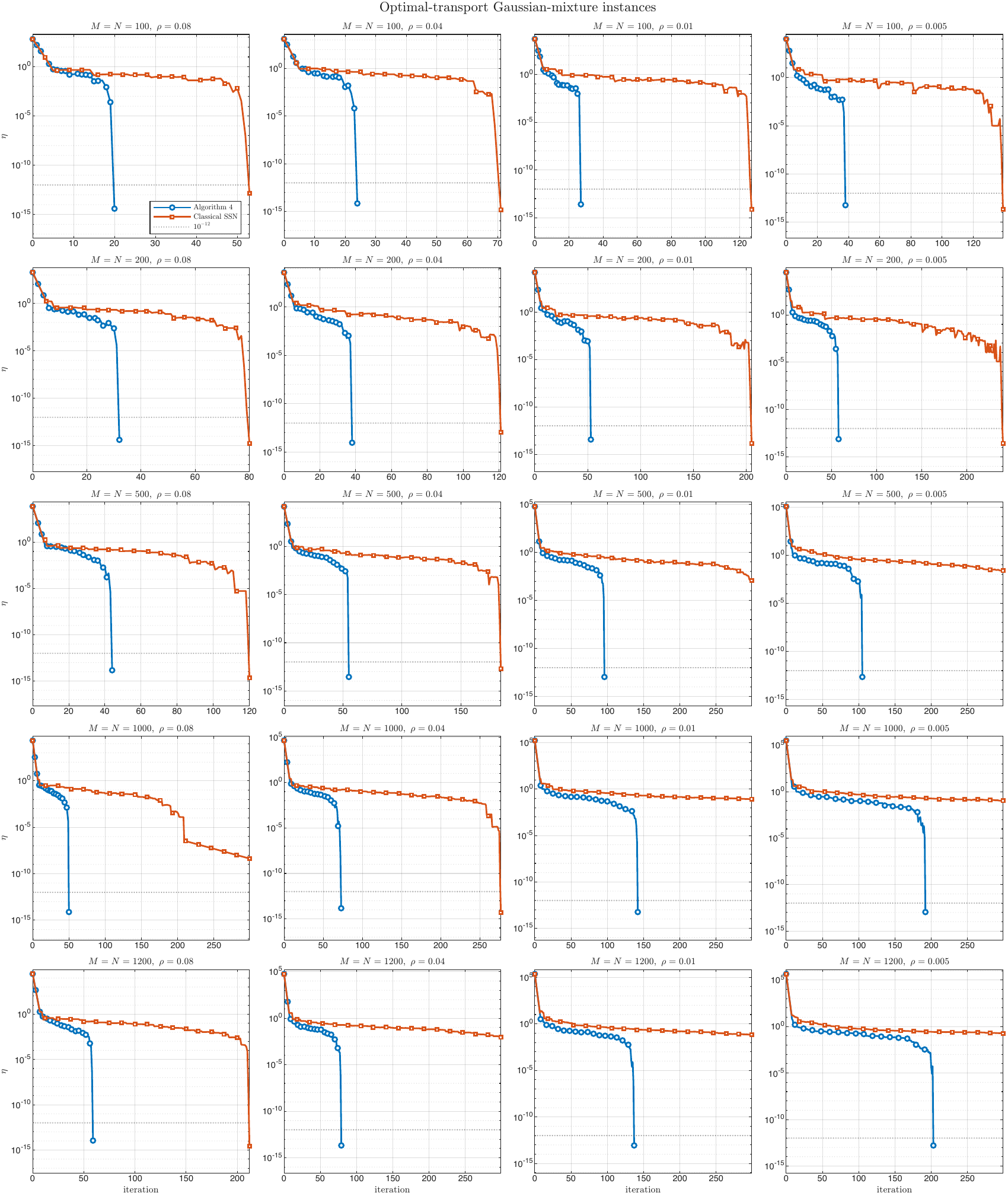}
\caption{Convergence histories of $\eta_{\rm KKT}$ versus the iteration count for synthetic Gaussian-mixture instances. Each panel compares Algorithm~\ref{alg:global} (blue) with the classical SSN (orange).}
\label{fig:optimal_transport_gaussian}
\end{figure}

Figure~\ref{fig:optimal_transport_gaussian} reports the convergence histories of the relative KKT residual \(\eta_{\rm KKT}\) versus the iteration count for all tested Gaussian-mixture instances. The plots show that the classical SSN suffers from prolonged stagnation on difficult instances, whereas Algorithm~\ref{alg:global} makes steady progress and eventually exhibits local superlinear or quadratic convergence.

\subsubsection{DOTmark image instances}
The second family is generated from the DOTmark collection \cite{schrieber2016dotmark}. We select two grayscale images of the same size from a prescribed DOTmark class. Their pixel intensities are first truncated below at zero, then vectorized and normalized to unit total mass. This produces two probability vectors
$
\omega,\nu\in\Rbb_+^{S^2},
$
where $S$ is the image side length. The transport support is the full $S\times S$ pixel grid. Thus each pixel location is represented by its two-dimensional grid coordinate, and the cost matrix is generated from the pairwise Manhattan distance between pixel locations:
\[
C_{ij}=\|p_i-p_j\|_1 = |x_i-x_j| + |y_i-y_j|,
\quad p_i=(x_i,y_i), \, p_j=(x_j,y_j), \quad i,j=1,\ldots,S^2.
\]
As in the synthetic tests, the cost matrix is normalized by its largest entry. In the experiments, we set $S=32$, and select three classes from the DOTmark collection, i.e., the ClassicImages class, the GRFmoderate class, and the GRFrough class. For each pair, we test five different regularization parameters $\rho\in\{8\times 10^{-1},4\times 10^{-1},1\times 10^{-1},5\times 10^{-2},1\times 10^{-2}\}$. The stopping criteria are the same as in the synthetic tests. The detailed results are summarized in Table~\ref{tab:dotmark-qrot-hess} and Figure \ref{fig:dotmark-eta}.

\begin{table}[!htbp]
\centering
\caption{Comparisons between Algorithm~\ref{alg:global} and the classical SSN for DOTmark image instances with $n=(S^2)^2=1{,}048{,}576$ and $m_E=2S^2-1=2{,}047$.}
\label{tab:dotmark-qrot-hess}
\scriptsize
\begin{tabular}{cccccccc}
\toprule
\multicolumn{2}{c}{Instance} & \multicolumn{3}{c}{Algorithm~\ref{alg:global}} & \multicolumn{3}{c}{Classical SSN} \\
\cmidrule(lr){1-2}\cmidrule(lr){3-5}\cmidrule(lr){6-8}
data set & $\rho$ & iter. & $\eta_{\rm KKT}$ & time & iter. & $\eta_{\rm KKT}$ & time \\
\midrule
\multirow{5}{*}{ClassicImages}
& 0.8  & 59 & 8.4e-15 & 18.745 & \textbf{300} & \textbf{9.5e-5} & 3.241 \\
& 0.4  & 65 & 1.7e-14 & 19.745 & \textbf{300} & \textbf{5.6e-4} & 3.274 \\
& 0.1  & 79 & 6.8e-14 & 19.101 & \textbf{300} & \textbf{2.8e-3} & 3.296 \\
& 0.05 & 78 & 1.5e-13 & 19.309 & \textbf{300} & \textbf{1.6e-2} & 4.242 \\
& 0.01 & 85 & 6.8e-13 & 20.384 & \textbf{300} & \textbf{8.0e-2} & 4.463 \\
\midrule
\multirow{5}{*}{GRFmoderate}
& 0.8  & 44 & 1.2e-14 & 18.131 & 199 & 3.7e-13 & 2.334 \\
& 0.4  & 56 & 2.3e-14 & 16.628 & 273 & 6.1e-15 & 3.126 \\
& 0.1  & 55 & 9.3e-14 & 16.536 & \textbf{300} & \textbf{2.6e-3} & 3.317 \\
& 0.05 & 67 & 1.8e-13 & 18.192 & \textbf{300} & \textbf{1.3e-2} & 3.446 \\
& 0.01 & 70 & 8.9e-13 & 21.644 & \textbf{300} & \textbf{1.4e-2} & 3.927 \\
\midrule
\multirow{5}{*}{GRFrough}
& 0.8  & 48 & 3.7e-15 & 18.326 & 236 & 2.4e-15 & 3.070 \\
& 0.4  & 54 & 7.6e-15 & 18.316 & \textbf{300} & \textbf{9.7e-5} & 3.202 \\
& 0.1  & 61 & 2.8e-14 & 17.264 & \textbf{300} & \textbf{1.1e-3} & 2.452 \\
& 0.05 & 66 & 5.8e-14 & 18.452 & \textbf{300} & \textbf{4.6e-3} & 2.720 \\
& 0.01 & 64 & 3.0e-13 & 19.423 & \textbf{300} & \textbf{8.0e-3} & 3.271 \\
\bottomrule
\end{tabular}
\end{table}

It can be observed from Table \ref{tab:dotmark-qrot-hess} that, for the DOTmark image instances, Algorithm~\ref{alg:global} consistently attains high accuracy for all tested regularization parameters. In contrast, the classical SSN method fails to reach high accuracy within the 300-iteration limit for 12 out of the 15 instances, with terminal KKT residuals ranging from approximately $10^{-4}$ to $10^{-1}$.

\begin{figure}[htbp]
\centering
\includegraphics[width=\textwidth]{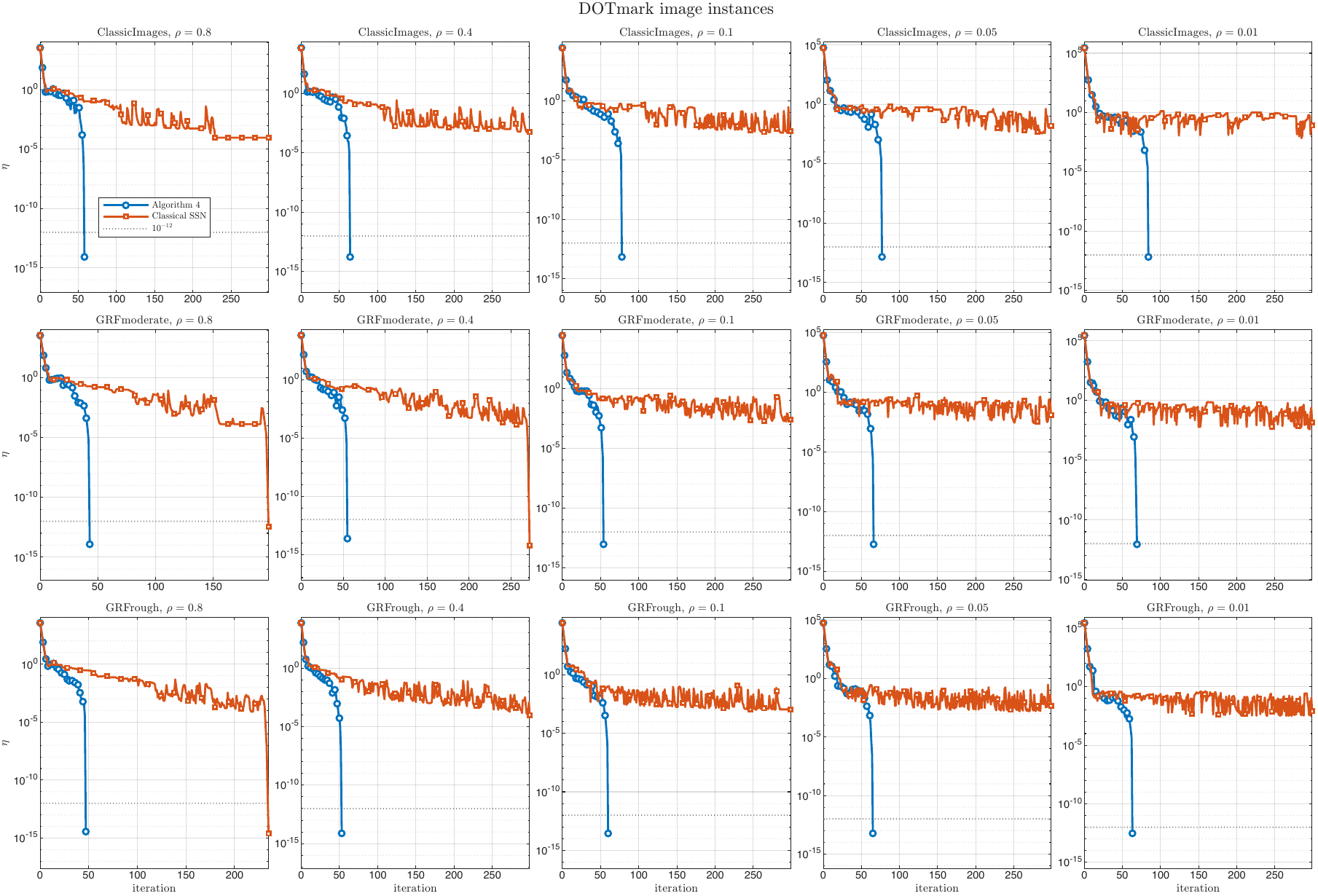}
\caption{Convergence histories of $\eta_{\rm KKT}$ versus the iteration count for DOTmark image instances. Each panel compares Algorithm~\ref{alg:global} (blue) with the classical SSN (orange).}
\label{fig:dotmark-eta}
\end{figure}

Meanwhile, Figure~\ref{fig:dotmark-eta} further highlights the distinct local behavior of the two methods on the DOTmark image instances. The KKT residuals produced by the classical SSN often exhibit pronounced oscillations and stagnate around \({\cal O}(10^{-2})\), with only slow subsequent decrease. This behavior is consistent with the degeneracy of these instances, where the generalized Jacobians selected by the classical SSN can become nearly singular, leading to unstable Newton models and jagged residual histories. By contrast, Algorithm~\ref{alg:global} shows a much more stable local behavior. After steady global progress, it enters a fast local regime in which the residual decreases smoothly and drops by several orders of magnitude within only a few iterations. These observations suggest that the proposed correction strategy effectively stabilizes the local Newton linearization on highly degenerate instances, thereby enabling Algorithm~\ref{alg:global} to compute high-accuracy solutions more reliably.

\subsection{Feasibility restoration in battery scheduling}
\label{sec:num-battery-box}
We next consider a real-data box-constrained projection problem motivated by feasibility restoration in battery scheduling. Energy-storage systems are widely used to absorb surplus renewable generation, discharge during high-demand periods, and mitigate the mismatch between intermittent renewable supply and time-varying demand. In energy management, a physically meaningful storage schedule is typically described by linear state-of-charge dynamics, period-wise power-balance constraints, terminal energy requirements, and lower and upper bounds on charging/discharging power, stored energy, and grid exchange \cite{kraning2014dynamic}. In practice, however, a nominal schedule generated by a simplified model, a forecast-based rule, or a learning-based policy may violate these physical constraints. Projecting such a nominal schedule onto the feasible set therefore provides a natural feasibility-restoration step \cite{mohammadian2025restoring} and leads to a projection problem with linear equality constraints and simple box constraints.

In our experiments, we construct the test instances using the time-series data taken from the Open Power System Data (OPSD) platform
\cite{wiese2019open,opsd2020timeseries}. Let \(T\) be the number of selected time periods.
For each time period \(t=1,\ldots,T\), 
let \(L_t\), \(S_t\), and \(W_t\) denote the original load, solar generation,
and wind generation data. 
Let
$
\bar L=\sum_{t=1}^T L_t/T
$
be the average load over the selected time window. The normalized demand and
the scaled renewable generation are then defined by
\[
d_t=\frac{L_t}{\bar L},
\qquad
r_t=\kappa\frac{S_t + W_t}{\bar L},
\qquad t=1,\ldots,T,
\]
where \(\kappa>0\) is a scaling factor.
Let \(E_{\rm init}\) and \(E_{\rm tar}\) be the initial and terminal target
state of charge of the battery, respectively. The parameters
\(P^{\rm ch}\), \(P^{\rm dis}\), \(E_{\max}\), and \(G_{\max}\) represent the
maximum charging power, maximum discharging power, battery energy capacity, and
maximum grid-import power, respectively. For each time period \(t=1,\ldots,T\),
we use \(p_t\), \(e_t\), \(g_t\), and \(q_t\) to represent the battery power,
state of charge, grid import, and curtailed renewable power. We adopt the sign
convention that \(p_t>0\) corresponds to discharging and \(p_t<0\) corresponds
to charging. Given a possibly
infeasible reference schedule
$(\bar p,\bar e,\bar g,\bar q) \in \mathbb{R}^{T}\times \mathbb{R}^T \times \mathbb{R}^T \times \mathbb{R}^T,
$
we compute its Euclidean projection onto the set of schedules satisfying battery
dynamics, power balance, terminal energy requirement, and physical box bounds:
\begin{equation*}
\begin{array}{ll}
\min & \displaystyle
\frac12\|(p,e,g,q)-(\bar p,\bar e,\bar g,\bar q)\|^2 \\[1mm]
\text{s.t.}
& e_1+ p_1=E_{\rm init},\\
& e_t-e_{t-1}+ p_t=0,\quad t=2,\ldots,T,\\
& p_t+g_t-q_t=d_t-r_t,\quad t=1,\ldots,T,\\
& e_T=E_{\rm tar},\\
& -P^{\rm ch}\le p_t\le P^{\rm dis},\quad t=1,\ldots,T,\\
& 0\le e_t\le E_{\max},\quad 0\le g_t\le G_{\max},\quad
  0\le q_t\le r_t,\quad t=1,\ldots,T.
\end{array}
\end{equation*}
The above problem is clearly an instance of \eqref{eq:gen-box} with \(n=4T\) variables and \(m_E=2T+1\) equality constraints, and the simple box constraint  \(K=[l,u]\).

To construct the reference schedule \((\bar p,\bar e,\bar g,\bar q)\), we use
a noisy net-load-following rule. Specifically, let
\[
{\rm net}_t=d_t-r_t
\]
be the normalized net load. The nominal battery power is generated by
\[
\bar p_t=\theta_{\rm ref} \max\{P^{\rm ch}, P^{\rm dis}\} \tanh({\rm net}_t)+\xi_t^p,
\]
where $\{\xi^p_t\}_{t=1}^T$ are independent small Gaussian perturbations and  \(\theta_{\rm ref}>0\) controls the response intensity of the reference schedule to the net load.
The hyperbolic tangent function  \(\tanh(\cdot)\) is used as a smooth saturation function. It makes the battery tend to discharge when \({\rm net}_t>0\) and charge when \({\rm net}_t<0\), while preventing the reference power from growing unboundedly with the net load.
The remaining components are generated according to 
\[ \bar e_t=E_{\rm init}-\sum_{s=1}^t \bar p_s+\xi_t^e, \qquad \bar g_t={\rm net}_t-\bar p_t+\xi_t^g, \qquad \bar q_t=\max\{-{\rm net}_t,0\}+\xi_t^q . \] The independent Gaussian perturbation terms \(\xi_t^e,\xi_t^g,\xi_t^q\) are added to mimic forecasting and modeling errors.

\begin{table}[htbp]
\centering
\caption{Comparisons between Algorithm \ref{alg:global} and the classical SSN for feasibility restoration instances.}
\label{tab:opsd-battery}
\scriptsize
\begin{tabular}{cccccccccc}
\toprule
\multicolumn{4}{c}{Instance} & \multicolumn{3}{c}{Algorithm~\ref{alg:global}} & \multicolumn{3}{c}{Classical SSN} \\
\cmidrule(lr){1-4}\cmidrule(lr){5-7}\cmidrule(lr){8-10}
mode & $T$ & $n$ & $m_E$ & iter. & $\eta_{\rm KKT}$ & time & iter. & $\eta_{\rm KKT}$ & time \\
\midrule
\multirow{5}{*}{mild}
 & 720  & 2880  & 1441  & 3  & 1.5e-15 & 0.671 & 56      & 5.3e-15 & 0.107 \\
 & 1440 & 5760  & 2881  & 3  & 4.9e-15 & 1.880 & 45      & 1.1e-15 & 0.064 \\
 & 2880 & 11520 & 5761  & 4  & 4.8e-15 & 7.201 & 66      & 5.3e-13 & 0.127 \\
 & 4320 & 17280 & 8641  & 4  & 2.6e-15 & 17.092 & 67      & 2.1e-13 & 0.185 \\
 & 8760 & 35040 & 17521 & 4  & 7.2e-15 & 72.394 & ${\bf 300}$ & {\bf 2.7e-2} & 2.517 \\
\midrule
\multirow{5}{*}{tight}
 & 720  & 2880  & 1441  & 7  & 5.2e-16 & 0.455 & 73      & 7.0e-13 & 0.087 \\
 & 1440 & 5760  & 2881  & 10 & 5.6e-15 & 1.404 & 148     & 2.7e-13 & 0.199 \\
 & 2880 & 11520 & 5761  & 4  & 1.9e-15 & 5.100 & 34      & 5.4e-15 & 0.070 \\
 & 4320 & 17280 & 8641  & 8  & 1.6e-15 & 12.573 & 48      & 1.6e-15 & 0.158 \\
 & 8760 & 35040 & 17521 & 9  & 3.8e-15 & 53.802 & 43      & 3.8e-15 & 0.332 \\
\bottomrule
\end{tabular}
\vspace{0.3em}

\end{table}

In our experiments, we select a contiguous time window
starting from January 1, 2019. If the selected window contains missing entries in
the load, solar generation, or wind generation series, we move the starting time
forward until a clean window of length \(T\) is obtained. Specifically, we test 
\(
T\in\{720, 1440, 2880, 4320, 8760\}.
\) In the tests, we consider two parameter regimes, i.e., the mild and tight regimes. 
The mild regime uses moderate renewable generation and relatively loose limits on battery power, battery capacity, and grid import, resulting in fewer active box constraints. The tight regime, on the other hand, increases renewable generation while reducing battery power and energy capacities, leading to more frequent saturation of operational limits and thus more active box constraints. The detailed parameter values for the two regimes are reported below:
\begin{equation*}
\begin{aligned}
&\mbox{mild regime:} &\kappa=0.8,\qquad
P^{\rm ch}=P^{\rm dis}=0.35,\qquad
E_{\max}=1.20,\qquad
G_{\max}=2.50, 
\\
&\mbox{tight regime:} &\kappa=1.4,\qquad
P^{\rm ch}=P^{\rm dis}=0.18,\qquad
E_{\max}=0.55,\qquad
G_{\max}=2.20. 
\end{aligned}
\end{equation*}
For both regimes, we set the initial  and terminal target state of charge to
$
E_{\rm init}=E_{\rm tar}=E_{\max}/2.
$
In all experiments, we set
\(
\theta_{\rm ref}=2.5
\) so that the nominal schedule may exceed the actual
charging/discharging limits, making the subsequent projection nontrivial.
The perturbation terms are generated independently from centered Gaussian
distributions:
\[
\xi_t^p\sim {\cal N}(0,\sigma_p^2),
\qquad
\xi_t^e\sim {\cal N}(0,\sigma_e^2),
\qquad
\xi_t^g\sim {\cal N}(0,\sigma_g^2),
\qquad
\xi_t^q\sim {\cal N}(0,\sigma_q^2),
\qquad t=1,\ldots,T
\]
with
\(
\sigma_p=0.05,
\sigma_e=\sigma_g=\sigma_q=0.10.
\)
We stop the tested algorithms when $\eta_{\rm KKT} \le 10^{-12}$ or when the iteration limit of $300$ is reached. 
The numerical results are summarized in Table \ref{tab:opsd-battery} and Figure \ref{fig:opsd-battery-eta}.

\begin{figure}[htbp]
\centering
\includegraphics[width=\textwidth]{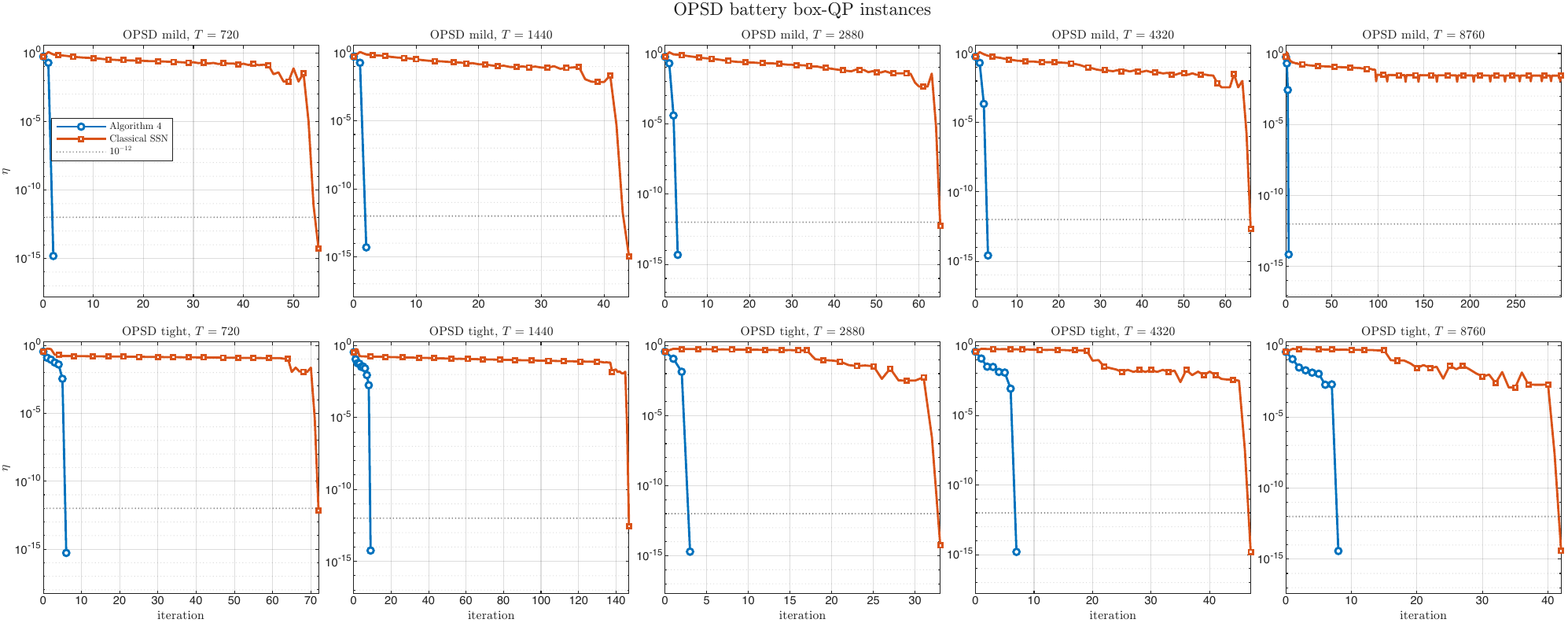}
\caption{Convergence histories of the relative KKT residual $\eta_{\rm KKT}$ versus the iteration count for the feasibility restoration instances. Each panel compares Algorithm~\ref{alg:global} (blue) with the classical SSN (orange).}
\label{fig:opsd-battery-eta}
\end{figure}

As can be observed from Table~\ref{tab:opsd-battery} and Figure~\ref{fig:opsd-battery-eta}, under both the mild and tight settings, Algorithm~\ref{alg:global} requires only 3--10 iterations to compute solutions with terminal KKT residuals of order ${\cal O}(10^{-15})$. This demonstrates clearly its global and fast local convergence. In contrast, the classical SSN generally requires substantially more iterations and fails to reach the target accuracy on the largest instance in the mild regime.

\subsection{Nearest occupation measure in convex MDPs}
We next consider a class of simplex-constrained projection problems arising
from Markov decision processes (MDPs).  Let
\({\cal S}=\{1,\ldots,S\}\) be the state space and
\({\cal A}=\{1,\ldots,M\}\) be the action space.  We write
\[
    x=(x(s,a))_{(s,a)\in{\cal S}\times{\cal A}}
    \in\mathbb R^{SM}
\]
for the normalized discounted occupation measure, i.e., 
\[
x \in \Delta_{SM} := \{x\in\mathbb R^{SM}\mid x\ge 0,\ {\bf 1}_{SM}^\top x=1\}.
\]
The standard discounted
flow-balance equations for MDPs are
\[
    \sum_{a\in{\cal A}}x(s,a)
    -\gamma\sum_{s'\in{\cal S}}\sum_{a'\in{\cal A}}
    P(s\mid s',a')x(s',a')
    =
    (1-\gamma)\nu(s),
    \qquad s\in{\cal S},
\]
where \(\gamma\in(0,1)\) is the discount factor, \(P(\cdot\mid s,a)\)
is the transition probability, and \(\nu\) is the initial distribution. Let \(F\in\mathbb R^{S\times SM}\) denote the full flow-balance matrix.
Its entries are
\[
    F_{s,(s',a')}
    =
    {\bf 1}_{\{s=s'\}}-\gamma P(s\mid s',a').
\]
Since summing all discounted flow-balance equations yields the normalization
condition \({\bf 1}_{SM}^\top x=1\), one flow-balance equation is redundant.  In our formulation,  we drop the last row of $F$, and set
\[
    A= F(1:S-1,:),
    \qquad
    b=(1-\gamma)(\nu(1),\ldots,\nu(S-1))^\top .
\]
With these preparations, we consider the following projection problem arising from apprenticeship learning and convex MDPs \cite{abbeel2004apprenticeship,zahavy2021reward}:
\begin{equation}\label{prob:num-simplex-proj}
\min_{x\in\Rbb^{SM}}
\left\{
\frac{1}{2}\|x - c\|^2
\;\middle|\;
Ax=b,\;\;x\in\Delta_{SM}
\right\},
\end{equation}
where \(c \in \mathbb{R}^{SM}\) is a prescribed target occupancy profile, which may be obtained, for example, from expert demonstrations or empirical occupancy estimates. The quadratic loss measures the deviation of an occupation measure from a prescribed target occupancy profile. Thus, the model seeks an occupation measure, and consequently a policy, that matches the expert's behavior in terms of occupancy statistics~\cite{abbeel2004apprenticeship}.

To further evaluate the performance of our algorithm, we explicitly construct instances of \eqref{prob:num-simplex-proj} where the flow-balance equations and the simplex constraint are linearly dependent at the solution, i.e., the LICQ condition fails. This is achieved by designing the transition probabilities and the reference occupation measure \(\bar{x}\) so that the last state \(S\) is absorbing and receives zero mass under \(\bar{x}\).
Our construction starts by drawing a probability vector
\[
d\in\Rbb^{S-1}_+,\qquad \one_{S-1}^\top d=1.
\]
For every state $s=1,\ldots,S-1$ and every action $a$, the next-state distribution is set to be the same vector $d$ on the first $S-1$ states, while state $S$ receives no mass:
\[
P(i\mid s,a)=d_i,\quad i=1,\ldots,S-1,
\qquad
P(S\mid s,a)=0,
\]
and the last state $S$ is absorbing, i.e.,
\[
P(S\mid S,a)=1,\qquad
P(i\mid S,a)=0,\quad i=1,\ldots,S-1.
\]
Define the reference occupation measure
\[
\bar x(s,a)=\frac{d_s}{M},
\quad s=1,\ldots,S-1,\ a=1,\ldots,M,
\qquad
\bar x(S,a)=0,\quad a=1,\ldots,M.
\]
Then $\bar x\in\Delta_{SM}$. Setting
$
\nu = {F\bar x}/{(1-\gamma)}
$
gives $\nu=(d^\top,0)^\top$ and hence $\bar x$ is feasible for~\eqref{prob:num-simplex-proj}. 
Under this construction, 
it is not difficult to see that $A$ has full row rank because it contains a submatrix $I_{S-1}-\gamma d\one_{S-1}^\top$, whose determinant is $1-\gamma>0$. Moreover, 
one can observe that the LICQ fails at $\bar x$ because the following linear dependence holds:   
\[
\sum_{s=1}^{S-1} F_s
-(1-\gamma)\one_{SM}^\top
+(1-\gamma)\sum_{a=1}^M e_{(S,a)}^\top
=0,
\]
where \(e_{(s,a)}\in\mathbb R^{SM}\) is the
standard basis vector associated with the state--action pair \((s,a)\).
To make $\bar x$ the unique optimal solution to~\eqref{prob:num-simplex-proj}, we first choose a vector $\bar \lambda \in\Rbb^{S-1}$ and a constant $\alpha > 0$, and construct a vector $q$ in the normal cone of $\Delta_{SM}$ at $\bar x$, denoted by $N_{\Delta_{SM}}(\bar x)$, as
\[
q=\alpha\one_{SM}-\mu \in \Rbb^{SM} \, \mbox{ with } \, 
\mu_i=0\ {\rm if}\ \bar x_i>0,\quad
\mu_i\ge0\ {\rm if}\ \bar x_i=0,
\]
and then set
\[
c=\bar x+A^\top \bar \lambda+q.
\]
Then, we have that the following KKT condition
\[
0\in \bar x-c+A^\top \bar \lambda+N_{\Delta_{SM}}(\bar x)
\]
holds, and $\bar x$ is the unique optimal solution.

In our experiments, \(d\in\mathbb{R}^{S-1}\) is obtained by normalizing $S-1$ independent uniform samples, while the entries of 
$\bar \lambda\in\mathbb{R}^{S-1}$ are drawn independently from the standard normal distribution.
To generate \(q\), we set the constant \(\alpha=0.3\), and, for each $i=1,\ldots,SM$, define
\[
\mu_i =
\begin{cases}
0, & \text{if } \bar{x}_i>0,\\
1+\xi_i, & \text{if } \bar{x}_i=0,
\end{cases}
\]
where the \(\xi_i\)'s are independent uniform random variables on \([0,1]\). In our experiments, we sweep over $S, M \in \{20, 30, 500, 1000\}$ and $\gamma \in \{0.8, 0.9, 0.99\}$ to construct $24$ test instances. We stop the tested algorithms when $\eta_{\rm KKT} \le 10^{-10}$ or when the iteration limit of $100$ is reached. We report the results in Table \ref{tab:mdp-licq-selected} and Figure \ref{fig:mdp-licq}.

\begin{table}[htbp]
\centering
\caption{Comparisons between Algorithm~\ref{alg:global} and the classical SSN for LICQ-failure instances.}
\label{tab:mdp-licq-selected}
\scriptsize
\begin{tabular}{cccccccccc}
\toprule
\multicolumn{4}{c}{Instance} & \multicolumn{3}{c}{Algorithm~\ref{alg:global}} & \multicolumn{3}{c}{Classical SSN} \\
\cmidrule(lr){1-4}\cmidrule(lr){5-7}\cmidrule(lr){8-10}
$(S,M)$ & $n$ & $m_E$ & $\gamma$ & iter. & $\eta_{\rm KKT}$ & time & iter. & $\eta_{\rm KKT}$ & time \\
\midrule
\multirow{3}{*}{$(500,20)$} & \multirow{3}{*}{10,000} & \multirow{3}{*}{499}
 & 0.8  & 3 & 3.6e-13 & 46.963 & ${\bf 100}$ & {\bf 4.5e-1} & 4.270 \\
 & & & 0.9  & 3 & 8.7e-13 & 50.549 & ${\bf 100}$ & {\bf 5.5e-1} & 4.123 \\
 & & & 0.99 & 3 & 5.3e-12 & 52.494 & ${\bf 100}$ & {\bf 4.7e-1} & 4.254 \\
\midrule
\multirow{3}{*}{$(1000,20)$} & \multirow{3}{*}{20,000} & \multirow{3}{*}{999}
 & 0.8  & 3 & 1.1e-12 & 515.143 & ${\bf 100}$ & {\bf 5.9e-1} & 10.990 \\
 & & & 0.9  & 3 & 1.8e-12 & 517.864 & ${\bf 100}$ & {\bf 4.4e-1} & 11.154 \\
 & & & 0.99 & 3 & 1.8e-11 & 500.332 & ${\bf 100}$ & {\bf 5.3e-1} & 10.999 \\
\midrule
\multirow{3}{*}{$(500,30)$} & \multirow{3}{*}{15,000} & \multirow{3}{*}{499}
 & 0.8  & 3 & 5.9e-13 & 60.520 & ${\bf 100}$ & {\bf 4.9e-1} & 4.900 \\
 & & & 0.9  & 3 & 9.6e-13 & 66.182 & ${\bf 100}$ & {\bf 4.9e-1} & 4.399 \\
 & & & 0.99 & 3 & 9.0e-12 & 58.695 & ${\bf 100}$ & {\bf 6.3e-1} & 5.838 \\
\midrule
\multirow{3}{*}{$(1000,30)$} & \multirow{3}{*}{30,000} & \multirow{3}{*}{999}
 & 0.8  & 3 & 2.5e-12 & 820.258 & ${\bf 100}$ & {\bf 5.8e-1} & 14.649 \\
 & & & 0.9  & 3 & 2.9e-12 & 833.487 & ${\bf 100}$ & {\bf 7.2e-1} & 14.954 \\
 & & & 0.99 & 3 & 1.0e-11 & 780.591 & ${\bf 100}$ & {\bf 5.4e-1} & 14.685 \\
\midrule
\multirow{3}{*}{$(20,500)$} & \multirow{3}{*}{10,000} & \multirow{3}{*}{19}
 & 0.8  & 3 & 3.6e-13 & 0.062 & ${\bf 100}$ & {\bf 8.9e-1} & 1.350 \\
 & & & 0.9  & 3 & 9.6e-13 & 0.061 & ${\bf 100}$ & {\bf 7.3e-1} & 1.414 \\
 & & & 0.99 & 3 & 8.4e-12 & 0.062 & ${\bf 100}$ & {\bf 6.8e-1} & 1.691 \\
\midrule
\multirow{3}{*}{$(30,500)$} & \multirow{3}{*}{15,000} & \multirow{3}{*}{29}
 & 0.8  & 3 & 1.0e-12 & 0.243 & ${\bf 100}$ & {\bf 9.1e-1} & 2.426 \\
 & & & 0.9  & 3 & 1.3e-12 & 0.241 & ${\bf 100}$ & {\bf 9.5e-1} & 2.340 \\
 & & & 0.99 & 3 & 1.6e-11 & 0.245 & ${\bf 100}$ & {\bf 6.5e-1} & 2.988 \\
\midrule
\multirow{3}{*}{$(20,1000)$} & \multirow{3}{*}{20,000} & \multirow{3}{*}{19}
 & 0.8  & 3 & 5.5e-13 & 0.126 & ${\bf 100}$ & {\bf 8.9e-1} & 3.669 \\
 & & & 0.9  & 3 & 2.3e-12 & 0.123 & ${\bf 100}$ & {\bf 7.3e-1} & 3.155 \\
 & & & 0.99 & 5 & 4.9e-11 & 0.156 & ${\bf 100}$ & {\bf 9.4e-1} & 3.679 \\
\midrule
\multirow{3}{*}{$(30,1000)$} & \multirow{3}{*}{30,000} & \multirow{3}{*}{29}
 & 0.8  & 4 & 8.8e-11 & 0.527 & ${\bf 100}$ & {\bf 9.0e-1} & 5.818 \\
 & & & 0.9  & 3 & 3.3e-12 & 0.484 & ${\bf 100}$ & {\bf 9.2e-1} & 5.618 \\
 & & & 0.99 & 3 & 3.0e-11 & 0.483 & ${\bf 100}$ & {\bf 9.6e-1} & 6.130 \\
\bottomrule
\end{tabular}
\end{table}

As can be observed, for these degenerate problems, Algorithm~\ref{alg:global} solves all the tested instances within 3--5 iterations, whereas the classical SSN suffers from prolonged stagnation. In particular, the classical SSN only returns low-accuracy solutions with terminal KKT residuals between approximately $0.4$ and $1$. These results further confirm the effectiveness of the proposed correction step and the robustness of Algorithm~\ref{alg:global}. 
We also note that the correction mechanism can be substantially more expensive than a classical SSN iteration, especially when $m_E$ is large. For instances with $S=1000$, Algorithm \ref{alg:global} requires approximately 500--830 seconds, whereas the classical SSN takes only about 11--15 seconds, although it terminates at relatively low accuracy. As stated at the beginning of this section, the present experiments primarily demonstrate the reliability of the correction strategy and its ability to attain high-accuracy solutions, rather than uniform superiority in runtime. 

\begin{figure}[htbp]
\centering
\includegraphics[width=\textwidth]{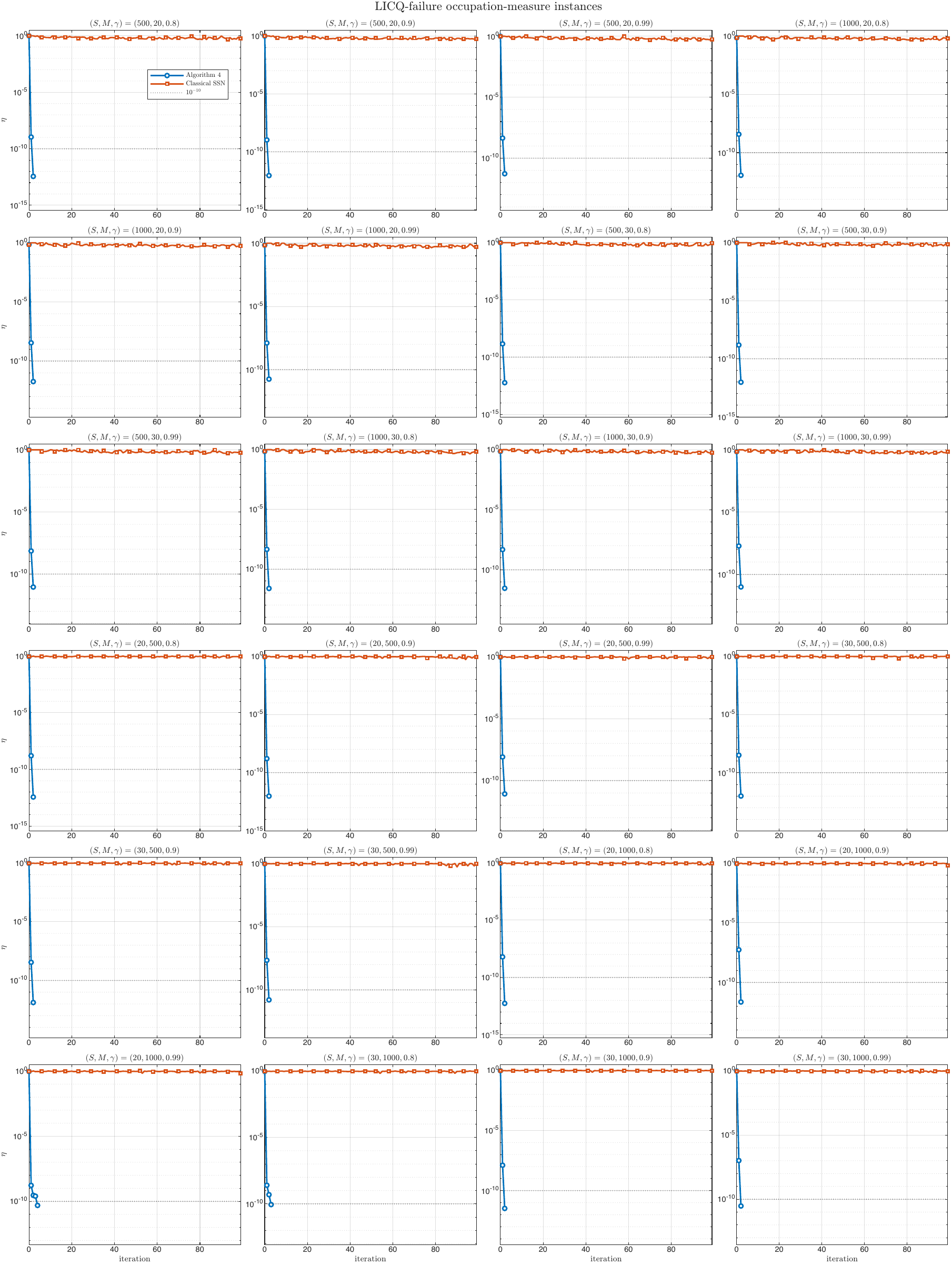}
\caption{Convergence histories of the relative KKT residual $\eta_{\rm KKT}$ versus the iteration count for LICQ-failure instances. Each panel compares Algorithm~\ref{alg:global} (blue) with the classical SSN (orange).}
\label{fig:mdp-licq}
\end{figure}

\section{Conclusion}\label{sec:conclusion}
In this paper, we studied dual semismooth Newton methods for degenerate
polyhedral projection problems. By introducing a primal--dual lifted
projection-equivalent set, we showed that a nonsingular generalized Jacobian of $\nabla \varphi$ can be
constructed at every extreme point of this set. We further established the equivalence among this extreme-point
property, a full-column-rank condition, and a W-SRCQ-type condition, together with
displacement bounds that connect representative selection with the local
Newton analysis. Based on these results, we developed a local inexact semismooth Newton method
and a globalized semismooth Newton method combining monotone
extreme-point identification with a Wolfe line search. Without imposing any a priori nonsingularity or regularity assumption at the optimal
solution, we showed that the local method converges at least superlinearly, and that the globalized method is globally convergent,
eventually accepts full Newton steps, and consequently inherits the same
local superlinear convergence rate. Numerical experiments on
optimal transport, battery-scheduling feasibility restoration, and
occupation-measure projection illustrated the reliability of the proposed
correction mechanism on highly degenerate instances. Developing more efficient structured correction procedures and integrating the proposed mechanism with other fast first- and second-order methods are important directions for future research.


\begin{thebibliography}{10}

\bibitem{abbeel2004apprenticeship}
Pieter Abbeel and Andrew~Y. Ng.
\newblock Apprenticeship learning via inverse reinforcement learning.
\newblock In {\em Proceedings of the Twenty-First International Conference on
  Machine Learning}, pages 1--8, New York, NY, 2004. Association for Computing
  Machinery.

\bibitem{BaiChuSun2007DualIQEP}
Zheng-Jian Bai, Delin Chu, and Defeng Sun.
\newblock A dual optimization approach to inverse quadratic eigenvalue problems
  with partial eigenstructure.
\newblock {\em SIAM Journal on Scientific Computing}, 29(6):2531--2561, 2007.

\bibitem{Bertsekas2016Nonlinear}
Dimitri~P. Bertsekas.
\newblock {\em Nonlinear Programming}.
\newblock Athena Scientific, Belmont, MA, 3rd edition, 2016.

\bibitem{bertsimas1997introduction}
Dimitris Bertsimas and John~N. Tsitsiklis.
\newblock {\em Introduction to Linear Optimization}.
\newblock Number~6 in Athena Scientific Series in Optimization and Neural
  Computation. Athena Scientific, Belmont, MA, 1997.

\bibitem{BioucasDiasEtAl2012}
Jos{\'e}~M. Bioucas-Dias, Antonio Plaza, Nicolas Dobigeon, Mario Parente, Qian
  Du, Paul Gader, and Jocelyn Chanussot.
\newblock Hyperspectral unmixing overview: Geometrical, statistical, and sparse
  regression-based approaches.
\newblock {\em IEEE Journal of Selected Topics in Applied Earth Observations
  and Remote Sensing}, 5(2):354--379, 2012.

\bibitem{BleiNgJordan2003}
David~M. Blei, Andrew~Y. Ng, and Michael~I. Jordan.
\newblock Latent {Dirichlet} allocation.
\newblock {\em Journal of Machine Learning Research}, 3:993--1022, 2003.

\bibitem{BlondelEtAl2018}
Mathieu Blondel, Vivien Seguy, and Antoine Rolet.
\newblock Smooth and sparse optimal transport.
\newblock In {\em Proceedings of the Twenty-First International Conference on
  Artificial Intelligence and Statistics}, volume~84 of {\em Proceedings of
  Machine Learning Research}, pages 880--889. PMLR, 2018.

\bibitem{BonnansShapiro2000}
J.~Fr{\'e}d{\'e}ric Bonnans and Alexander Shapiro.
\newblock {\em Perturbation Analysis of Optimization Problems}.
\newblock Springer, New York, NY, 2000.

\bibitem{chu2023efficient}
Hong T.~M. Chu, Ling Liang, Kim-Chuan Toh, and Lei Yang.
\newblock An efficient implementable inexact entropic proximal point algorithm
  for a class of linear programming problems.
\newblock {\em Computational Optimization and Applications}, 85(1):107--146,
  2023.

\bibitem{CourtyEtAl2017}
Nicolas Courty, R{\'e}mi Flamary, Devis Tuia, and Alain Rakotomamonjy.
\newblock Optimal transport for domain adaptation.
\newblock {\em IEEE Transactions on Pattern Analysis and Machine Intelligence},
  39(9):1853--1865, 2017.

\bibitem{Cuturi2013}
Marco Cuturi.
\newblock {Sinkhorn} distances: Lightspeed computation of optimal transport.
\newblock In {\em Advances in Neural Information Processing Systems},
  volume~26, pages 2292--2300. Curran Associates, Inc., 2013.

\bibitem{Dontchev1998}
Asen~L. Dontchev.
\newblock A proof of the necessity of linear independence condition and strong
  second-order sufficient optimality condition for {Lipschitzian} stability in
  nonlinear programming.
\newblock {\em Journal of Optimization Theory and Applications},
  98(2):467--473, 1998.

\bibitem{facchinei2003finite}
Francisco Facchinei and Jong-Shi Pang.
\newblock {\em Finite-Dimensional Variational Inequalities and Complementarity
  Problems}.
\newblock Springer, New York, NY, 2003.

\bibitem{feng2024quadratically}
Fuxiaoyue Feng, Chao Ding, and Xudong Li.
\newblock A quadratically convergent semismooth {Newton} method for nonlinear
  semidefinite programming without generalized {Jacobian} regularity.
\newblock {\em Mathematical Programming}, 214(1--2):643--683, 2025.

\bibitem{FogelEtAl2013}
Fajwel Fogel, Rodolphe Jenatton, Francis Bach, and Alexandre d'Aspremont.
\newblock Convex relaxations for permutation problems.
\newblock In {\em Advances in Neural Information Processing Systems},
  volume~26, pages 1016--1024. Curran Associates, Inc., 2013.

\bibitem{gabay1983chapter}
Daniel Gabay.
\newblock Applications of the method of multipliers to variational
  inequalities.
\newblock In Michel Fortin and Roland Glowinski, editors, {\em Augmented
  Lagrangian Methods: Applications to the Numerical Solution of Boundary-Value
  Problems}, volume~15 of {\em Studies in Mathematics and Its Applications},
  pages 299--331. North-Holland, Amsterdam, 1983.

\bibitem{glowinski1975approximation}
Roland Glowinski and Americo Marroco.
\newblock Sur l'approximation, par {\'e}l{\'e}ments finis d'ordre un, et la
  r{\'e}solution, par p{\'e}nalisation-dualit{\'e} d'une classe de
  probl{\`e}mes de {Dirichlet} non lin{\'e}aires.
\newblock {\em Revue fran{\c{c}}aise d'automatique, informatique, recherche
  op{\'e}rationnelle. Analyse num{\'e}rique}, 9(R2):41--76, 1975.

\bibitem{goldfarb1983numerically}
Donald Goldfarb and Ashok Idnani.
\newblock A numerically stable dual method for solving strictly convex
  quadratic programs.
\newblock {\em Mathematical Programming}, 27(1):1--33, 1983.

\bibitem{golub2013matrix}
Gene~H. Golub and Charles~F. Van~Loan.
\newblock {\em Matrix Computations}.
\newblock Johns Hopkins University Press, Baltimore, MD, 4th edition, 2013.

\bibitem{HagerZhang2016Projection}
William~W. Hager and Hongchao Zhang.
\newblock Projection onto a polyhedron that exploits sparsity.
\newblock {\em SIAM Journal on Optimization}, 26(3):1773--1798, 2016.

\bibitem{han1997newton}
Jiye Han and Defeng Sun.
\newblock {Newton} and quasi-{Newton} methods for normal maps with polyhedral
  sets.
\newblock {\em Journal of Optimization Theory and Applications},
  94(3):659--676, 1997.

\bibitem{hu2024semismooth}
Hao Hu, Xinxin Li, Haesol Im, and Henry Wolkowicz.
\newblock A semismooth {Newton}-type method for the nearest doubly stochastic
  matrix problem.
\newblock {\em Mathematics of Operations Research}, 49(2):729--751, 2024.

\bibitem{KeshavaMustard2002}
Nirmal Keshava and John~F. Mustard.
\newblock Spectral unmixing.
\newblock {\em IEEE Signal Processing Magazine}, 19(1):44--57, 2002.

\bibitem{kraning2014dynamic}
Matt Kraning, Eric Chu, Javad Lavaei, and Stephen Boyd.
\newblock Dynamic network energy management via proximal message passing.
\newblock {\em Foundations and Trends in Optimization}, 1(2):73--126, 2014.

\bibitem{LellmannEtAl2009}
Jan Lellmann, J{\"o}rg~H. Kappes, Jing Yuan, Florian Becker, and Christoph
  Schn{\"o}rr.
\newblock Convex multi-class image labeling by simplex-constrained total
  variation.
\newblock In {\em Scale Space and Variational Methods in Computer Vision},
  volume 5567 of {\em Lecture Notes in Computer Science}, pages 150--162.
  Springer, 2009.

\bibitem{li1995error}
Wu~Li.
\newblock Error bounds for piecewise convex quadratic programs and
  applications.
\newblock {\em SIAM Journal on Control and Optimization}, 33(5):1510--1529,
  1995.

\bibitem{li2018efficiently}
Xudong Li, Defeng Sun, and Kim-Chuan Toh.
\newblock On efficiently solving the subproblems of a level-set method for
  fused lasso problems.
\newblock {\em SIAM Journal on Optimization}, 28(2):1842--1866, 2018.

\bibitem{LiSunToh2020BirkhoffJacobian}
Xudong Li, Defeng Sun, and Kim-Chuan Toh.
\newblock On the efficient computation of a generalized {Jacobian} of the
  projector over the {Birkhoff} polytope.
\newblock {\em Mathematical Programming}, 179(1--2):419--446, 2020.

\bibitem{LorenzMannsMeyer2021}
Dirk~A. Lorenz, Paul Manns, and Christian Meyer.
\newblock Quadratically regularized optimal transport.
\newblock {\em Applied Mathematics \& Optimization}, 83(3):1919--1949, 2021.

\bibitem{Malick2004DualSDLS}
J{\'e}r{\^o}me Malick.
\newblock A dual approach to semidefinite least-squares problems.
\newblock {\em SIAM Journal on Matrix Analysis and Applications},
  26(1):272--284, 2004.

\bibitem{mohammadian2025restoring}
Mostafa Mohammadian, Anna Van~Boven, and Kyri Baker.
\newblock Restoring feasibility in power grid optimization: A counterfactual
  {ML} approach.
\newblock In {\em 2025 {IEEE} International Conference on Communications,
  Control, and Computing Technologies for Smart Grids ({SmartGridComm})}, pages
  1--6. IEEE, 2025.

\bibitem{nesterov2013introductory}
Yurii Nesterov.
\newblock {\em Introductory Lectures on Convex Optimization: A Basic Course},
  volume~87 of {\em Applied Optimization}.
\newblock Springer, New York, NY, 2004.

\bibitem{nesterov1994interior}
Yurii Nesterov and Arkadii Nemirovskii.
\newblock {\em Interior-Point Polynomial Algorithms in Convex Programming}.
\newblock SIAM, Philadelphia, PA, 1994.

\bibitem{NocedalWright2006}
Jorge Nocedal and Stephen~J. Wright.
\newblock {\em Numerical Optimization}.
\newblock Springer Series in Operations Research and Financial Engineering.
  Springer, New York, NY, 2nd edition, 2006.

\bibitem{nutz2025quadratically}
Marcel Nutz.
\newblock Quadratically regularized optimal transport: Existence and
  multiplicity of potentials.
\newblock {\em SIAM Journal on Mathematical Analysis}, 57(3):2622--2649, 2025.

\bibitem{opsd2020timeseries}
{Open Power System Data}.
\newblock Data package time series.
\newblock \url{https://data.open-power-system-data.org/time_series/}, 2020.
\newblock Version 2020-10-06.

\bibitem{pang1995globally}
Jong-Shi Pang and Liqun Qi.
\newblock A globally convergent {Newton} method for convex {$SC^1$}
  minimization problems.
\newblock {\em Journal of Optimization Theory and Applications},
  85(3):633--648, 1995.

\bibitem{PeyreCuturi2019}
Gabriel Peyr{\'e} and Marco Cuturi.
\newblock Computational optimal transport.
\newblock {\em Foundations and Trends in Machine Learning}, 11(5--6):355--607,
  2019.

\bibitem{Qi1993convergence}
Liqun Qi.
\newblock Convergence analysis of some algorithms for solving nonsmooth
  equations.
\newblock {\em Mathematics of Operations Research}, 18(1):227--244, 1993.

\bibitem{qi1993nonsmooth}
Liqun Qi and Jie Sun.
\newblock A nonsmooth version of {Newton}'s method.
\newblock {\em Mathematical Programming}, 58(1):353--367, 1993.

\bibitem{robinson1981some}
Stephen~M. Robinson.
\newblock Some continuity properties of polyhedral multifunctions.
\newblock In Heinz K{\"o}nig, Bernhard Korte, and Klaus Ritter, editors, {\em
  Mathematical Programming at Oberwolfach}, volume~14 of {\em Mathematical
  Programming Studies}, pages 206--214. Springer, Berlin, 1981.

\bibitem{rockafellar1997convex}
R.~Tyrrell Rockafellar.
\newblock {\em Convex Analysis}, volume~28 of {\em Princeton Mathematical
  Series}.
\newblock Princeton University Press, Princeton, NJ, 1970.

\bibitem{SchootUiterkampHurinkGerards2021}
Martijn H.~H. Schoot~Uiterkamp, Johann~L. Hurink, and Marco E.~T. Gerards.
\newblock A fast algorithm for quadratic resource allocation problems with
  nested constraints.
\newblock {\em Computers \& Operations Research}, 135:105451, 2021.

\bibitem{schrieber2016dotmark}
J{\"o}rn Schrieber, Dominic Schuhmacher, and Carsten Gottschlich.
\newblock {DOTmark}---a benchmark for discrete optimal transport.
\newblock {\em IEEE Access}, 5:271--282, 2017.

\bibitem{ShiouraShakhlevichStrusevich2016}
Akiyoshi Shioura, Natalia~V. Shakhlevich, and Vitaly~A. Strusevich.
\newblock Application of submodular optimization to single machine scheduling
  with controllable processing times subject to release dates and deadlines.
\newblock {\em INFORMS Journal on Computing}, 28(1):148--161, 2016.

\bibitem{Tamir1980}
Arie Tamir.
\newblock Efficient algorithms for a selection problem with nested constraints
  and its application to a production-sales planning model.
\newblock {\em SIAM Journal on Control and Optimization}, 18(3):282--287, 1980.

\bibitem{VidalGribelJaillet2019}
Thibaut Vidal, Daniel Gribel, and Patrick Jaillet.
\newblock Separable convex optimization with nested lower and upper
  constraints.
\newblock {\em INFORMS Journal on Optimization}, 1(1):71--90, 2019.

\bibitem{WangShenZhangYang2023DGASS}
Yunlong Wang, Chungen Shen, Lei-Hong Zhang, and Wei~Hong Yang.
\newblock Proximal gradient/semismooth {Newton} methods for projection onto a
  polyhedron via the {Duality-Gap-Active-Set} strategy.
\newblock {\em Journal of Scientific Computing}, 97(1):3, 2023.

\bibitem{wiese2019open}
Frauke Wiese, Ingmar Schlecht, Wolf-Dieter Bunke, Clemens Gerbaulet, Lion
  Hirth, Martin Jahn, Friedrich Kunz, Casimir Lorenz, Jonathan
  M{\"u}hlenpfordt, Juliane Reimann, and Wolf-Peter Schill.
\newblock {Open Power System Data}---frictionless data for electricity system
  modelling.
\newblock {\em Applied Energy}, 236:401--409, 2019.

\bibitem{wright1997primal}
Stephen~J. Wright.
\newblock {\em Primal--Dual Interior-Point Methods}.
\newblock SIAM, Philadelphia, PA, 1997.

\bibitem{zahavy2021reward}
Tom Zahavy, Brendan O'Donoghue, Guillaume Desjardins, and Satinder~P. Singh.
\newblock Reward is enough for convex {MDP}s.
\newblock In {\em Advances in Neural Information Processing Systems},
  volume~34, pages 25746--25759. Curran Associates, Inc., 2021.

\bibitem{zhao2010newton}
Xin-Yuan Zhao, Defeng Sun, and Kim-Chuan Toh.
\newblock A {Newton--CG} augmented {Lagrangian} method for semidefinite
  programming.
\newblock {\em SIAM Journal on Optimization}, 20(4):1737--1765, 2010.

\end{thebibliography}
\end{document}